\documentclass[12pt]{article}
\usepackage{graphicx}
\usepackage{amsmath,amsthm,amssymb,enumerate}%, esint}
\usepackage{euscript,mathrsfs}
\usepackage{color}
\usepackage{dsfont}
\usepackage{multirow}
\usepackage{multicol}
\usepackage{url}
\usepackage[citecolor=blue,colorlinks=true]{hyperref}
\usepackage[left=2cm,right=2cm,top=3.5cm,bottom=3.5cm]{geometry}
\usepackage{color}
\usepackage[framemethod=tikz]{mdframed}
\allowdisplaybreaks

\usepackage{cancel}
\usepackage{soul}

\catcode`\@=11 \@addtoreset{equation}{section}

\catcode`\@=12

\usepackage{subcaption}

\newtheorem{Theorem}{Theorem}[section]
\newtheorem{Proposition}[Theorem]{Proposition}
\newtheorem{Lemma}[Theorem]{Lemma}
\newtheorem{Corollary}[Theorem]{Corollary}

\theoremstyle{definition}
\newtheorem{Definition}[Theorem]{Definition}

\newtheorem{Remark}[Theorem]{Remark}

\newcommand{\bTheorem}[1]{
	\begin{Theorem} \label{T#1} }
	\newcommand{\eT}{\end{Theorem}}

\newcommand{\bProposition}[1]{
	\begin{Proposition} \label{P#1}}
	\newcommand{\eP}{\end{Proposition}}

\newcommand{\bLemma}[1]{
	\begin{Lemma} \label{L#1} }
	\newcommand{\eL}{\end{Lemma}}

\newcommand{\bCorollary}[1]{
	\begin{Corollary} \label{C#1} }
	\newcommand{\eC}{\end{Corollary}}

\newcommand{\bRemark}[1]{
	\begin{Remark} \label{R#1} }
	\newcommand{\eR}{\end{Remark}}

\newcommand{\bDefinition}[1]{
	\begin{Definition} \label{D#1} }
	\newcommand{\eD}{\end{Definition}}

\newcommand{\Ds}{\mathbb{D}_x}

\newcommand{\jump}[1]{\left[ \left[ #1 \right] \right]}

\newcommand{\vrh}{\vr_h}

\newcommand{\ds}{\,\mathrm{d}S_x}

\newcommand{\bFormula}[1]{
	\begin{equation} \label{#1}}
	\newcommand{\eF}{\end{equation}}

\newcommand{\timezone}{(0,T)}
\newcommand{\timezoneA}{0,T}

\newcommand{\mesh}{\mathcal{T}_h}

\newcommand{\facesK}{\mathcal{E}(K)}
\newcommand{\vuh}{\vu_h}

\newcommand{\TS}{\Delta t}

\newcommand{\Ov}[1]{\overline{#1}}

\newcommand{\Curl}{{\rm curl}_x}

\newcommand{\aleq}{\stackrel{<}{\sim}}

\newcommand{\Un}[1]{\underline{#1}}
\newcommand{\vr}{\varrho}

\newcommand{\tvr}{\wtilde \vr}
\newcommand{\tvu}{{\wtilde u}}
\newcommand{\tvt}{\wtilde \vt}

\newcommand{\vt}{\vartheta}
\newcommand{\vu}{\vc{u}}
\newcommand{\vm}{\vc{m}}

\newcommand{\vn}{\vc{n}}

\newcommand{\vc}[1]{{\bf #1}}

\newcommand{\Div}{{\rm div}_x}
\newcommand{\Grad}{\nabla_x}

\newcommand{\dx}{\,{\rm d} {x}}

\newcommand{\dt}{\,{\rm d} t }

\newcommand{\dxdt}{\dx  \dt}
\newcommand{\intO}[1]{\int_{\Omega} #1 \dx}
\newcommand{\intQ}[1]{\int_{{\Omega}} #1  \dx}

\newcommand{\intOB}[1]{\int_{\Omega} \left( #1 \right)  \dx}

\newcommand{\intTd}[1]{\int_{\mathbb{T}^d} #1  \dx}

\newcommand{\intTO}[1]{\int_0^T \int_{\Omega} #1  \dxdt}
\newcommand{\intTOB}[1]{ \int_0^T \int_{\Omega} \left( #1 \right)  \dxdt}

\newcommand{\vv}{\vc{v}}

\newcommand{\D}{{\rm d}}
\newcommand{\bD}{\mathbb{D}}

\newcommand{\R}{\mathbb{R}}

\newcommand{\I}{\mathbb{I}}
\newcommand{\bS}{\mathbb{S}}

\newcommand{\br}{ \nonumber \\ }

\newcommand{\cred}{\color{red}}

\def\softd{{\leavevmode\setbox1=\hbox{d}%
		\hbox to 1.05\wd1{d\kern-0.4ex{\char039}\hss}}}
\definecolor{Cgrey}{rgb}{0.85,0.85,0.85}
\definecolor{Cblue}{rgb}{0.50,0.85,0.85}
\definecolor{Cred}{rgb}{1,0,0}
\definecolor{fancy}{rgb}{0.10,0.85,0.10}
\definecolor{amaranth}{rgb}{0.9, 0.17, 0.31}

\newcommand{\cblue}{\color{blue}}

\newcommand\Cbox[2]{%
	\newbox\contentbox%
	\newbox\bkgdbox%
	\setbox\contentbox\hbox to \hsize{%
		\vtop{
			\kern\columnsep
			\hbox to \hsize{%
				\kern\columnsep%
				\advance\hsize by -2\columnsep%
				\setlength{\textwidth}{\hsize}%
				\vbox{
					\parskip=\baselineskip
					\parindent=0bp
					#2
				}%
				\kern\columnsep%
			}%
			\kern\columnsep%
		}%
	}%
	\setbox\bkgdbox\vbox{
		\color{#1}
		\hrule width  \wd\contentbox %
		height \ht\contentbox %
		depth  \dp\contentbox
		\color{black}
	}%
	\wd\bkgdbox=0bp%
	\vbox{\hbox to \hsize{\box\bkgdbox\box\contentbox}}%
	\vskip\baselineskip%
}

\mdfdefinestyle{MyFrame}{%
	linecolor=black,
	outerlinewidth=1pt,
	roundcorner=5pt,
	innertopmargin=\baselineskip,
	innerbottommargin=\baselineskip,
	innerrightmargin=10pt,
	innerleftmargin=10pt,
	backgroundcolor=white!20!white}

\newcommand{\wtilde}{\widetilde}

\newcommand{\pd}{\partial}
\newcommand{\Divh}{{\rm div}_h}
\newcommand{\Gradh}{\nabla_h}
\newcommand{\Curlh}{{\rm curl}_h}
\newcommand{\Laph}{\Delta_h}

\newcommand{\Gradd}{\nabla_{\faces}}
\newcommand{\faces}{\mathcal{E}}

\newcommand{\Divmesh}{{\rm div}_{\mathcal{T}}}
\newcommand{\pdedge}{\eth_ \faces}
\newcommand{\pdedgei}{\eth_ \faces^{(i)}}

\newcommand{\pdmesh}{\eth_{\cal T}}
\newcommand{\pdmeshi}{\eth_{\cal T}^{(i)}}

\newcommand{\bI}{\mathbb{I}}
\newcommand{\abs}[1]{{\left| #1 \right|}}
\newcommand{\vth}{\vt_h}
\newcommand{\norm}[1]{\left\lVert#1\right\rVert}
\newcommand{\avs}[1]{ \left\{\hspace{-3pt}\left\{ #1 \right\}\hspace{-3pt} \right\} }
\newcommand{\myangle}[1]{\langle #1\rangle}
\newcommand{\EB}[1]{E_B\left( #1 \right)}

\newcommand{\facesint}{\faces_{\rm int}}
\newcommand{\facesext}{\faces_{\rm ext}}
\newcommand{\intfacesint}[1]{\int_{\facesint}{ #1 \ds}}

\newcommand{\intfacesext}[1]{\int_{\facesext}{ #1 \ds}}

\newcommand{\intT}[1]{\int_{0}^{T} #1 \dt}

\newcommand{\bbT}{ \mathbb{T}}

\newcommand{\Dhuh}{\mathbb{D}_h(\vuh)}
\newcommand{\vthB}{\vt_{B,h}}
\newcommand{\vtB}{\vt_{B}}
\newcommand{\hvt}{\Theta}
\newcommand{\hvth}{\Theta_h}
\newcommand{\difuh}{\bS_h:\Gradh \vuh }
\newcommand{\vvh}{\vc{v}_h}
\newcommand{\vwh}{\vc{w}_h}
\newcommand{\vw}{\vc{w}}

\newcommand{\PiQ}{\Pi_Q}
\newcommand{\PiL}{\Pi^L_h}
\newcommand{\PiW}{\Pi_W}
\newcommand{\PiWi}{\Pi_{W}^{(i)} }
\newcommand{\PiF}{\Pi_{\faces}}
\newcommand{\PiFi}{\Pi_{\faces}^{(i)}}
\newcommand{\Up}{{\rm Up}}
\newcommand{\Fup}{F_h^{\alpha}}
\newcommand{\bfphi}{\boldsymbol{\Phi}}
\newcommand{\muh}{h^\alpha}

\allowdisplaybreaks

\newcommand {\Wh} {W_h}
\newcommand {\Whi} {\Wh^{(i)}}

\newcommand{\sumSi}{ \!\! \sum_{\sigma \in \facesi} \!\!}
\newcommand{\facesi}{\faces _i}
\newcommand{\ve}{\vc{e}}

\newcommand{\sumSK}{ \!\! \sum_{\sigma \in \faces(K)} \!\!}
\newcommand{\sumK}{ \!\! \sum_{K \in \mesh} \!\!}
\newcommand{\vWh}{ {\bf W}_h}
\newcommand{\facei}{ {\faces}_i}

\begin{document}

%%%%%%%%%%%%%%%%%%%%%%%%%%%%%%%%

\title{\bf Convergence of a finite volume method to weak solutions for the compressible Navier--Stokes--Fourier system}

\author{Eduard Feireisl
	\thanks{
		The work of E.F.\ was partially supported by the
		Czech Sciences Foundation (GA\v CR), Grant Agreement
		24--11034S. The Institute of Mathematics of the Academy of Sciences of
		the Czech Republic is supported by RVO:67985840.
		E.F.\ is a member of the Ne\v cas Center for Mathematical Modelling.} 
	\and M\' aria Luk\'a\v{c}ov\'a-Medvi\softd ov\'a\thanks{The work of  M.L.-M. was supported by the Gutenberg Research College and by
		the Deutsche Forschungsgemeinschaft (DFG, German Research Foundation) -- project number 233630050 -- TRR 146 and
		project number 525853336 -- SPP 2410 ``Hyperbolic Balance Laws: Complexity, Scales and Randomness".
		She is also grateful to  the  Mainz Institute of Multiscale Modelling  for supporting her research.}
	\and Bangwei She\thanks{ The work of B.S. was supported by National Natural Science Foundation of China under grant No.\ 12571433.} 
	\and Yuhuan Yuan\thanks{ The work of Y.Y. was supported by National Natural Science Foundation of China under grant No.\ 12401527 and 12571433, and Natural Science Foundation of Jiangsu Province under grant No. BK20241364.}
}

\date{}

\maketitle
{\small
\vspace{-0.75cm}
\centerline{$^*$ Institute of Mathematics of the Academy of Sciences of the Czech Republic}
\centerline{\v Zitn\' a 25, CZ-115 67 Praha 1, Czech Republic}
\centerline{feireisl@math.cas.cz}

\smallskip
\centerline{$^\dag$Institute of Mathematics, Johannes Gutenberg-University Mainz}
\centerline{Staudingerweg 9, 55 128 Mainz, Germany}
\centerline{RMU Co-Affiliate Technical University Darmstadt, Germany}
\centerline{lukacova@uni-mainz.de}

\smallskip
\centerline{$^\ddag$Academy for Multidisciplinary studies, Capital Normal University}
\centerline{ West 3rd Ring North Road 105, 100048 Beijing, P. R. China}
\centerline{bangweishe@cnu.edu.cn}

\smallskip
\centerline{$^\S$School of Mathematics, Nanjing University of Aeronautics and Astronautics}
\centerline{Jiangjun Avenue No. 29, 211106 Nanjing, P. R. China}
\centerline{yuhuanyuan@nuaa.edu.cn}
}

%\smallskip
\vspace{-0.1cm}
\begin{abstract}%\vspace{-0.1cm}
We prove strong convergence of an upwind-type finite volume method to a weak solution of the Navier--Stokes--Fourier system with the Dirichlet boundary conditions.  The limit solution  satisfies a weak form of the mass and momentum equations, together with a weak form of the entropy and ballistic energy inequalities, and complies with the weak-strong uniqueness principle. The finite volume method uses piecewise-constant spatial approximations. The convergence proof is based on a combination of delicate consistency estimates with a careful analysis of the oscillations of numerical densities via renormalisation of the continuity equation.	
\end{abstract}

%\bigskip

{\small
	
	\noindent
	{\bf 2020 Mathematics Subject Classification: }{35Q30,  65M08, 65M12
		(primary); 
		(secondary) 35M13, 35Q35, 74S10 }
	
	\medbreak
	}

\tableofcontents

\section{Introduction}
\label{i}

The motion of a compressible viscous and heat conducting fluid  in the time-space cylinder $(0,T) \times \Omega$ is described by the \emph{Navier--Stokes--Fourier} system: 
\begin{align} 
\partial_t \vr + \Div (\vr \vu) &= 0, \label{i1} \\ 
\partial_t(\vr \vu) + \Div (\vr \vu \otimes \vu) + \Grad p &= 
\Div \mathbb{S} + \vr \vc{g}, \label{i2} \\
\partial_t (\vr e) + \Div (\vr e \vu)	
+ \Div \vc{q} &= \mathbb{S} : \Grad \vu - p \Div \vu.
\label{i3}
\end{align}	
Here $\vr = \vr(t,x)$ is the mass density,  $\vu = \vu(t,x)$ is the (macroscopic) velocity,  $e= e(t,x)$ is the internal energy, $\vc{g}$ is the gravitational force, $\mathbb{S}$ is the \emph{viscous stress} given by Newton's law
\begin{equation} \label{i4}
\mathbb{S} = \mu \left( \Grad \vu + \Grad^T \vu - \frac{2}{d} \Div \vu 
\mathbb{I} \right) + \eta \Div \vu \mathbb{I},\ \mu > 0,\ \eta \geq 0,
\end{equation}
and $\vc{q}$ is the \emph{heat flux} given by Fourier's law 
\begin{equation} \label{i5} 
\vc{q} = - \kappa \Grad \vt,\ \kappa > 0, 
\end{equation}
where $\vt = \vt(t,x)$ is the (absolute) temperature.		
The pressure $p = p(\vr, \vt)$ are related to the temperature through the standard \emph{Boyle--Mariotte equation of state}  
\begin{equation} \label{i7}
p(\vr, \vt) = \vr \vt,\ e = c_v \vt,\ c_v > 0, 
\end{equation} 
with the associated entropy 
\begin{equation} \label{i7a}
s(\vr, \vt) = c_v \log (\vt) - \log(\vr).
\end{equation}

%Motivated by the Rayleigh--B\' enard convection problem, 
\noindent The fluid domain is a (space periodic) slab 
\[
\Omega = {\mathbb{T}}^{d-1} \times [-H,H],\ 
\mathbb{T}^{d-1} = \left( [-L,L]|_{\{ -L, L\} } \right)^{d-1},\ d = 2,3, 
 %\cgrey {\mathcal{T}}^d \times [0,H],\ 
%\mathcal{T}^d = \left( [-1,1]|_{\{ -1, 1\} } \right)^d,\ d = 1,2, 
\]
specifically, the periodicity is imposed in the horizontal direction. 
The velocity field satisfies the no--slip boundary conditions 
\begin{equation} \label{i8}
\vu|_{\partial \Omega} = 0, 
\end{equation}	
while the temperature $\vt$ is prescribed as
\begin{equation} \label{i9}
\vt|_{\partial \Omega} = \vtB, 
\end{equation}
on the horizontal boundary $ x_d = -H, H$.
The system is formally closed by fixing the initial state
\begin{equation}\label{i10}
\vr(0,\cdot)=\vr_0,\quad \vu(0,\cdot)=\vu_0, \quad \vt(0,\cdot)=\vt_0
 \quad \text{in }\Omega, \quad \vr_0>0, \ \vt_0>0. 
\end{equation}

This is a typical configuration arising in real-world applications, such as the Rayleigh--B\'enard convection problem.
The objective of the present paper is 
to establish \emph{strong convergence} of a stable and consistent finite volume method to a \emph{weak solution} of \eqref{i1}-\eqref{i10}. We consider the first-order accurate upwind-type finite volume method proposed in our previous work \cite{FeLMShYu:2024}. There are several reasons why the class of weak solutions may be relevant for 
physically grounded problems: 
\begin{enumerate}
	
\item Singularities, for instance, shock-type discontinuities of the density/pressure, may be  imposed through the initial data. 
It is worth noting that these will not be smoothed by the motion but rather persist at any time, cf.\ Hoff \cite{HOF1,HOF7}, Hoff and Santos \cite{HofSan}. This is relevant, for example, to piston-type problems or to the study of fluid flows in the high-Reynolds number regime (turbulence) \cite{davidson, FeLMShYu:2025II}.

\item The underlying spatial domain may not be smooth enough, but merely Lipschitz. Accordingly, the problem does not admit classical solutions, no matter how smooth or small the data are. Polygonal and generally non--convex domains arise naturally in numerical approximations.

\item The solutions may simply develop ``blow--up'' singularities in a finite time, cf.\ the recent examples by Merle et al.\ \cite{MeRaRoSz,MeRaRoSzbis}, Buckmaster et al.~\cite{BUCLGS}, while the weak solutions exist globally.  
	
\end{enumerate}

Convergence to weak solutions for models of compressible flows is rather delicate because of the lack of suitable stability bounds on the (approximate) densities.  
At the continuous level, this problem can be handled by the approach proposed by Lions \cite{Lions} based on renormalization of the equation of continuity and a careful analysis of density oscillations. This method has been successfully carried over to the numerical framework in a seminal work of Karper 
\cite{Karper} by means of an ingenious but computationally a bit awkward combination of a finite volume and finite element method, see also Eymard  et al. \cite{GL}  for  related results. We note that in \cite{Karper} numerical velocities are approximated by linear finite elements, whereas in our method we use piecewise constant approximations for all variables: density, velocity, and temperature. This makes the application of the discrete version of the Lions approach delicate.
Karper's method has been developed and successfully applied to a larger class of problems, including the Navier--Stokes--Fourier system, cf.~\cite{FeKaNo}. The apparent drawback of \cite{FeKaNo} is the weak formulation for the continuous system that does not admit any relative energy functional; whence nowadays known results on the 
weak--strong uniqueness as well as conditional regularity are not available for the class of weak solutions in \cite{FeKaNo}. 

\subsubsection*{Main novelties and organisation of this paper}
Our aim is to remove the above-mentioned difficulties by proving strong convergence to a weak solution in the framework introduced recently in the monograph \cite{FeiNovOpen}.  This new approach 
allows a straightforward implementation of physically relevant inhomogeneous Dirichlet boundary conditions.
The main novelties of this paper are:
\begin{itemize}
	
\item We examine the first-order accurate upwind-type finite volume scheme applicable to the Navier--Stokes--Fourier system \eqref{i1}--\eqref{i7}, with physically relevant Dirichlet boundary conditions \eqref{i8}, \eqref{i9}, proposed in \cite{FeLMShYu:2024}.  We show that the scheme is stable and consistent under the assumption of uniform boundedness of the numerical approximations.  Careful analysis of the consistency errors allows even low-regularity test functions, as shown in Section~\ref{DEC}.

\item
Since the numerical approximations are spatially piecewise constant, the differential operators have to be discretized appropriately to preserve the discrete integration-by-parts formulae (compatibility conditions, see~Lemma~\ref{lem_CI}) as well as the discrete structures required by the so-called Lions identity, see Lemma~\ref{lem} and Appendix \ref{sec-diffop}. 
In particular, unlike in the mixed finite element-finite volume method proposed in \cite{Karper}, the broken $H^1$-norm of the velocity is no longer bounded, cf. Lemma \ref{lm_ub}, \eqref{ap4}, which results in many new technical difficulties arising in the proof of the strong convergence of the density.

\item As already pointed out, we prove strong  convergence to a weak solution in the framework introduced in the monograph \cite{FeiNovOpen}, see Theorem~\ref{thm_main}. Thus, the limit 
system complies with the weak--strong uniqueness property and enjoys conditional regularity. In particular, we recover immediately unconditional strong convergence as soon as the initial data are smooth enough. 

\end{itemize}

The paper is organized as follows. In Section \ref{n}, we introduce the weak formulation of the NSF system as well as the numerical method.  
In Section \ref{sec:convergence}, we state our main result along with the first steps of its proof. Section \ref{app:con-d} is the heart of the paper. 
We establish the so--called Lions identity and show strong (pointwise) convergence of the densities. Note that in the present purely finite volume context, the proof is rather different from \cite{Karper}.

\section{Weak formulation and the numerical method}
\label{n}

We start by introducing the concept of a weak solution to the 
Navier--Stokes--Fourier system, see \cite[Chapter 12, Definition 7]{FeiNovOpen}.
\begin{Definition}\label{Dw1}
We say that $(\vr,\vu,\vt)$ is a \emph{weak solution} to the Navier--Stokes--Fourier system \eqref{i1}--\eqref{i7}
on $(0,T)\times \Omega$ with the boundary conditions \eqref{i8}, \eqref{i9} and the initial data \eqref{i10} if the following holds:
\begin{itemize}
\item The solution belongs to the following {\bf regularity class}:
\begin{align}
\vr &\in L^\infty(0,T; L^p(\Omega)) \ \mbox{for some}\ p > 1, \br
\vu &\in L^2(0,T; W^{1,2}_0 (\Omega; \R^d)), \quad
(\vt - \vtB) \in L^2(0,T; W^{1,2}_0 (\Omega)).
\label{w6}
\end{align}

\item The {\bf equation of continuity} \eqref{i1} is satisfied in the sense of distributions:
\begin{equation}\label{w3}
\intTOB{ \vr \partial_t \varphi + \vr \vu \cdot \Grad \varphi } + \intO{ \vr_0 \varphi(0,\cdot)} =0
\end{equation}
for any $\varphi \in C^\infty_c([0,T)\times \Ov{\Omega})$.

Moreover, its {\bf renormalized formulation} holds:
\begin{equation}\label{w4}
\intTOB{ b(\vr) \partial_t \varphi + b(\vr) \vu \cdot \Grad \varphi + \big( b(\vr) - b'(\vr) \vr \big)\Div \vu \, \varphi }
+ \intO{ b(\vr_0)\varphi(0,\cdot)} = 0
\end{equation}
for any $\varphi \in C^\infty_c([0,T)\times \Ov{\Omega})$ and any $b \in C^1(\R)$ with $b' \in C_c(\R)$.

\item The {\bf momentum equation} \eqref{i2} is satisfied in the sense of distributions:
\begin{equation}\label{w5}
\begin{aligned}
&\intTOB{ \vr \vu \cdot \partial_t \bfphi + \vr \vu \otimes \vu : \Grad \bfphi + p \Div \bfphi }
+ \intO{ \vr_0 \vu_0 \cdot \bfphi(0,\cdot)}
\\
&\qquad = \intTOB{ \mathbb{S} : \Grad \bfphi - \vr \vc{g} \cdot \bfphi } 
\end{aligned}
\end{equation}
for any $\bfphi \in C^\infty_c([0,T)\times \Omega; \R^d)$.

\item The internal energy equation \eqref{i3} is replaced by the {\bf entropy inequality}:
\begin{equation}\label{w7}
\begin{aligned}
-\intTOB{ \vr s \partial_t \varphi + \vr s \vu \cdot \Grad \varphi + \frac{\vc{q}}{\vt} \cdot \Grad \varphi }
&\geq \intTOB{ \frac{\varphi}{\vt} \left[ \mathbb{S} : \Ds \vu -  \frac{\vc{q} \cdot \Grad \vt }{\vt} \right] }
\\&+ \intO{ \vr_0 s(\vr_0,\vt_0)\varphi(0,\cdot)}
\end{aligned}
\end{equation}
for any $\varphi \in C^\infty_c([0,T)\times \Omega)$, $\varphi \geq 0$.

\item The {\bf ballistic energy inequality} holds in the differential form:
For any $\psi \in C^\infty_c[0,T)$, $\psi \geq 0$, and any $\hvt \in C^1([0,T]\times \Ov{\Omega})$,
$\inf_{[0,T]\times\Omega} \hvt > 0$, $\hvt|_{\partial \Omega} = \vtB$, we have
\begin{equation}\label{w8}
\begin{aligned}
&-\int_0^T \partial_t \psi(t)\intO{ \left[ \frac{1}{2} \vr |\vu|^2 + \vr e - \hvt \vr s \right] } \dt
+ \int_0^T \psi(t)\intO{ \frac{\hvt}{\vt} \left[ \mathbb{S}: \Ds \vu - \frac{\vc{q} \cdot \Grad \vt }{\vt} \right] } \dt
\\
&\leq \psi(0)\intO{ \left[ \frac{1}{2} \vr_0 |\vu_0|^2 + \vr_0 e(\vr_0,\vt_0) - \hvt(0,\cdot) \vr_0 s(\vr_0,\vt_0) \right] }
\\
&\qquad + \int_0^T \psi(t)\intO{ \left[ \vr \vu \cdot \vc{g} - \vr s \,\partial_t \hvt - \vr s \vu \cdot \Grad \hvt - \frac{\vc{q}}{\vt} \cdot \Grad \hvt \right] } \dt .
\end{aligned}
\end{equation}
\end{itemize}
\end{Definition}

We refer the reader to the monograph \cite{FeiNovOpen} for the basic properties of the weak solutions defined above. In particular, 
they exist globally in time for any finite-energy initial data and comply with the weak--strong uniqueness principle. The existence theory 
developed in \cite{FeiNovOpen} is based on more complex but still physically grounded constitutive equations replacing \eqref{i7}.
We are ready to introduce our numerical method.

\subsection{Notation}

Let the domain $\Omega$ be divided into uniform cubes (or squares in 2D) of size $h\in(0,1)$, named as $\mesh$. Let $Q_h$ be the space consisting of piecewise constant functions on the discrete mesh $\mesh$. 
We denote by $\PiQ : L^1(\Omega)\to Q_h$ the cell-average projection,
\begin{equation}\label{proj-avg}
\PiQ v|_K := \frac1{|K|}\int_K v\,\dx \qquad \text{for each cell } K\in\mesh.
\end{equation}

The set of all faces of $\mesh$ is denoted by $\faces$, and  $\facesext = \faces \cap \partial Q$ (resp.\ $\facesint = \faces \setminus \facesext$) stands for the set of all exterior (resp.\ interior) faces.
%{\cgrey Similarly to the above, the edge average operator is defined as 
%\begin{equation} \label{proj-edge}
%	\Pi_W v|_{\sigma} := \frac{1}{|\sigma|} \int_{\sigma} v \ \D s, \ \sigma 
%	\in \mathcal{E}, \ v \in W^{1,1}(\Omega).
%\end{equation}}
Moreover, we denote %by $\faces$ the set of all faces of $\grid$
by $\facei$, $i=1,\dots, d$, the set of all faces that are orthogonal to the canonical basis vector ${\bf e}_i$. 
%Let $\faces$ be the set of all faces of $\mesh$ and $\facei$ be the set of all faces that are orthogonal to the canonical basis vector ${\bf e}_i$. 
Moreover, we define the $i^{\rm th}$ dual grid ${\cal D}_i$ as set of all cubes of the same size $h$ with mass centres sitting at the same position as $\sigma \in  \facei$. Let $W_h^{(i)}$ be the space of piecewise constant functions on ${\cal D}_i$  and $\vWh = \{W_h^{(1)},\cdots,W_h^{(d)}\}$. 
For each $i\in\{1,\dots,d\}$ we also introduce a edge-based average operator $\PiFi$ defined for each dual cell $D\in{\cal D}_i$ by
\begin{equation}\label{projE}
%&\PiWi : L^1(\Omega)\to W_h^{(i)}, \quad (\PiWi v)|_{D} := \frac1{|D|}\int_{D} v\,\dx, \quad v \in L^{1}(\Omega), \\
\PiFi : L^1(\Omega)\to W_h^{(i)}, \quad  (\PiFi v)|_{D} := \frac{1}{|\sigma|} \int_{\sigma} v \ \ds,  \quad v \in W^{1,1}(\Omega).
%\\&\PiW \vv = \left(\PiW^{(1)}v_1, \cdots, \PiW^{(d)}v_d \right),\quad \PiF \vv = \left(\PiF^{(1)}v_1, \cdots, \PiF^{(d)}v_d \right), \quad \vv = (v_1,\dots, v_d).
\end{equation}
%{\cgrey and $\vWh = \{W_h^{(1)},\cdots,W_h^{(d)}\}$. }

For a generic function $v \in Q_h$ we denote on $\sigma \in \faces$ that 
\begin{equation*}
\begin{aligned}
& v^{\rm in} = \lim_{\delta \rightarrow 0^+} v(x -\delta \vn ), \ \
v^{\rm out} = \lim_{\delta \rightarrow 0^+} v(x +\delta \vn ), \  \
\jump{v}= v^{\rm out} - v^{\rm in}, \ \
 \avs{v}= (v^{\rm out} + v^{\rm in} )/2.
\end{aligned}
\end{equation*}
Given a velocity field $\vuh$ we denote the upwind flux at $\sigma \in \faces$  for $r_h \in Q_h$ as 
\begin{align*}
\Up [r_h, \vuh]|_\sigma   =r_h^{\rm up} \avs{\vuh }_\sigma \cdot \vn_\sigma , \quad 
r_h^{\rm up} =
\begin{cases}
r_h^{\rm in} & \mbox{if } \ \avs{\vuh }_\sigma \cdot \vn_\sigma \geq 0, \\
r_h^{\rm out} & \mbox{if } \ \avs{\vuh }_\sigma \cdot \vn_\sigma < 0.
\end{cases}
\end{align*}
Moreover, we introduce the following discrete difference operators for $v\in Q_h, \vv\in Q_h^d$ and $\vw \in \vWh$
\begin{align*}
&\pdedgei v(x) = \sum_{D_\sigma \in {\cal D}_i} \mathds{1}_{D_\sigma} (x) \frac{\jump{r_h}|_{\sigma}}{h} ,\  \Gradd v(x)  =  \left(\pdedge^{(1)} v, \dots, \pdedge^{(d)} v \right);
 %\Gradd v(x)  =  \sumS  \mathds{1}_{D_\sigma}(x) \vn \frac{\jump{v}}{h}  , 
\
\Gradh v(x)  = \sumK  \mathds{1}_K{(x)}\sum_{\sigma\in \facesK} \vn  \otimes \frac{\avs{v}}{h}; 
\\& 
\Divmesh \vw (x)= \sumK \mathds{1}_K{(x)}\sum_{\sigma\in \facesK} \vn \cdot \frac{\vw}{h};
\quad  
\Divh \vv(x) = \sumK  \mathds{1}_K{(x)} \sum_{\sigma\in \facesK} \vn \cdot \frac{\avs{\vv}}{h} = \mbox{tr}(\bD_h \vv);
\\&
\bD_h \vv = (\Gradh \vv + \Gradh^T \vv)/2; 
\quad
 \Laph v (x)= \sumK\mathds{1}_K{(x)}  \sum_{\sigma\in \facesK} \frac{\jump{v}}{h^2} = \Divmesh \Gradd v(x).
\end{align*}
%We suppose $\TS \approx h$ and denote  $t^k= k\TS$ for $k=0,\ldots,N_T(=T/\TS)$. Then we denote a discrete function $v_h$ at time $t^k= k\TS$ by $v_h^k$ and write $v_h \in L_{\TS}(0,T;Q_h)$ if $v_h^k \in Q_h$ for all $k= 0,\dots, N_T$ with  \[ v_h(t,\cdot) =v_h^0 \mbox{ for } t \leq 0,\ v_h(t,\cdot)=v_h^k \mbox{ for } t\in [k\TS,(k+1)\TS),\ k= 0,\ldots,N_T.
%\]
Further, we choose a time step $ \Delta t \approx h, \TS \in (0,1)$ and define the discrete time derivative by the backward Euler method
\[
 D_t v_h(t) = \frac{v_h (t) - v_h^{\triangleleft}(t)}{\TS}  \quad \mbox{ for } t\in (0,T) \quad \mbox{ with } \quad   v_h^{\triangleleft}(t) =  v_h(t - \Delta t) .
\]

\subsection{Numerical scheme}

Suppose the boundary temperature $\vtB$ is a restriction of a positive function $\vtB \in W^{2,\infty}((0,T) \times \Omega)$.
We recall the fully discrete finite volume method introduced in \cite{FeLMShYu:2024}. 
\paragraph{Discrete initial data.}
Given the initial data  $(\vr_0,\vu_0,\vt_0)$ as in \eqref{i10}, we initialize the scheme by the cell-average projection $\PiQ$,
\begin{equation}\label{init-def}
\vrh^0 = \PiQ \vr_0,\qquad
\vuh^0 = \PiQ \vu_0,\qquad
\vth^0 = \PiQ \vt_0.
\end{equation}

%Equivalently, for each cell $K\in\mesh$, we set
%\[
%\vuh^0|_K =
%\begin{cases}
%\dfrac{\PiQ(\vr_0\vu_0)|_K}{\PiQ\vr_0|_K}, & \PiQ\vr_0|_K>0,\\[1ex]
%0, & \PiQ\vr_0|_K=0.
%\end{cases}
%\]
%{\cmag Y: I prefer to delete this part. In the numerical simulations, discrete initial data $\vU_h^0$ are not taken as above. Since our method is first-order, I initialized them more simply with $\vU_0^h = \vU_0 + \mathcal{O}(h)$. \\
%ML: I would keep it as it was, it is not relevant if numerical simulations are slightly different. But we should be specific here! }

\begin{Definition}[{\bf Numerical method}]
Starting with the initial data $(\vrh^0,\vuh^0, \vth^0) \in Q_h^{d+2}$, we say that the sequence
$\{ \vrh^k,\vuh^k, \vth^k \}_{k=1}^{N_T}$ is a finite volume approximation of the Navier--Stokes-Fourier problem \eqref{i1}--\eqref{i9} if it solves the following system of algebraic equations: 
%The FV scheme \eqref{VFV} shall be rewritten in the following strong form 
%\begin{subequations}\label{VFV-1}
%\begin{align}\label{VFV_D-1}
%&D_t \vrh^{k}   + \Divmesh ( \Fup (\vrh^{k} ,\vuh^{k} ) \cdot \vn) = 0,   \\
% \label{VFV_M-1}
%&D_t  (\vrh^{k}  \vuh^{k} )  + \Divmesh ({\bf F}_h^{\alpha}  (\vrh^k  \vuh^k ,\vuh^k ) \cdot \vn )= \Divh (\bS_h^k -p_h^k \bI ) + \vrh^{k} \vc{g} ,   \\\label{VFV_IE-1}
%&c_v D_t (\vrh^k  \vth^k )  + c_v \Divmesh ({\bf F}_h^{\alpha}  (\vrh^k \vth^k  ,\vuh^k ) \cdot \vn) - \kappa \Laph \vth^k
%= (\bS_h^k -p_h^k \bI ) : \Gradh \vuh^k.
%\end{align}
%\end{subequations}
%Here $\Fup (r_h,\vuh)$ is taken as the diffusive upwind numerical flux defined by
%\begin{equation*}%\label{num_flux}
%\Fup (r_h,\vuh)
%={\Up}[r_h, \vuh] - \muh \jump{ r_h }, \quad \alpha >-1,
%%\quad \mbox{ with } \  \ \Up [r_h, \vuh] = r_h^{\rm up} \avs{\vuh} \cdot \vn,
%\end{equation*}
%and 
%$$\bS_h = 2\mu \bD_h \vuh + \lambda \Divh\vuh \bI, \; \Dhuh = (\Gradh \vu_h+\Gradh^T \vu_h)/2, \; \lambda = \eta -  \frac2{d} \mu,$$
%coupled with the following boundary conditions
%\begin{align}
% \Fup (r_h , \vuh )|_{\sigma} = 0, \quad \avs{\vth^k}_{\sigma } = \vthB, \quad \avs{\vuh^k}_{\sigma } = 0 \ \mbox{ and }\ \jump{\bS_h^k -p_h^k \bI}_{\sigma} \cdot \vn = 0, \quad \sigma \in \facesext.
%\end{align}
%\end{Definition}
%\begin{Definition}[{{\bf Weak form}}]
\label{VFV-weak}
%Let $\vU_h^0\equiv(\vrh^0,\vuh^0, \vth^0) \in Q_h^{d+2}$ and $\vthB  \in W_h^{(d)}$. %:=\PiQ (\vr_0, \vu_0, \vt_0). = \Piw \vtB
%We say that the sequence
%$\{\vU_h^k\equiv(\vrh^k,\vuh^k, \vth^k)\}_{k=1}^{N_T}$ is a finite volume approximation of the Navier--Stokes-Fourier problem \eqref{i1}--\eqref{i9} if it solves the following system of algebraic equations: 
%The FV scheme \eqref{VFV-1} shall be rewritten in the following weak form
\begin{subequations}\label{VFV}
\begin{align}\label{VFV_D}
&\intO{ D_t \vrh^k  \phi_h } - \intfacesint{  \Fup (\vrh^{k} ,\vuh^{k} )
\jump{\phi_h}   }  = 0 \quad \mbox{ for all }\ \phi_h \in Q_h, \\
 \label{VFV_M}
&\intO{ D_t  (\vrh^k  \vuh^k ) \cdot \bfphi_h } - \intfacesint{ {\bf F}_h^{\alpha}  (\vrh^{k}  \vuh^{k} ,\vuh^{k} ) \cdot \jump{\bfphi_h}   } + \intO{ (\bS_h^{k} -p_h^{k} \bI ) : \bD_h \bfphi_h }
\br
&\hspace{1. cm}
= \intO{\vrh^{k} \vc{g} \cdot  \bfphi_h}\ 
\mbox{ for all } \bfphi_h \in Q_h^d, \ \mbox{where we set}\ \avs{\bfphi_h}|_{\sigma} = 0 \ \mbox{for any}\ \sigma \in \mathcal{E}_{\rm ext},
\\ \label{VFV_E}
&c_v\intO{ D_t (\vrh^k  \vth^k ) \phi_h } - c_v\intfacesint{  \Fup (\vrh^{k} \vth^{k} ,\vuh^{k} )\jump{\phi_h} }+\intfacesint{  \frac{\kappa}{ h } \jump{\vth^{k}}  \jump{ \phi_h}  }
\br
&\hspace{1.3cm}
+ 2\intfacesext{  \frac{\kappa}{ h } \left( (\vt_h^{k})^{\rm in} - \vtB \right)  \phi_h^{\rm in}  }
= \intO{ (\bS_h^{k} -p_h^{k} \bI ) : \Gradh \vuh^{k} \phi_h} \quad\ \mbox{for all}\ \phi_h \in Q_h,
\end{align}
\end{subequations}
where 
$\Fup (r_h,\vuh)$ is the diffusive upwind flux taken as
\begin{equation*}%\label{num_flux}
\Fup (r_h,\vuh)
={\Up}[r_h, \vuh] - \muh \jump{ r_h }, \quad \alpha >-1,
%\quad \mbox{ with } \  \ \Up [r_h, \vuh] = r_h^{\rm up} \avs{\vuh} \cdot \vn,
\end{equation*}
and $\bS_h = 2\mu \bD_h \vuh + \lambda \mbox{tr} ( \Dhuh ) \bI, \; \Dhuh = (\Gradh \vu_h+\Gradh^T \vu_h)/2, \; \lambda = \eta -  \frac2{d} \mu$, 
together with the  boundary conditions
\begin{equation} \label{VFV_bound} 
\avs{ \vu^k_h }_{\sigma} = 0,\ \ \avs{\vth^k}_{\sigma} =  \vt_{B,h}^k|_\sigma, \ \mbox{ for all }\ \sigma \in \facesext, \quad  \vt_{B,h}^k:= \PiF^{(d)} \vtB(t_k,\cdot).
\end{equation}
\end{Definition}

\medskip
\noindent \textbf{Time interpolation.} 
Given a time step $\Delta t \approx h$ and a discrete sequence $\{\phi_h^k\}_{k=0}^{N_T}$, we define two approximations in time $\phi_h(t)$ and $\tilde{\phi}_h(t)$
by applying piecewise constant and piecewise linear interpolations in time, respectively. Specifically, for $t_k = k\TS$ we set
\begin{align*}
&\phi_h(t, \cdot) = \phi_h^0 \mbox{ for } t \leq t_0, \  \phi_h(t, \cdot) = \phi_h^{N_T} \mbox{ for } t=t_{N_T}, \\
&{\phi}_h(t, \cdot) =\phi_h^{k-1} \ \mbox{ for } \ t = [t_{k-1}, t_{k});  \ \ k=1,\dots, N_T ;
\end{align*}
and 
\begin{align*}
&\tilde{\phi}_h(t, \cdot) =\phi_h^0 \ \mbox{ for } \ t \leq t_0, \\
&\tilde{\phi}_h(t,\cdot)=\phi_h^{k-1} + \frac{\phi_h^{k}- \phi_h^{k-1}}{\TS}(t-t_{k-1}) \ \mbox{ for } \ t\in [t_{k-1},t_k], \ \ k=1,\dots, N_T.
\end{align*}
The latter implies $\pd_t \tilde{\phi}_h(t) = D_t \phi_h^k$ if $ t \in (t_{k-1}, t_k), \; k = 1, \dots, N_T$.

\begin{Remark} \label{rRr}
	
The former condition in \eqref{VFV_bound} is a convention necessary to define the differential operator 
$\Grad \vu_h$ on $\mathcal{E}_{\rm ext}$. Strictly speaking, the latter is not explicitly used in the scheme but needed in the 
analysis for the definition of the discrete operator $\nabla_{\mathcal{E}}$.	

\end{Remark}

\subsection{Uniform bounds - stability}
To begin, we recall the uniform bounds established in our previous paper
\cite[Lemma A.4, Remarks A.5 and A.6]{FeLMShYu:2024}.
\begin{Lemma}[{\bf Stability estimates} ]\label{lm_ub}
Let $\{ \vrh ,\vuh, \vth \}_{h \searrow 0}$ be a family of numerical solutions obtained by the FV method \eqref{VFV}, \eqref{VFV_bound} with $h \approx \TS$, and $\alpha \in (-1,1)$.
%Assume that there exist $\Un{\vr}, \Ov{\vr}, \Ov{u}$ and $\Un{\vt}, \Ov{\vt}, \Un{\vt} \leq \vtB \leq \Ov{\vt}$ such that {\bf (B)} holds, see \eqref{HP}.
Suppose the numerical solutions are uniformly bounded, specifically,  
\begin{equation}\label{HP}
   0< \Un{\vr} \leq \vrh \leq \Ov{\vr}, \
0< \Un{\vt} \leq \vth \leq \Ov{\vt}, \
\abs{\vuh} \leq \Ov{u} \   \mbox{ uniformly for } h \to 0.
\end{equation}

Then we have
\begin{subequations}\label{ap}
\begin{align}\label{ap1}
&\norm{\Gradd \vth}_{L^2(\timezone \times \Omega; \R^{d})} + \norm{ \Gradh \vuh}_{L^2(\timezone \times \Omega; \R^{d\times d})} \leq C ,
\end{align}
\begin{align}\label{ap2}
&(\TS)^{1/2} \norm{D_t (\vrh ,\vuh, \vth)  }_{L^2(\timezone \times \Omega; \R^{d+2})}  \leq C ,
\end{align}
\begin{align} \label{ap3}
\int_{0}^{T}\intfacesint{ \left( h^\alpha + \abs{ \avs{\vuh}\cdot \vc{n} } \right) \, \abs{\jump{ (\vrh ,\vuh, \vth) }}^2  }\dt \leq C, \\ 
\label{ap5} 
\int_{0}^{T} \intfacesext{\abs{\vth^{\rm in} - \vt_{B,h}}^2} \leq Ch, 
\end{align}
which implies
\begin{equation}\label{ap4}
\norm{\Gradd \vuh}_{L^2(\timezone \times \Omega; \R^{d\times d})} \leq C h^{-(1+\alpha)/2}.
\end{equation}
\end{subequations}
The constant $C$ depends on $T$, $\|\vtB\|_{W^{1,\infty}((0,T) \times\Omega)}$ and $\Un{\vr}, \Ov{\vr}, \Un{\vt}, \Ov{\vt}, \Ov{u}$, but it is independent of the discretization parameters $(h, \TS)$.
\end{Lemma}

%{\cmag Note that, \cite[Lemma A.4, Remarks A.5 and A.6]{FeLMShYu:2024} %only states 
%\begin{align*}
%\norm{\Gradd \vth}_{L^2(\timezone \times \Omega; \R^{d})} + \norm{ %\Dhuh}_{L^2(\timezone \times \Omega; \R^{d\times d})} \leq C. 
%\end{align*}
%Together with Lemma \ref{lem_norm} and $\Divh = tr(\Gradh) = %tr(\mathbb{D}_h)$, we obtain \eqref{ap1}.} 

Strictly speaking, the bounds on the velocity gradient obtained in \cite{FeLMShYu:2024} are stated in terms of $\mathbb{S}_h: \nabla_h \vu_h$. The bound \eqref{ap1} follows easily by applying the 
Korn--Poincar\' e formula stated in Lemma \ref{lem_norm}.

Furthermore, for any  $F \in W^{1,\infty}(\R) $ such that  $|F'| \aleq 1$, any $v_h \in Q_h$, and any compact $K \subset \Omega$,  there holds 
\begin{align}
	&\norm{\Gradd F( v_h)}_{L^2(K)} \leq \norm{\Gradd v_h}_{L^2(K)},  \nonumber \\
	& \norm{\nabla_h F(v_h)}_{L^2(K)} \leq  \norm{\nabla_h v_h}_{L^2(K)} \label{Lips}.
\end{align}
This observation yields the following corollary.

\begin{Corollary} \label{Cor1}
For any $(
\vrh, \vuh,  \vth)$ satisfying \eqref{HP}, \eqref{ap3}, \eqref{ap5}, the composed functions  
$G(\vrh)$, $G(\vth)$, $G(\vuh) = (G(u_h^1), \dots, G(u_h^d))$ satisfy 
\begin{align}\label{Cap1}
	&\norm{\Gradd G(\vth)}_{L^2(\timezone \times \Omega; \R^{d})} + \norm{ \Gradh G(\vuh)}_{L^2(\timezone \times \Omega; \R^{d\times d})} \leq C ,
\end{align}
\begin{align} \label{Cap3}
	\int_{0}^{T}\intfacesint{ \left( h^\alpha + \abs{ \avs{\vuh}\cdot \vc{n} } \right) \, \abs{\jump{ (G(\vrh) ,G(\vuh), G(\vth)) }}^2  }\dt \leq C.
\end{align}
for any \emph{locally Lipschitz} $G: \R \to \R$, where we have set 
\[
\avs{ G(\vuh) }|_{\sigma} = \frac{1}{2} \Big( G(\vuh^{\rm in}) + G(- \vuh^{\rm in}) \Big)|_\sigma,\ 
\avs{ G(\vth) }|_{\sigma} = \frac{1}{2} \Big( G(\vth^{\rm in}) + G (2 \vt_{B,h} - \vth^{\rm in}) \Big)|_{\sigma}
\]
for any $\sigma \in \mathcal{E}_{\rm ext}$. 
\end{Corollary}

\subsection{Discretization error - compatibility and consistency} \label{DEC}
Here and hereafter, we suppose the numerical solutions 
satisfy the bound \eqref{HP}. 
Moreover, the time step $\Delta t$ and the grid discretization parameter $h$ are related as
\[
0 < \underline{C_{h, \Delta t}} h \leq \Delta t \leq \Ov{C_{h, \Delta t}} h.
\]
We set  
\begin{equation} \label{lambda}
\Lambda = T + \| \vtB \|_{W^{2, \infty}((0,T) \times \Omega)} + \frac{1}{\underline{\vr}} + \Ov{\vr} + \frac{1}{\underline{\vt}} + \Ov{\vt} + \Ov{u} +  
\frac{1}{\underline{C_{h, \Delta t}} } +\Ov{C_{h, \Delta t}} . 
\end{equation}

Our goal is to estimate the consistency/compatibility errors 
for the numerical approximations under the hypothesis \eqref{HP}. They are related to  both 
the numerical scheme \eqref{VFV_D}--\eqref{VFV_E} and 
the numerical counterparts of the entropy inequality \eqref{w7} and 
the ballistic energy inequality \eqref{w8}. The results are similar but 
not exactly the same as in  \cite[Lemma A.7]{FeLMShYu:2024}. 
In the estimates below, we will systematically use the following notation
\[
U_h \in_{b} X \ \Longleftrightarrow \  U_h \ \mbox{is bounded in}\  X \ \mbox{uniformly for}\  h \to 0.
\]

\subsubsection{Equation of continuity}

We start with the consistency errors related to the equation of 
continuity \eqref{w3}. We set
\begin{align} 
\myangle{ {\it error}_{\vr}; \phi}  &=   \intO{ \vrh^0 \phi (0, \cdot) }  +
\intTOB{ \vrh \partial_t \phi + \vrh \vuh \cdot \Grad \phi }, 
\label{cons-1} \\ 
\myangle{{\it error}_{\widetilde{\vr}}; \phi} &= \intTOB{ \partial_t \widetilde{\vrh}  \phi - \vrh \vuh \cdot \Grad \phi }.   
\label{cons-1a}
\end{align} 

\begin{Lemma}[{\bf Consistency error in the continuity equation}] \label{Lcons1}
For the consistency errors defined in \eqref{cons-1}, \eqref{cons-1a} 
there holds 
\begin{equation} 
\left|	\myangle{ {\it error}_{\vr}; \phi} \right| \leq C(\Lambda) 
\| \phi \|_{C ([0,T];C^{2}(\Ov{\Omega}))  \cap W^{1,2}(0,T; L^2 (\Omega)) }
\left( h^{(1-\alpha)/2} + h^{(1+\alpha)/2} \right)
\label{con-e1}
\end{equation}
for any $\phi \in C^1_c([0,T) ; C^2(\Ov{\Omega}))$,
\begin{equation}
\left| \myangle{{\it error}_{\tvr}; \phi} \right| \leq C(h , \Lambda)\| \phi \|_{L^2(0,T; W^{1,2}(\Omega))}, \ \mbox{where} \  C(h , \Lambda) \to 0\ \mbox{as}\ h \to 0
\label{con-e1a}
\end{equation} 
for any $\phi \in L^2(0,T; W^{1,2}(\Omega))$. 
In particular, 
\begin{equation} \label{derivr}
\partial_t \tvr_h \in_b L^2(0,T; W^{-1,2}(\Omega)), 
\end{equation}
and 
\begin{equation} \label{contr}
\{  \tvr_h \}_{h \searrow 0} \ \mbox{is precompact in}\ 
C_{\rm weak}([0,T] ; L^q(\Omega)) \ \mbox{for any}\ 1 \leq q < \infty. 
\end{equation}	
\end{Lemma}	

 For the proof of \eqref{con-e1} see Appendix \ref{ap_css}, for \eqref{con-e1a} see Appendix \ref{sec-rmk}. 

\subsubsection{Momentum equation}

Similarly to the above, we define the consistency errors 
\begin{align}
\myangle{ {\it error}_{\vr \vu}; \bfphi}  &=   \intO{ \vrh^0 \vuh^0 \cdot \bfphi(0, \cdot) }  + 
\intTOB{ \vrh \vuh \cdot \partial_t \bfphi + \vrh \vuh \otimes \vuh : \Grad \bfphi } 
\br
& - \intTO{ ( \bS_h - p_h \I) : \Grad \bfphi }   - \intTO{\vrh \vc{g} \cdot \bfphi},   \label{cons-2} \\ 
\myangle{{\it error}_{\widetilde{\vr \vu}};\bfphi} &=
\intTOB{ \partial_t \widetilde{\vrh \vuh} \cdot  \bfphi - \vrh \vuh \otimes \vuh : \Grad \bfphi }  
 + \intTO{ ( \bS_h - p_h \I) : \Grad \bfphi }\br  &+ \intTO{\vrh \vc{g} \cdot \bfphi} . 
\label{cons-2a}
\end{align}

\begin{Lemma}[{\bf Consistency error in the momentum equation}] \label{Lcons2}
For the consistency errors defined in \eqref{cons-2}, \eqref{cons-2a} 
there holds 
\begin{equation} 
\left| { \myangle{{\it error}_{\vr \vu}, \bfphi}  } \right|  \leq C (\Lambda) \| \bfphi \|_{C ([0,T];C^{2}(\Ov{\Omega}; \R^d)) \cap W^{1,2}(0,T; L^2(\Omega; \R^d))} \left( h^{(1-\alpha)/2} + h^{(1+\alpha)/2} \right) \label{cone2} \end{equation}
for any $\bfphi \in C^1_c([0,T ); C^2_c({\Omega}; \R^d))$, 
\begin{equation}
\left| \myangle{{\it error}_{\widetilde{\vr \vu}};\bfphi} \right| \leq 
C(h, \Lambda) \| \bfphi \|_{L^2(0,T; W^{1,2}_0 (\Omega; R^d))} ,
\ \mbox{where }  C(h, \Lambda) \to 0 \mbox{ uniformly for }\ h  \to 0
\label{cone2a}
\end{equation}	
for any $\bfphi \in L^2(0,T; W^{1,2}_0(\Omega; \R^d))$. 
In particular, 
\begin{equation} \label{derivm}
\partial_t \widetilde{\vr_h \vu_h} \in_b L^2(0,T; W^{-1,2}(\Omega; \R^d)), 
\end{equation} 	
and 
\begin{equation} \label{contm}
\{ 	\widetilde{\vr_h \vu_h} \}_{h \searrow 0} 
\ \mbox{is precompact in}\ C_{\rm weak}([0,T]; L^q(\Omega; \R^d)) 
\ \mbox{for any}\ 1 \leq q < \infty.
\end{equation}	
\end{Lemma}	

For the proof of \eqref{cone2} see Appendix \ref{ap_css}, for 
\eqref{cone2a} see Appendix \ref{sec-rmk}.

\subsubsection{Internal energy equation and entropy balance}

Similarly to the preceding part, we define the consistency error in the internal energy equation
\begin{align} 
 \myangle{{\it error}_{\widetilde{\vr \vt}}; \phi}  =&\intTOB{  c_v \pd_t \widetilde{\vrh \vth}  \phi  -   (c_v \vrh \vth  \vuh - \kappa  \Gradd \vth)\cdot \Grad \phi }  \br
& -\intTO{ (\bS_h-p_h \I)  : \Gradh \vuh \phi },
\label{cons3}   
\end{align}
together with the associated entropy balance error 
\begin{align}
\myangle{ {\it error}_{\vr s}; \phi}  =&    \intO{ \vrh^0 s(\vr^0_h, \vt^0_h) \phi (0, \cdot) }   +
\intTOB{ \vrh s_h (\pd_t\phi + \vuh \cdot \Grad \phi ) -  \frac{\kappa}{\vth} \Gradd \vth \cdot \Grad \phi }  
\br & 
+ \intTO{ \frac{\phi}{\vth} \left(\kappa \frac{ \chi_h}{ \vth} \abs{\Gradd \vth }^2 + \difuh  \right)} .
\label{cons3-a}
\end{align}
Here, $\chi_h$ is a measurable function satisfying 
\begin{equation} \label{chi}
0 < \underline{\chi} \leq \chi_h \leq \Ov{\chi},\ \chi \to 1 
\ \mbox{in}\ L^1((0,T) \times \Omega) \ \mbox{as}\ h \to 0.
\end{equation}	

\begin{Lemma}[{\bf Consistency error in the internal energy/entropy equation}] \label{Lcons3}
	For the consistency errors defined in \eqref{cons3}, \eqref{cons3-a} 
	there holds 
	\begin{equation} 
\left| \left< {\it error}_{\widetilde{\vr \vt}}; \phi \right> \right| \leq 
C(\Lambda) \| \phi \|_{L^2(0,T; W^{1,2}_0 (\Omega; R^d))}  
\label{cons4}, \end{equation}
for any $\phi \in L^2(0,T; W^{1,2}_0 (\Omega; R^d))$ and $h \to 0$,
\begin{equation} 
\left| \left< {\it error}_{{\vr s}}; \phi \right>  \right| \leq 
C(\Lambda)  \| \phi \|_{C ([0,T];C^{2}(\Ov{\Omega})) \cap W^{1,2}(0,T; L^2 (\Omega)) }\left( h^{(1-\alpha)/2} + h^{(1+\alpha)/2} \right)
\label{cons5}
		\end{equation}
for any $\phi \in C^1_c([0,T]; C^2_c({\Omega}))$, $\phi \geq 0$. In particular, 
\begin{equation} \label{derivt}
\partial_t \widetilde{\vrh \vth} \in_b L^2(0,T; W^{-1,2}(\Omega)) + L^1((0,T) \times \Omega),  
\end{equation}	
and 	
\begin{equation} \label{contt}
\{  \widetilde{\vrh \vth} \}_{h \searrow 0} \ \mbox{is precompact in}\ L^q(0,T; W^{-1,2}(\Omega)) \ \mbox{for any}\ 1 \leq q < \infty.
	\end{equation}
\end{Lemma}
For the proof of \eqref{cons4} see Appendix \ref{sec-rmk}, for 
\eqref{cons5} see Appendix \ref{ap_css}.

\subsubsection{Ballistic energy balance}

Finally, we set 
\begin{align} \label{cons-4}
\myangle{ {\it error}_{BE};  \psi}  &=  \psi (0) \intO{ \left(\frac{1}{2} \vrh^0 |\vuh^0 |^2 + c_v \vrh^0 \vth^0 - \vrh^0 s(\vrh^0, \vth^0) \Theta(0, \cdot) \right)} 
+ \int_0^T \psi \intO{\vrh \vc{g} \cdot \vuh} \dt
\br
&
+ \int_0^T \partial_t \psi \intO{ \left(\frac{1}{2} \vrh |\vuh |^2 + c_v \vrh \vth - \vrh s_h \Theta  \right)    } \dt 
\br
&  
- \int_0^T \psi \intO{ \bigg( \frac{ \kappa\hvt \chi_h}{\vth^2} \; \abs{\Gradd \vth}^2 + \frac{\hvt}{\vth }\difuh \bigg)} \dt \br
&- \int_0^T \psi \intO{\bigg( \vrh s_h \partial_t \hvt + 
	\vrh s_h \vuh \cdot \Grad \hvt - \frac{\kappa}{\vth} \; \Gradd \vth \cdot \Grad \hvt \bigg)} \dt, 
\end{align}
where $\Theta \in W^{2, \infty}((0,T) \times \Omega)$ is an arbitrary (time-dependent) positive extension of the boundary temperature 
\begin{equation} \label{extbt}
	\Theta > 0,\ \Theta|_{\partial \Omega} = \vtB. 
\end{equation}	

\begin{Lemma}[{\bf Consistency error in the ballistic inequality}] \label{Lconsbe}
Given $\Theta$ satisfying \eqref{extbt}, there holds 
\begin{equation} \label{cons-4a} 
\myangle{ {\it error}_{BE};  \psi} \geq - C \Big( \Lambda, \| \Theta \|_{W^{2, \infty}(\Omega)} \Big) \| \psi \|_{W^{1,2}(0,T)}  \left( h^{(1-\alpha)/2} {+  h^{(1+\alpha)/2}}\right)
\end{equation}	
for any $\psi \in W^{1,2}[0,T]$, $\psi \geq 0$, $\psi(T) = 0$.
	\end{Lemma}

For the proof, see Appendix \ref{ap_css}. 

\subsubsection{Compatibility of discrete differential operators}

We complete the preliminary material by stating the  compatibility
of the discrete and continuous differential operators.
 
\begin{Lemma}[{\bf Compatibility}]\label{lem_CI}
Let $(\vrh ,\vuh, \vth)$ be a trio of discrete functions satisfying \eqref{HP}, \eqref{ap1}, and 
\eqref{ap3} with $\alpha \in  \left(-1,1\right)$.

Then there holds 
\begin{equation} \label{cf1}
\left|	\intTOB{ \vu_{h} \cdot \Div \bbT + \Dhuh : \bbT  } \right| \leq C(\Lambda) \| \bbT \|_{L^2(0,T; W^{1,2}(\Omega; \R^{d \times d}_{\rm sym})) } h^{(1-\alpha)/2};
	\end{equation}
\begin{equation} \label{cf3}
\left| \intTOB{  (\vth - \vtB) \Div \Psi+ ( \Gradd\vth - \Grad \vtB) \cdot \Psi  } \right| \leq C(\Lambda)  \| \Psi \|_{L^2(0,T; W^{1,2}(\Omega; \R^d)) } h. 
\end{equation} 
%\begin{equation} \label{cf2}
%\left| 	\intTOB{ |\vuh|^2  \Div \Psi + \Gradh (|\vuh|^2) \cdot \Psi  } \right| \leq C(\Lambda)  \| \Psi \|_{L^2(0,T; W^{1,2}_0(\Omega; \R^d)) } h^{(1-\alpha)/2};
%\end{equation}	
%\begin{equation} \label{cf4}
%\left|	\intTOB{ \vth^2 \Div \Psi+ \Gradd(\vth^2)  \cdot \Psi  } \right| \leq C(\Lambda)  \| \Psi \|_{L^2(0,T; W^{1,2}_0(\Omega; \R^d)) } h.
%\end{equation}	
\end{Lemma}

For the proof we refer to Appendix \ref{ap_comp}. As a direct consequence of Corollary \ref{Cor1} we obtain the following result.

\begin{Corollary} \label{cor1}
Let $(\vrh ,\vuh, \vth)$ be a numerical solution obtained by the finite volume method \eqref{VFV} with $h \approx \TS$ and $\alpha \in  \left(-1,1\right)$.

Then there holds
\begin{equation} \label{cf2}
\left| 	\intTOB{ |\vuh|^2  \Div \Psi + \Gradh (|\vuh|^2) \cdot \Psi  } \right| \leq C(\Lambda)  \| \Psi \|_{L^2(0,T; W^{1,2}(\Omega; \R^d)) } h^{(1-\alpha)/2};
\end{equation}	
\begin{equation} \label{cf4}
\left|	\intTOB{ \Big(\vth^2-\vtB^2 \Big) \Div \Psi+ \Big( \Gradd(\vth^2) - \Grad (\vtB^2)\Big) \cdot \Psi  } \right| \leq C(\Lambda)  \| \Psi \|_{L^2(0,T; W^{1,2}(\Omega; \R^d)) } h.
\end{equation}	
\end{Corollary}

\begin{Remark} \label{RrR} We point out that the test functions in \eqref{cf1}, \eqref{cf3} need not be compactly supported in 
	$\Omega$. This fact plays a crucial role in satisfying the boundary conditions in the convergence proof, cf. Section 
\ref{wco}. 
\end{Remark}

\section{Main result - convergence of the numerical scheme}\label{sec:convergence}

In this section, we state our main result. The numerical solutions converge (up to a subsequence)
to a weak solution of the Navier--Stokes--Fourier system on $(0,T) \times \Omega$ as soon as the discretization parameters $h \approx \Delta t \to 0$.
Note carefully that the convergence is only conditional (up to a suitable subsequence) as the weak solutions of the Navier--Stokes--Fourier system 
are not (known to be) uniquely determined by the initial/boundary data. 

\begin{Theorem}[{\bf Convergence}]\label{thm_main}
Let the initial data $(\vr_0, \vu_0, \vt_0)$ be measurable functions in $\Omega$ satisfying  
\begin{equation} \label{idata}
0 < \underline{\vr} \leq \vr_0 \leq \Ov{\vr},\ 0 < \underline{\vt} \leq \vt_0 \leq \Ov{\vt},\ |\vu_0| \leq \Ov{u} \ \mbox{a.a. in}\ \Omega. 
\end{equation}
Let the boundary temperature $\vtB$ be a restriction of a positive function $\vtB \in W^{2, \infty}((0,T) \times \R^d)$. 
We set discrete initial data as 
\begin{equation}\label{init-pro}
		\vrh^0 = \Pi_Q \vr_0,\ \vt^0_h = \Pi_Q \vt_0,\ \vu^0_h = \Pi_Q \vu_0.
\end{equation}
Let $\{\vrh ,\vuh, \vth\}_{h \searrow 0}$ be the associated sequence of numerical solutions obtained by the FV method \eqref{VFV}
on the time interval $(0,T)$ with $\alpha \in  (-1,1)$ and $h \approx \TS$. In addition, suppose 
\begin{align}
	0 < \liminf_{h > 0} \left( \inf_{t \in [0,T] \times \Omega} \vr_h  \right) &\leq \sup_{h > 0} \left( \sup_{t \in [0,T] \times \Omega} \vr_h \right) < \infty, \br 
	0 < \liminf_{h > 0} \left( \inf_{t \in [0,T] \times \Omega} \vt_h  \right) &\leq \sup_{h > 0} \left( \sup_{t \in [0,T] \times \Omega} \vt_h \right) < \infty, \br
	\sup_{h > 0} \left( \sup_{[0,T] \times \Omega} |\vu_h| \right) &< \infty.
\label{HP1bis}
\end{align}

Then there exists a subsequence (not relabeled) such that
\begin{align}\label{strong-conv}
\vrh \to \vr,\ \vth \to \vt,\ \vuh \to \vu \ \ \mbox{ a.a. in }\ (0,T) \times \Omega \ \mbox{ for } h \to 0,
\end{align}
where $(\vr,\vu,\vt)$ is a weak solution of the Navier--Stokes--Fourier system with the initial data $(\vr_0, \vu_0, \vt_0)$ and the boundary temperature $\vtB$ in the sense of Definition~\ref{Dw1}.
\end{Theorem}

\begin{Remark} \label{RR1}
%Note that \eqref{idata} and \eqref{init-pro} implies
%\begin{equation}\label{init-conv}
%\vrh^0 \to \vr_0 \ \mbox{in}\ L^1(\Omega),\
%\vth^0 \to \vt_0 \ \mbox{in}\ L^1(\Omega),\
%\vrh^0 \vuh^0 \to \vr_0 \vu_0 \ \mbox{in}\ L^1(\Omega;\R^d).
%\end{equation}

Modifying the values of the constants in \eqref{idata} as the case may be, we may write the hypothesis 
\eqref{HP1bis} in the same form, namely  
\begin{equation} \label{HP1}
0 < \underline{\vr} \leq \vr_h \leq \Ov{\vr},\ 0 < \underline{\vt} \leq  \vt_h  \leq \Ov{\vt},\ |\vu_h| \leq \Ov{u} \ \mbox{in}\ [0,T] \times \Omega
\end{equation}	
uniformly for $h \to 0$. As a matter of fact, it is \eqref{HP1} rather than \eqref{HP1bis} that is systematically used in the text, notably in the definition 
of the parameter $\Lambda$ in \eqref{lambda}.
	\end{Remark}

The proof of Theorem \ref{thm_main} is rather involved and will be carried out in several steps. The most difficult part is to establish the strong convergence of the approximate densities - a result of independent interest obtained in Section \ref{app:con-d}.

\subsection{Weak convergence}
\label{wco}

It follows from the stability estimates in Lemma~\ref{lm_ub}, combined with the compatibility formulae \eqref{cf1}, \eqref{cf3},   
that
\begin{align}\label{weak-con-1}
&(\vr_{h} ,\vu_{h}, \vt_{h})  \to (\vr ,\vu, \vt) \ \mbox{ weakly-(*) in } L^\infty((0,T) \times \Omega; \R^{d+2}) ,
\\
\label{weak-con-3}
&\Gradh \vu_{h}  \to \Grad \vu \ \mbox{ weakly in } L^2( (0,T) \times \Omega; \R^{d\times d}), 
\\\label{weak-con-4}
&\Gradd \vt_{h}  \to \Grad \vt \ \mbox{ weakly in } L^2( (0,T) \times \Omega; \R^{d}), 
\end{align}
at least for a suitable subsequence. Moreover, it follows from \eqref{HP1} that
\begin{align} 
0 < \underline{\vr} \leq \vr \leq \Ov{\vr} \ &\mbox{ a.a. in }\ (0,T) \times \Omega	, \br
\vu \in L^2(0,T; W^{1,2}(\Omega; \R^d)),\ | \vu | \leq \Ov{u} \ &\mbox{ a.a. in }\ (0,T) \times \Omega, \br 
\vt \in L^2(0,T; W^{1,2}(\Omega)),\ 0 < \underline{\vt} \leq \vt \leq \Ov{\vt} \ &\mbox{ a.a. in }\ (0,T) \times \Omega.
	\label{regul}
\end{align}	 

In order to see that the limits satisfy the relevant boundary conditions 
\eqref{i8} and \eqref{i9}, we use the compatibility of the discrete differential operators stated in 
\eqref{cf1} and \eqref{cf2}, respectively. 

As \eqref{cf1} is satisfied for \emph{any} $\mathbb{T}$, say continuously differentiable, we first let $h \to 0$ and use 
\eqref{weak-con-4} to deduce 
\[
\int_0^T \intO{ \Big( \vu : \Div \mathbb{T} + \Ds \vu : \mathbb{T} \Big) } \dt =0,  
\]
which, after integration-by-parts, yields 
\[
\int_0^T \int_{\partial \Omega} \vu \cdot [\mathbb{T} \cdot \vc{n}] \D S_x \dt = 0 \ \mbox{for any continuous symmetric tensor}\ \mathbb{T} = \mathbb{T}(t,x).
\]
The same argument applied to \eqref{cf2} yields 
\[
\int_0^T \int_{\partial \Omega} (\vt - \vtB) \Psi \cdot \vc{n} \D S_x \dt = 0 \ \mbox{for any continuous vector field}\ \Psi = \Psi (t,x).
\]	
Thus, we conclude
\begin{equation} \label{bcu} 
	\vu|_{\partial \Omega} = 0 \ \mbox{in the sense of traces, } 
\end{equation}	
and 
\begin{equation} \label{bct} 
	\vt|_{\partial \Omega} = \vtB \ \mbox{in the sense of traces.}
\end{equation}	

Finally,  we deduce
\begin{align*}
\abs{(\tvr_h, \tvu_h, \tvt_h) - (\vrh, \vuh, \vth) } \leq \TS \abs{D_t (\vrh, \vuh, \vth) };  
\end{align*}
{whence, thanks to the estimates \eqref{ap2} and the uniform boundedness of the numerical solutions, we get}
\begin{equation} \label{prox1}
\| (\tvr_h, \tvu_h, \tvt_h) - (\vrh, \vuh, \vth) \|_{L^q((0,T) \times \Omega; \R^{d+2})} \to 0 \ \mbox{ as }\ h \to 0
\end{equation} 
for any $1 \leq q < \infty$. In particular, 
\begin{equation} \label{weak-con-2}
(\tvr_h, \tvu_h, \tvt_h)  \to \ (\vr, \vu, \vt) \mbox{ weakly-(*) in }  L^\infty ((0,T)\times \Omega; \R^{d+2}).
\end{equation}

To complete the proof of Theorem \ref{thm_main}, 
it remains to show the weak convergences \eqref{weak-con-1} and \eqref{weak-con-2} are, in fact, strong, and the limit $(\vr, \vu, \vt)$ is a weak solution of the Navier-Stokes-Fourier system in $(0,T) \times \Omega$.

\subsection{Strong convergence of the approximate velocities}

We start with a simple argument based on the Div-Curl Lemma applied systematically in the section. 
In accordance with the consistency formulation of the equation of continuity, we have 
\begin{equation} \label{divcurl-d}
	\pd_t \widetilde{\vrh} + \Div (\vrh \vuh)  \in_b L^2(0,T; W^{-1,2}(\Omega)) 
	\ \mbox{as}\ h \to 0.
\end{equation}
Moreover, in accordance with the convergence \eqref{weak-con-1}, \eqref{weak-con-3}, 
we can write 
\[
\Grad \vu_h = \nabla_h \vu_h + ( \Grad \vu_h - \nabla_h \vu_h) 
\]
where  $\nabla_h \vu_h \in_b L^2((0,T)\times\Omega; \R^{d\times d})$. In addition,  it follows from the compatibility estimate \eqref{cf1} that
\[
( \Grad \vu_h - \nabla_h \vu_h) \to 0 \ \mbox{in}\ L^2(0,T; W^{-1,2}(\Omega; \R^{d \times d})). 
\]
Thus, we may apply a ``compensated compactness'' argument, specifically  
Proposition~\ref{col-div-curl} with $r_h = \widetilde{\vrh}, \, v_n = \vuh^{(j)}$ to conclude
\begin{align}\label{68}
\widetilde{\vrh} \vuh  \to \vr \vu  \ \mbox{ weakly-(*) in } L^\infty((0,T)\times \Omega; \R^d).
\end{align}
In addition, combining \eqref{68} with the bounds
\begin{align*}
\abs{\widetilde{\vrh} \vuh - \vrh \vuh} \aleq \TS \abs{D_t  \vrh}, \quad 
\norm{D_t  \vrh}_{L^2((0,T)\times \Omega)} \aleq \TS^{-1/2},\\
\abs{\widetilde{\vrh \vuh} - \vrh \vuh} \aleq \TS (\abs{D_t  \vrh}+\abs{D_t  \vuh}), \quad 
\norm{D_t  \vuh}_{L^2((0,T)\times \Omega;\R^{d})} \aleq \TS^{-1/2}, 
%\cmag \abs{\widetilde{\vrh} \vuh - \widetilde{\vrh \vuh}} \aleq \TS \abs{D_t  \vuh}, \quad 
%\norm{D_t  \vuh}_{L^2((0,T)\times \Omega;\R^{d})} \aleq \TS^{-1/2},
\end{align*}
we obtain
\begin{equation} \label{prox2}	
\left\| \widetilde{\vrh \vuh} - \vrh \vuh \right\|_{L^q((0,T) \times \Omega; \R^d)} \to 0 \ \mbox{as}\ h \to 0 \ \mbox{for any}\ 
1 \leq q < \infty,
\end{equation}
together with
\begin{align}\label{con-m}
\vrh \vuh  \to \vr \vu  \ \mbox{ weakly-(*) in } L^\infty ((0,T)\times \Omega; \R^d), \\ \label{con-m-1}
\widetilde{\vrh \vuh}  \to \vr \vu  \ \mbox{ weakly-(*) in } L^\infty((0,T)\times \Omega; \R^d).
\end{align}

Letting $h \to 0$ in the consistency formulation  \eqref{cons-1} of the equation of continuity, we conclude that the limits $\vr, \vu$ satisfy the weak formulation of the equation of continuity \eqref{w3}.  
Moreover, as $\vr$ and $\vu$ are bounded  and 
$\Grad \vu \in L^2((0,T) \times \Omega; \R^{d \times d})$, we can apply the nowadays standard DiPerna-Lions theory~\cite{DiPena-Lions} to deduce the renormalized formulation \eqref{w4}.

\medskip
Next, using the consistency formulation of the momentum equation stated in Lemma \ref{Lcons2}, we get  
\begin{align*}
& \pd_t \widetilde{\vrh \vuh} + \Div (\vrh \vuh \otimes \vuh - \bS_h + p_h \bI)  = \vrh \vc{g} + h_{\vm}, \\
& \mbox{ with } \vrh \vc{g} \mbox{ bounded in } L^\infty((0,T)\times\Omega;\R^d), \quad
h_{\vm} \in_b {L^2(0,T;W^{-1,2}(\Omega;\R^d))}.
\end{align*}
Consequently, following the same argument as in the proof of \eqref{con-m}, 
we apply Proposition~\ref{col-div-curl} with $r_h = \widetilde{\vrh \vuh^{(i)}}, \, v_h = \vuh^{(j)}$ to obtain 
\begin{align}
&\widetilde{\vrh \vuh} \cdot \vuh \to \vr \abs{\vu}^2  \ \mbox{ weakly-(*) in } L^\infty((0,T)\times \Omega),\\ \label{con-mmd-tilde}
& \widetilde{\vrh \vuh} \otimes \vuh \to \vr \vu \otimes \vu  \ \mbox{ weakly-(*) in } L^\infty((0,T)\times \Omega; \R^{d\times d}), 
\end{align}
which yields
\begin{align}\label{con-mmd}
&\vrh \vuh \cdot \vuh \to \vr \abs{\vu}^2  \ \mbox{ weakly-(*) in } L^\infty((0,T)\times \Omega),\\ 
& \vrh \vuh \otimes \vuh \to \vr \vu \otimes \vu  \ \mbox{ weakly-(*) in } L^\infty((0,T)\times \Omega; \R^{d\times d}).
\end{align}

%Furthermore, we know from Lemma \ref{lem_CI} that $\Gradh (\abs{\vuh}^2)$ satisfies the compatibility equation \eqref{cf2} and belongs to the regularity class  $L^p((0,T)\times \Omega; \R^d)$ with $p \in \left[1, \frac{3}{3+\alpha} \right], \alpha \in (-1,0]$.
Furthermore, we know from Corollary~\ref{cor1} that $\Gradh (\abs{\vuh}^2)$ satisfies the compatibility estimates \eqref{cf2}.
Moreover, Corollary \ref{Cor1} for $G(v) = v^2$ yields
$\Gradh (|\vu_h|^2 ) \in_b $ $L^2((0,T)\times \Omega; \R^d)$.
Hence, we apply Proposition~\ref{col-div-curl} once more, this time with $r_n = \vr - \widetilde{\vrh}, \, v_n = \abs{\vuh}^2$,  obtaining
\begin{align}
(\vr - \widetilde{\vrh}) \abs{\vuh}^2 \to 0 \mbox{ weakly-(*) in } \ L^\infty((0,T)\times\Omega),\\ \label{con-mmd-1}
(\vr - \vrh) \abs{\vuh}^2 \to 0 \mbox{ weakly-(*) in } \ L^\infty((0,T)\times\Omega). 
\end{align}

Finally, we infer from \eqref{con-mmd} and \eqref{con-mmd-1} that 
\begin{align*}
\vr  \abs{\vuh}^2 \to \vr  \abs{\vu}^2 \mbox{ weakly-(*) in } \ L^\infty((0,T)\times\Omega). 
\end{align*}
As $\vr$ is bounded and strictly positive, and $\vuh$ converges weakly to $\vu$, we establish the strong convergence of the velocity in the Hilbert space $L^2((0,T) \times \Omega; \R^d)$. This, 
combined with the uniform bounds \eqref{HP1}, yields the final conclusion
\begin{align}\label{con-u}
\vuh \to \vu \mbox{ in } \ L^q((0,T)\times\Omega; \R^d) \ \mbox{for any}\ 1 \leq q  < \infty, 
\end{align}
or, up to a subsequence, 
\begin{align}\label{con-u1}
	\vuh \to \vu \ \mbox{a.a. in} \ (0,T)\times\Omega. 
\end{align}

\subsection{Strong convergence of the approximate temperatures}

We follow the arguments of the preceding part. First, 
thanks to \eqref{weak-con-1}, \eqref{divcurl-d}, \eqref{ap1} and \eqref{cf3}, we can apply Proposition~\ref{col-div-curl}  with $r_n = \widetilde{\vrh}$ and $v_n = \vth$ to conclude 
\begin{align*}
\widetilde{\vrh} \vth  \to \vr \vt \ \mbox{ weakly-(*) in } L^\infty((0,T)\times \Omega).	
\end{align*}
Since
\begin{align*}
\abs{\widetilde{\vrh} \vth - \vrh \vth} \aleq \TS \abs{D_t  \vrh}, \quad 
\norm{D_t  \vrh}_{L^2((0,T)\times \Omega)} \aleq \TS^{-1/2},
\end{align*}
we get
\begin{align}\label{con-p}
\vrh \vth  \to \vr \vt \ \mbox{ weakly-(*) in } L^\infty((0,T)\times \Omega).	
\end{align}
Passing to the limit $h\to 0$ in the consistency equation of momentum \eqref{cons-2}, we obtain from \eqref{con-m}, \eqref{con-mmd}, and \eqref{con-p} that the limits $\vr, \vu, \vt$ satisfy the momentum equation \eqref{w5}. Note carefully that the pressure is a \emph{linear} function of the density.

\medskip
Next, recalling the consistency formulation of the internal energy \eqref{cons3}, we have 
\begin{align*}
&  \pd_t \widetilde{\vrh \vth} + \Div (c_v \vrh \vth  \vuh - \kappa  \Gradd \vth)  = (\bS_h-p_h \I)  : \Gradh \vuh  + h_{\vt},
\end{align*}
where, thanks to the consistency estimates stated in Lemma \ref{Lcons3}, 
\begin{align*}
(\bS_h-p_h \I)  : \Gradh \vuh  \in_b L^1((0,T)\times\Omega),\quad
h_\vt \in_b L^2(0,T; W^{-1,2}(\Omega)).
\end{align*}
Applying Proposition~\ref{col-div-curl}  with $r_n = \widetilde{\vrh \vth}$ and $v_n = \vth$, we obtain
\begin{align*}
&\widetilde{\vrh \vth} \cdot \vth \to \vr \vt^2  \ \mbox{ weakly-(*) in } L^\infty((0,T)\times \Omega), \\
&\vrh \vth \cdot \vth \to \vr \vt^2 \ \mbox{ weakly-(*) in } L^\infty((0,T)\times \Omega).
\end{align*}
%Furthermore, thanks to the compatibility results  \eqref{cf3}, \eqref{cf4}, with the bounds \eqref{regularity-sq},  
Furthermore, thanks to the compatibility results  \eqref{cf3}, \eqref{cf4}, and Corollary \ref{Cor1} for $G(v) = v^2$, 
we apply Proposition~\ref{col-div-curl} once again with $r_n = \vr - \widetilde{\vrh}, \, v_n = \vth^2$ and obtain 
\begin{align*}
&(\vr - \widetilde{\vrh}) \vth^2 \to 0 \mbox{ weakly-(*) in } \ L^\infty((0,T)\times\Omega), \\
&(\vr - \vrh) \vth^2 \to 0 \mbox{ weakly-(*) in } \ L^\infty((0,T)\times\Omega).
\end{align*}
Consequently, 
\begin{align*}
\vr  \vth^2 \to \vr  \vt^2 \mbox{ weakly-(*) in } \ L^\infty((0,T)\times\Omega).
\end{align*}
Thus, similarly to preceding section, we obtain the strong convergence of the temperature, specifically
\begin{align}\label{con-T}
\vth \to \vt \mbox{ in } \ L^q((0,T)\times\Omega) \ \mbox{for any}\ 1 \leq q  < \infty, 
\end{align} 
or, extracting a suitable subsequence as the case may be, 
\begin{align}\label{con-T1}
	\vth \to \vt  \ \mbox{a.a. in}\ (0,T)\times\Omega.
\end{align}

\subsection{Final step of the proof of convergence} 

Summarizing the previous discussion, we have already established: 
\begin{itemize} 
	\item Strong (a.a.\ pointwise) convergence of the sequence $\{ \vth, \vuh \}_{h \searrow 0}$; weak convergence of $\{ \vrh \}_{h \searrow 0}$ as well as of the approximate 
		gradients $\{ \nabla_h \vuh \}_{h \searrow 0}$, $\{ \nabla_{\mathcal{E}} \vth \}_{h \searrow 0}$.
	\item Satisfaction of the boundary conditions \eqref{w6}. 	 
	\item Validity of the weak formulation of the equation of continuity \eqref{w3}, its renormalized variant \eqref{w4}, 
	as well as the momentum equation \eqref{w5}.
\end{itemize}

Thus, it remains to show the strong convergence of the approximate densities, 
\begin{equation} \label{strongd}
	\vrh \to \vr \ \mbox{ a.a. in }\ (0,T) \times \Omega, 	
\end{equation}
and the validity of both the entropy inequality \eqref{w7} and the ballistic energy inequality \eqref{w8}.	

Assume for a moment we have already shown \eqref{strongd}. 
Then the limiting passage in the approximate entropy inequality \eqref{cons3-a}
as well as the ballistic energy inequality \eqref{cons-4} is a routine matter. Indeed, by virtue of the weak lower semi--continuity 
of convex functionals and the convergences \eqref{weak-con-3}, \eqref{weak-con-4} we have 
\begin{align} 
\liminf_{h \to 0} &\intTO{ \frac{\phi}{\vth} \left(\kappa \frac{ \chi_h}{ \vth} \abs{\Gradd \vth }^2 + \difuh  \right)} \br &\geq 
 \intTO{ \frac{\phi}{\vt} \left(\kappa \frac{|\nabla \vt|^2 }{\vt}  + \mathbb{S}(\Grad \vu) : \Grad \vu  \right)}
\label{nonumber}
\end{align}
for any $\phi \geq 0$. The limit in the remaining terms follows from the pointwise convergence of all quantities in question. Consequently, 
the desired conclusion \eqref{w7}, \eqref{w8} follows from 
the consistency estimates stated in Lemmas \ref{Lcons3}, \ref{Lconsbe}.  

We may infer that the proof of Theorem \ref{thm_main} is complete as soon as we establish \eqref{strongd}. This is indeed 
the most delicate part of the proof. 
As this result is of independent interest, we state it as 
a separate Theorem \ref{THM:density} proved in the next section.

\section{Strong convergence of the density}\label{app:con-d}
In this section, we complete the proof of 
Theorem \ref{thm_main} by showing the strong convergence of the density. 
\begin{Theorem}\label{THM:density}
Under the hypotheses of Theorem \ref{thm_main}, there holds
\begin{align}\label{conv-d}
\vrh \to \vr \mbox{ in } \ L^q((0,T)\times\Omega) \ \mbox{for any}\ 1 \leq q  < \infty.
\end{align}
\end{Theorem}
The key ingredient of the proof of Theorem \ref{THM:density} is a discrete version of the so--called Lions identity (cf. Lions \cite{Lions}) stated in the following lemma. 

\begin{Lemma}\label{lem}
Under the hypotheses of Theorem \ref{thm_main}, there holds
\begin{align}\label{lions-tool}
\lim_{h\to 0} \int_{0}^{T}\intO{\phi \psi \vrh \Big( \vrh \vth - (2\mu+\lambda) \Divh \vuh\Big)} \dt =  \int_{0}^{T}\intO{\phi \psi \vr \Big( \vr  \vt - (2\mu+\lambda)  \Div \vu\Big)} \dt
\end{align}
for any $\psi \in C_c^{\infty}((0,T)\times \Omega)$ and $\phi \in C_c^{\infty}(\Omega)$.
\end{Lemma}

To the best of our knowledge, there is only one result similar 
to Lemma \ref{lem} in the context of numerical analysis, namely that of Karper \cite{Karper}. Our aim is to show  \eqref{lions-tool}
for the present finite-volume scheme. Although the proof shares some similarity with \cite{Karper}, the main ideas are rather different since we only work with piecewise constant approximate solutions.

Before proving Lemma \ref{lions-tool} and Theorem \ref{THM:density} we introduce the necessary notation and some preliminary estimates. For simplicity, we focus on the case $d=3$; the case $d=2$ can be handled similarly. 

\subsection{Notation}
Let us introduce an interpolation operator
\begin{align}
\PiL: Q_h \to \mathcal{P}_h, \quad  \PiL r_h (x_{\sigma}) = \frac1{\mbox{card}(N_{\sigma})} \sum_{K\in N_\sigma} r_h|_K,
\end{align}
where $\mathcal{P}_h$ is the space of continuous piecewise linear functions,  $x_{\sigma}$ is a generic vertex and $N_{\sigma}$ is the set of elements $K\in \mesh$ having $x_{\sigma}$ as a  vertex.
We recall the following estimates of $\PiL$ \cite[Lemma 5.8]{GL}: %or \cite[Lemma 8.1]{Karper}
\begin{subequations}\label{es-lift}
\begin{align}
& \norm{\Grad \PiL r_h}_{L^2(\Omega)}  \leq C \left( \intfacesint{ \frac{\jump{r_h}^2}{h}}  \right)^{1/2}, \\
&\norm{\PiL r_h - r_h}_{L^2(\Omega)} \leq h C \left( \intfacesint{ \frac{\jump{r_h}^2}{h}}  \right)^{1/2}
\end{align}
\end{subequations}
for any $r_h \in Q_h$.

\medskip
Next, we introduce the following pseudo-differential operators
\begin{subequations}\label{op-lift}
\begin{align}
& \mathcal{A}^i: L^p(\Omega) \to W^{1,p}(\Omega), \quad  \mathcal{A}^i(r) = \frac{\pd }{\pd x_i} \Delta^{-1} \left[ \mathds{1}_\Omega r \right]|_{\Omega}, \quad i = 1,2,3, \\
& \mathcal{A}^{\nabla}[r] = (\mathcal{A}^1[r], \ \mathcal{A}^2[r], \ \mathcal{A}^3[r]), \\
& \mathcal{A}^{\rm{div}}[\vv] = \mathcal{A}^1[v_1]+ \mathcal{A}^2[v_2]+ \mathcal{A}^3[v_3], \quad \vv \equiv (v_1,v_2,v_3),
\end{align}
\end{subequations}
where $\Delta^{-1}$ is defined on the whole space $\R^3$ as the convolution with the Newtonian potential 
\[
\Delta^{-1} \phi (x) = - \frac{1}{4 \pi} \int_{\R^3} \frac{\phi(y)}{|x-y|} {\rm d}y.
\] 

We recall the basic properties of these operators proved in \cite{Fe:2004,Karper}.
\begin{Lemma}[{\cite[Lemma 6.1]{Karper}}]
For any $\vv\in L^p(\Omega;\R^3)$, $g\in L^q(\Omega)$ with $1/p+1/q=1$ and $p,q\in [1,\infty)$ there holds
\begin{align}\label{op-lift-IBP}
\intO{ \vv \cdot \mathcal{A}^{\nabla}[g \phi]} = - \intO{ \mathcal{A}^{\rm{div}}[\vv]\, g \phi}, \quad \forall \phi \in C_0^{\infty}(\Omega). 
\end{align}
In addition, 
\begin{align}
 \norm{\Grad\mathcal{A}^i[f]}_{L^q(\Omega)} &\leq C \norm{f}_{L^q(\Omega)} \ \mbox{ for any }\ 1 < q < \infty, \br
 \norm{\mathcal{A}^i[f]}_{L^p(\Omega)} &\leq C \norm{f}_{L^q(\Omega)} 
 \ \mbox{ for }\ p \leq \frac{3q}{3 - q},\ 1 \leq q < 3, \ 
 p \ \mbox{ arbitrary finite if }\ q = 3, \br
 \| \mathcal{A}^i[f] \|_{C^{0, \nu}(\Ov{\Omega}) } &\leq C \norm{f}_{L^q(\Omega)} \ \mbox{ if }\ 0 < \nu < \frac{3}{q}.
\label{op-lift-IBP-1}
\end{align}
\end{Lemma}

\begin{Lemma}[{\cite[Corollary 6.1]{Fe:2004}}]\label{lem-divcurl-1}
Let $v_n$ and $w_n$ be sequences of vector valued functions such that $v_n \to v$ weakly in $L^p(\Omega;\R^d)$ and $w_n \to w$ weakly in $L^q(\Omega;\R^d)$ with $1/p+1/q \leq 1$.
Then
\[
v_n \cdot \Grad \mathcal{A}^{\rm div}[w_n] - w_n \cdot \Grad \mathcal{A}^{\rm div}[v_n]
\to v \cdot \Grad \mathcal{A}^{\rm div}[w] - w \cdot \Grad \mathcal{A}^{\rm div}[v]
\qquad \mbox{in } \mathcal{D}'(\Omega).
\]
Similarly, if $B_n \to B$ weakly in $L^p(\Omega)$, then
\[
B_n \Grad \mathcal{A}^{\rm div}[w_n] - \Grad \mathcal{A}^{\nabla}[B_n]\, w_n
\to B \Grad \mathcal{A}^{\rm div}[w] - \Grad \mathcal{A}^{\nabla}[B]\, w
\qquad  \quad \mbox{in } \mathcal{D}'(\Omega;\R^d).
\]
\end{Lemma}

Lemma \ref{lem-divcurl-1} is the so-called Commutator Lemma
and can be seen as a direct consequence of the Div-Curl lemma. It 
is one of the main ingredients of the existence theory developed in 
\cite{Fe:2004}. 

\subsection{Preliminary estimates} 

Following the proof in the continuous case, we show Lemma~\ref{lem} 
by considering the quantity $\psi \mathcal{A}^{\nabla}[\phi \vr_h]$ as the test function in the approximate momentum equation, where
$\psi \in C_c^{\infty}((0,T)\times \Omega)$ and  
$\phi \in C^\infty_c(\Omega)$. Accordingly, we define
\begin{align}\label{test-func}
\vv=  \psi \mathcal{A}^{\nabla}[\phi \vr_h], \quad \vvh = \PiQ \vv; \quad  \vv_L =  \psi \mathcal{A}^{\nabla}[\phi \PiL\vr_h].
\end{align}
Then, it holds 
\begin{align}\label{curl-d}
& \Curl \vv %\equiv \Curl \Big( \psi \mathcal{A}^{\nabla}[\phi \vr_h]) \Big) 
= \Grad \psi \times \Big(  \mathcal{A}^{\nabla}[\phi \vr_h] \Big) + \psi \Curl  \Big(  \mathcal{A}^{\nabla}[\phi \vr_h] \Big) = \Grad \psi \times \Big(  \mathcal{A}^{\nabla}[\phi \vr_h] \Big), \\ \label{div-d}
&\Div \vv %\equiv \Div \Big( \psi \mathcal{A}^{\nabla}[\phi \vr_h]) \Big) 
= \Grad \psi \cdot \Big(  \mathcal{A}^{\nabla}[\phi \vr_h] \Big) + \psi \Div  \Big(  \mathcal{A}^{\nabla}[\phi \vr_h] \Big) = \Grad \psi \cdot \Big(  \mathcal{A}^{\nabla}[\phi \vr_h] \Big) + \psi \phi \vr_h,
\end{align}
due to the fact that $
\Curl \mathcal{A}^{\nabla}[r] = 0,  \Div\mathcal{A}^{\nabla}[r] = r $ for any $r \in L^p(\Omega)$.

The next step is to use the quantity 
\[
\bfphi_h = \vc{v}_h = \Pi_Q [ \psi \mathcal{A}^\nabla [\phi \vrh] ] 
\]
as a test function in the approximate momentum balance \eqref{VFV_M}. Note carefully that this is completely legal, as both 
$\psi$ and $\phi$ are compactly supported; whence vanish in a neighbourhood of the boundary together with $\vc{v}_h$  and a sufficient small $h$. Thus, we obtain 
\begin{align}\label{VFV_M1}
& \intTO{D_t  (\vrh  \vuh ) \cdot  \vvh} - \int_{0}^{T} \intfacesint{ {\bf F}_h^{\alpha}  (\vrh^{k}  \vuh^{k} ,\vuh^{k} ) \cdot \jump{\vvh}   }\dt   \br
& \quad = \intTO{\Divh (\bS_h -p_h \bI ) \cdot \vvh} + \intTO{\vrh \vc{g} \cdot \vvh}.
\end{align}
 Applying the integration-by-parts formula \eqref{InByPa}, we have
\begin{align}\label{eq-lem}
&\intTO{ D_t  (\vrh  \vuh ) \cdot \vvh } - \int_{0}^{T} \intfacesint{ {\bf F}_h^{\alpha}  (\vrh \vuh ,\vuh ) \cdot \jump{\vvh} }\dt -\intTO{\vrh \vc{g} \cdot  \vvh}  \br
& \quad =  -\intTO{ (\bS_h -p_h \bI ) : \Gradh \vvh }.
\end{align}
Here and hereafter, we extend $\vvh^{out}|_{\sigma} = 0$ and $(\Gradh \vvh)^{out}|_{\sigma}  = 0$ for any $\sigma \in \facesext$.

\begin{Lemma}\label{lem-vis}
	
Let us set 	
\begin{align}
&\mathcal{F}_1(\vr_h) \br &= \intTO{ \Gradh \vc{u}_h : \Gradh \vvh} - 
  \intTO{ \left(  \Curlh \vc{u}_h \cdot \left(\Grad \psi \times  \mathcal{A}^{\nabla}[\phi \vr_h]\right)  + \Divh \vc{u}_h  \Grad \psi \cdot  \mathcal{A}^{\nabla}[\phi \vr_h]  \right)} \br
&\quad - \intTO{ \psi \phi \vr_h \Divh \vc{u}_h} , %\label{F1}
\nonumber
\end{align}
\begin{align}
&\mathcal{F}_2(\vr_h) \br &=  \intTO{ \Divh \vuh \, \Divh \vvh}  -
 \intTO{ \Divh \vc{u}_h  \Grad \psi \cdot  \mathcal{A}^{\nabla}[\phi \vr_h] }  - \intTO{ \psi \phi \vr_h \Divh \vc{u}_h}. %\label{F2}
\nonumber
\end{align}

Under the hypotheses of Theorem \ref{thm_main}, we have
\begin{align*}
|\mathcal{F}_1(\vr_h)| + |\mathcal{F}_2(\vr_h)|  \leq C (\Lambda) h^{(1-\alpha)/2}.
\end{align*}
\end{Lemma}
\begin{proof}
%{\cgrey Recalling Lemma \ref{lem-lap} we know that
%\begin{align*}
%\intTO{ \Gradh \vc{u}_h : \Gradh \vvh}
%& = \intTO{ \left( \Curlh \vuh \cdot \Curl \vv_L+\Divh \vuh \cdot \Div \vv_L \right)} + \int_{0}^{T}  E(\vuh,\vvh) \dt. 
%\end{align*}
%where $E(\vuh, \vvh) = \intO{ \Gradh \vuh : \Gradh \vvh} - \intO{\Curlh \vuh \cdot \Curl \vv_L} - \intO{\Divh \vuh \cdot \Div \vv_L}$
%satisfies 
%\begin{align*}
%E(\vuh,\vvh)  \leq \norm{\Gradh \vuh}_{L^2(\Omega)} \norm{\Gradh \vvh - \Grad \vv_L}_{L^2(\Omega)} .
%\end{align*}
%
%Using the formulae \eqref{curl-d} and \eqref{div-d} that 
%\begin{align*}
%\intTO{ \Gradh \vc{u}_h : \Gradh \vvh } 
%& = \intTO{ \left( \Curlh \vuh \cdot \Curl \vv+\Divh \vuh \cdot \Div \vv \right)} +  \mathcal{F}_1(\vr_h) \\ 
%& = \intTO{ \left(  \Curlh \vc{u}_h \cdot \left(\Grad \psi \times  \mathcal{A}^{\nabla}[\phi \vr_h]\right)  + \Divh \vc{u}_h  \Grad \psi \cdot  \mathcal{A}^{\nabla}[\phi \vr_h]  \right)} \\
%& + \intTO{ \psi \phi \vr_h \Divh \vc{u}_h  }+ \mathcal{F}_1(\vr_h),
%\end{align*}
%where we have denoted 
%\begin{align*}
%\mathcal{F}_1(\vr_h) & = \int_{0}^{T}  E(\vuh,\vvh) \dt + \intTO{ \Big( \Curlh \vc{u}_h  \cdot \Curl (\vv_L - \vv)  +  \Divh \vc{u}_h  \cdot \Div (\vv_L - \vv)  \Big)}. 
%\end{align*}}
Thanks to the identities \eqref{curl-d} and \eqref{div-d} we can rewrite $\mathcal{F}_1(\vr_h)$ in the form  
\begin{align*}
\mathcal{F}_1(\vr_h)  
& = \intTO{ \Gradh \vc{u}_h : \Gradh \vvh } - \intTO{ \left( \Curlh \vuh \cdot \Curl \vv+\Divh \vuh \cdot \Div \vv \right)}  \\ 
&= \int_{0}^{T}  E(\vuh,\vvh) \dt + \intTO{ \Big( \Curlh \vc{u}_h  \cdot \Curl (\vv_L - \vv)  +  \Divh \vc{u}_h  \cdot \Div (\vv_L - \vv)  \Big)},  
\end{align*}
where, by virtue of Lemma \ref{lem-lap} in Appendix \ref{sec-diffop}, 
\begin{align}
	E(\vuh, \vvh) = \intO{ \Gradh \vuh : \Gradh \vvh} - \intO{\Curlh \vuh \cdot \Curl \vv_L} - \intO{\Divh \vuh \cdot \Div \vv_L}.
\end{align}

Next, we estimate the integrals on the right--hand side of the above inequality. 
First we claim
\[
\abs{\int_{0}^{T}  E(\vuh,\vvh) \dt} \leq \norm{\Gradh \vuh}_{L^2((0,T)\times\Omega)} \norm{\Gradh \vvh - \Grad \vv_L}_{L^2((0,T)\times\Omega)} \aleq h^{(1-\alpha)/2}
\]
since
\begin{align*}
& \norm{\Gradh \vvh - \Grad \vv_L}_{L^2(\Omega)} \leq  \norm{\Gradh \vvh - \Gradh \PiQ \vv_L}_{L^2(\Omega)} +  \norm{\Gradh \PiQ \vv_L - \Grad \vv_L}_{L^2(\Omega)} \\
& =  \norm{\Gradh \PiQ (\vv -  \vv_L)}_{L^2(\Omega)} +\norm{\Gradh \PiQ \vv_L - \Grad \vv_L}_{L^2(\Omega)} \\
& \aleq \norm{\Grad(\vv - \vv_L)}_{L^{2}(\Omega)} + h \norm{\Grad^2 \vv_L}_{L^{2}(\Omega)} \\
& \aleq \norm{\vrh - \PiL\vrh}_{L^{2}(\Omega)} + h \norm{\Grad \PiL\vrh}_{L^{2}(\Omega)} \aleq h \left( \intfacesint{\frac{\jump{\vrh}^2}{h}} \right)^{1/2}.
\end{align*}
Here, we have used estimates \eqref{es-lift}, \eqref{op-lift-IBP-1} in the last two inequalities. The most right integral is controlled by \eqref{ap3} 
as required. 

The second term can be handled as
\begin{align*}
&\abs{\intTO{ \Big( \Curlh \vc{u}_h  \cdot \Curl (\vv_L - \vv)  +  \Divh \vc{u}_h  \cdot \Div (\vv_L - \vv)  \Big)} } \\
&\aleq  \norm{\Gradh \vuh}_{L^2((0,T)\times\Omega)} \norm{\Grad( \vv -  \vv_L)}_{L^2((0,T)\times\Omega)} \aleq \norm{\vrh - \PiL\vrh}_{L^{2}([0,T]\times\Omega)}  \aleq h^{(1-\alpha)/2}.
\end{align*}
 The estimate of $\mathcal{F}_2(\vr_h)$ can be done in a similar manner.
\end{proof}

\begin{Lemma}\label{lem-bS}
%Let  $\vvh, \Gradh \vvh$ extended with $\jump{\vvh}|_{\facesext}=0, \jump{\Gradh \vvh}\cdot \vc{n}|_{\facesext}=0$. Denote
Let
\begin{align*}
\mathcal{F}_3(\vr_h) = &\intTO{ \bS_h : \Gradh \vvh }  - 
 \mu \intTO{  \Curlh \vc{u}_h \cdot \left(\Grad \psi \times  \mathcal{A}^{\nabla}[\phi \vr_h]\right)  } \\
& -(2\mu+\lambda) \intTO{  \Divh \vc{u}_h  \Grad \psi \cdot  \mathcal{A}^{\nabla}[\phi \vr_h] } 
%\\& 
-(2\mu+\lambda) \intTO{ \psi \phi \vr_h \Divh \vc{u}_h } . 
\end{align*}

Under the hypotheses of Theorem \ref{thm_main}, we have 
\[ 
|\mathcal{F}_3|  \leq C(\Lambda) h^{(1-\alpha)/2}.\]
\end{Lemma}

\begin{proof}
With $\bS_h = 2\mu \bD_h \vuh + \lambda \Divh\vuh \bI$, we have
\begin{align*}
\bS_h : \Gradh \vvh & = \mu \Gradh \vuh : \Gradh \vvh + \mu \Gradh^T \vuh : \Gradh \vvh + \lambda \Divh \vuh \, \Divh \vvh\\
&= \mu \Gradh \vuh : \Gradh \vvh + \mu \Gradh \vuh : \Gradh^T \vvh + \lambda \Divh \vuh \, \Divh \vvh.
\end{align*}
Thanks to Lemma~\ref{lem-op-comp} and the integration-by-parts formula \eqref{InByPa-1} in Appendix \ref{sec-diffop}, we reformulate the second term on the right-hand-side as
\begin{align*}
\intO{\Gradh \vuh : \Gradh^T \vvh} = - \intO{ \vuh \cdot \Divh\Gradh^T \vvh} = - \intO{ \vuh \cdot \Gradh \Divh \vvh} = \intO{ \Divh\vuh \, \Divh \vvh}, 
\end{align*}
which leads to 
\begin{align}
\intTO{ \bS_h : \Gradh \vvh } = \mu \intTO{ \Gradh \vuh : \Gradh \vvh } + \left( \mu + \lambda \right) \intTO{ \Divh \vuh \, \Divh \vvh }.
\label{coerc}
\end{align}

Using Lemma~\ref{lem-vis} we finish the proof since $\mathcal{F}_3(\vr_h) = \mu \mathcal{F}_1(\vr_h) + (\mu+\lambda)\mathcal{F}_2(\vr_h)$.
\end{proof}

\begin{Lemma}\label{lem-p}
Let us set 
\begin{align*}
\mathcal{F}_4(\vrh) = \intTO{ p_h \Divh \vvh } - \intTO{ p_h \Grad \psi \cdot \Big(  \mathcal{A}^{\nabla}[\phi \vr_h]) \Big)  } -  \intTO{ \psi \phi \vr_h p_h } . 
\end{align*}

Under the hypotheses of Theorem \ref{thm_main}, there holds
\[ 
|\mathcal{F}_4(\vr_h)| \leq C(\Lambda)\left(  h^{(1-\alpha)/2} + h^{1/2} \right).
\]
\end{Lemma}

\begin{proof}
Thanks to the identities \eqref{curl-d} and \eqref{div-d}, direct calculations give
\begin{align*}
\mathcal{F}_4(\vrh) & = \int_{0}^{T}\intO{ p_h \Divh \vvh } \dt-  \int_{0}^{T} \intO{ p_h \Div \vv } \dt\\
%& = \sum_{K\in \mesh} \sum_{\sigma \in \facesK} |\sigma| p_h \avs{\vvh} \cdot \vn - \sum_{K\in \mesh} \sum_{\sigma \in \facesK} |\sigma| p_h \PiF \vv \cdot \vn \\
& = \int_{0}^{T} \sum_{K\in \mesh} \sum_{\sigma \in \facesK} |\sigma| p_h \big( \avs{\vvh} - \vv \big) \cdot \vn \dt \\
& = \int_{0}^{T} \intfacesint{ \jump{p_h} \big( \Pi_{W^{(i)}} \vvh -  \vv \big)  \cdot \vn} \dt.
\end{align*}
Consequently, 
\begin{align*}
\abs{\mathcal{F}_4(\vrh)} 
&\aleq \left(\int_{0}^{T}\intfacesint{ \jump{p_h}^2}\dt\right)^{1/2} \left(\int_{0}^{T}\intfacesint{ \abs{ \Pi_{W^{(i)}} \vvh -  \vv }^2}\dt\right)^{1/2} \\
&\aleq h^{-\alpha/2} h^{1/2}  \norm{\Grad \vv}_{L^2} + h^{1/2}  \norm{\Grad \vv}_{L^2} \aleq \big(h^{1/2} + h^{(1-\alpha)/2} \big)  \norm{\vrh}_{L^2}  \aleq h^{(1-\alpha)/2} + h^{1/2},
\end{align*}
which completes the proof of Lemma~\ref{lem-p}.
\end{proof}

\begin{Lemma}\label{lem-f}
Let 
\begin{align}
&  \intTO{ D_t  (\vrh  \vuh ) \cdot \vvh } - \int_{0}^{T} \intfacesint{ {\bf F}_h^{\alpha}  (\vrh \vuh ,\vuh ) \cdot \jump{\vvh} }\dt \br
=\ & \intTOB{    \vrh \vuh \phi \cdot \Grad \left(  \mathcal{A}^{\rm div} [\vrh  \vuh \psi]   \right) -  \vrh\vuh \otimes \vuh  \psi  : \Grad ( \mathcal{A}^{\nabla}[\phi \vr_h])} \br
+\ & \intTOB{ \vrh \vuh \mathcal{A}^{\rm div} [\vrh  \vuh \psi]  \cdot \Grad \phi -  \vrh\vuh \otimes \vuh \cdot \mathcal{A}^{\nabla}[\phi \vr_h]  \cdot \Grad \psi} \br
-\ &  \intTO{   (\vrh  \vuh )^{\vartriangle} \cdot  \mathcal{A}^{\nabla}[\phi \vr_h] \, D_t \psi  } +\mathcal{F}_5(\vrh,\vrh \vuh).
\end{align}

Under the hypotheses of Theorem \ref{thm_main}, there holds
\[
\abs{\mathcal{F}_5(\vrh,\vrh \vuh)} \leq C(\Lambda) \left(  h^{(1-\alpha)/2}+h^{(1+\alpha)/2} \right).
\]
\end{Lemma}

\begin{proof}
Let us start by recalling the consistency error of the convective terms \eqref{es-conv} in Appendix \ref{ap_css}, specifically,
\begin{equation}\label{con-err-flux-1}
\begin{aligned}
& E_F(r_h, \phi) = \int_{0}^{T} \intfacesint{ {\bf F}_h^{\alpha}  (r_h ,\vuh ) \cdot \jump{\PiQ \phi} } \dt - \intTO{r_h \vuh  \cdot \Grad \phi },\\
& \mbox{with } \ \abs{E_F(r_h, \phi)} \leq C(\Lambda) \left[ h^{(1-\alpha)/2}+h^{(1+\alpha)/2} \right] \norm{\Grad \phi}_{L^2((0,T)\times\Omega; \R^{d})}.  
\end{aligned}
\end{equation}
Then we have
\begin{align}\label{con-d-f1}
 \int_{0}^{T} \intfacesint{ {\bf F}_h^{\alpha}  (\vrh \vuh ,\vuh ) \cdot \jump{\vvh} } \dt = \intTO{\vrh\vuh \otimes \vuh  : \Grad \vv } + E_F(\vrh \vuh, \vv).
\end{align}

\medskip
Next, we reformulate $\intTO{ D_t  (\vrh  \vuh ) \cdot \vvh }$.  
Thanks to the chain rule
\[
D_t(a_h b_h)  = b_h D_t a_h + a_h^{\vartriangle} D_t b_h,  \quad a_h^{\vartriangle} = a_h(t-\TS), \quad a_h,b_h \in Q_h,
\]
we have
\begin{align*}
&  \intTO{ D_t  (\vrh  \vuh ) \cdot \vvh }   = \intTO{ D_t  (\vrh  \vuh ) \cdot \vv } 
= - \intTO{   (\vrh  \vuh )^{\vartriangle} \cdot D_t \vv }  \\
& = - \intTO{   (\vrh  \vuh \psi)^{\vartriangle} \cdot D_t \left( \mathcal{A}^{\nabla}[\phi \vr_h] \right) } - \intTO{   (\vrh  \vuh )^{\vartriangle} \cdot  \mathcal{A}^{\nabla}[\phi \vr_h] \, D_t \psi  }  \\
& = - \intTO{   (\vrh  \vuh \psi)^{\vartriangle} \cdot  \left( \mathcal{A}^{\nabla}[\phi D_t \vr_h] \right) } - K_1  \\
& =  \intTO{  \mathcal{A}^{\rm div} [ (\vrh  \vuh \psi)^{\vartriangle} ] \,\phi D_t \vr_h } - K_1  =  \intTO{  \PiQ[\mathcal{A}^{\rm div} [ (\vrh  \vuh \psi)^{\vartriangle} ] \,\phi ] \, D_t \vr_h } - K_1 
\end{align*}
with 
\[
K_1 =  \intTO{   (\vrh  \vuh )^{\vartriangle} \cdot  \mathcal{A}^{\nabla}[\phi \vr_h] \, D_t \psi  }.
\]
Note that we have used \eqref{op-lift-IBP} for the last equality. 
Further, we substitute the density equation \eqref{VFV} into the above equality and obtain
\begin{align}\label{con-d-f2}
&\intTO{ D_t  (\vrh  \vuh ) \cdot \vvh } + K_1 =  \intTO{  \PiQ[\mathcal{A}^{\rm div} [ (\vrh  \vuh \psi)^{\vartriangle} ] \,\phi ] \, D_t \vr_h }  \br
%& =   -\intTO{ \phi  \mathcal{A}^{\rm div} [ (\vrh  \vuh \psi)^{\vartriangle} ]  \Divmesh {\bf F}_h^{\alpha} (\vrh,\vuh )  } \br
%& =   -\intTO{ \PiQ\left( \phi  \mathcal{A}^{\rm div} [ (\vrh  \vuh \psi)^{\vartriangle} ] \right) \Divmesh {\bf F}_h^{\alpha} (\vrh,\vuh )  } \br
& =   \int_{0}^{T} \intfacesint{ {\bf F}_h^{\alpha}  (\vrh,\vuh ) \cdot \jump{\PiQ\left( \phi  \mathcal{A}^{\rm div} [ (\vrh  \vuh \psi)^{\vartriangle} ] \right)} } \dt \br
&=  \intTO{   \vrh \vuh \cdot \Grad \left( \phi \mathcal{A}^{\rm div} [ (\vrh  \vuh \psi)^{\vartriangle} ]   \right) } + E_F\Big(\vrh, \phi \mathcal{A}^{\rm div} [ (\vrh  \vuh \psi)^{\vartriangle} ] \Big).
\end{align}
Note that we have used \eqref{con-err-flux-1} for the last equality. 
Subtracting \eqref{con-d-f2} from \eqref{con-d-f1}, we derive
\begin{align}\label{con-d-f3}
&  \intTO{ D_t  (\vrh  \vuh ) \cdot \vvh } - \int_{0}^{T} \intfacesint{ {\bf F}_h^{\alpha}  (\vrh \vuh ,\vuh ) \cdot \jump{\vvh} }\dt \br
&= \intTO{   \vrh \vuh \cdot \Grad \left( \phi \mathcal{A}^{\rm div} [ (\vrh  \vuh \psi)^{\vartriangle} ]   \right) } - \intTO{\vrh\vuh \otimes \vuh  : \Grad \vv }   - K_1 + K_2
\end{align}
with
\begin{align}
K_2 =  E_F\Big(\vrh, \phi \mathcal{A}^{\rm div} [ (\vrh  \vuh \psi)^{\vartriangle} ] \Big) - E_F(\vrh \vuh, \vv). 
\end{align}
Let us reformulate the first two terms on the right hand side of \eqref{con-d-f3}:
\begin{align*}
&\intTO{   \vrh \vuh \cdot \Grad \left( \phi \mathcal{A}^{\rm div} [ (\vrh  \vuh \psi)^{\vartriangle} ]   \right) } 
= \intTO{   \vrh \vuh \cdot \Grad \left( \phi \mathcal{A}^{\rm div} [\vrh  \vuh \psi]   \right) } - K_3 \\
= \ & \intTO{ \vrh \vuh \phi \cdot \Grad \left(  \mathcal{A}^{\rm div} [\vrh  \vuh \psi]   \right)} + \intTO{\vrh \vuh \mathcal{A}^{\rm div} [\vrh  \vuh \psi]  \cdot \Grad \phi  } - K_3
\end{align*}
and 
\begin{align*}
&   \intTO{\vrh\vuh \otimes \vuh  : \Grad (\psi \mathcal{A}^{\nabla}[\phi \vr_h])} \\
=\ &  \intTO{\vrh\vuh \otimes \vuh  \psi  : \Grad ( \mathcal{A}^{\nabla}[\phi \vr_h])} + \intTO{ \Grad \psi \cdot \vrh\vuh \otimes \vuh \cdot  \mathcal{A}^{\nabla}[\phi \vr_h] }
\end{align*}
with
\begin{align*}
K_3 = \TS  \intTO{\vrh \vuh \cdot  \Grad \left( \phi \mathcal{A}^{\rm div} [D_t(\vrh  \vuh \psi)]   \right)}.
\end{align*}
The above estimates finally yield
\begin{align}\label{con-d-f4}
&  \int_{0}^{T}\intO{ D_t  (\vrh  \vuh ) \cdot \vvh }\dt - \int_{0}^{T} \intfacesint{ {\bf F}_h^{\alpha}  (\vrh \vuh ,\vuh ) \cdot \jump{\vvh} }\dt \br
=\ & \intTOB{    \vrh \vuh \phi \cdot \Grad \left(  \mathcal{A}^{\rm div} [\vrh  \vuh \psi]   \right) -  \vrh\vuh \otimes \vuh  \psi  :\Grad ( \mathcal{A}^{\nabla}[\phi \vr_h])} \br
+\ & \intTOB{ \vrh \vuh \mathcal{A}^{\rm div} [\vrh  \vuh \psi]  \cdot \Grad \phi -  \Grad \psi \cdot \vrh\vuh \otimes \vuh  \cdot \mathcal{A}^{\nabla}[\phi \vr_h] } \br
-\ &  \intTO{   (\vrh  \vuh )^{\vartriangle} \cdot  \mathcal{A}^{\nabla}[\phi \vr_h] \, D_t \psi  } +\mathcal{F}_5(\vrh,\vrh \vuh)
\end{align}
with 
\[
\mathcal{F}_5(\vrh,\vrh \vuh) = K_2-K_3= E_F\Big(\vrh, \phi \mathcal{A}^{\rm div} [ (\vrh  \vuh \psi)^{\vartriangle} ] \Big) - E_F(\vrh \vuh, \vv)-K_3.
\]
Now, recalling that 
\begin{align*}
& \norm{\Grad \vv}_{L^2((0,T)\times\Omega; \R^{d\times d})} + \norm{\Grad \left(\phi \mathcal{A}^{\rm div} [ (\vrh  \vuh \psi)^{\vartriangle} ]\right) }_{L^2((0,T)\times\Omega; \R^{d})} \aleq 1, 
\end{align*}
we may control $\mathcal{F}_5(\vrh,\vrh \vuh)$ as follows
\begin{align*}
& \abs{E_F\Big(\vrh, \phi \mathcal{A}^{\rm div} [ (\vrh  \vuh \psi)^{\vartriangle} ] \Big) - E_F(\vrh \vuh, \vv)} \aleq    h^{(1-\alpha)/2}+h^{(1+\alpha)/2},\\
& \abs{K_3} \aleq \TS +  \TS \norm{D_t (\vrh  \vuh)}_{L^{1}([0,T]\times\Omega;\R^d)} \aleq \TS^{1/2} \approx h^{1/2},
\end{align*}
which finishes the proof of Lemma~\ref{lem-f}.
\end{proof}

\begin{Lemma}\label{lem-lim}
%Let the conditions in Lemma~\ref{lem-vis} hold. 
Under the hypotheses of Theorem \ref{thm_main}, with regard to the
the convergences \eqref{weak-con-1}-\eqref{weak-con-4}, \eqref{con-m},  \eqref{con-mmd} and \eqref{con-p},  there holds
\begin{align}\label{lim-flux-1}
& \lim_{h\to 0}  \intTO{  \Curlh \vuh \cdot \left(\Grad \psi \times  \mathcal{A}^{\nabla}[\phi \vrh]\right)  } = \intTO{  \Curl \vu \cdot \left(\Grad \psi \times  \mathcal{A}^{\nabla}[\phi \vr]\right)  }, \\
& \lim_{h\to 0} \intTO{  \Divh \vuh  \Grad \psi \cdot  \mathcal{A}^{\nabla}[\phi \vr_h] }  = \intTO{  \Div \vu  \Grad \psi \cdot  \mathcal{A}^{\nabla}[\phi \vr] }, \\
&\lim_{h\to 0} \intTO{ p_h \Grad \psi \cdot \Big(  \mathcal{A}^{\nabla}[\phi \vr_h]) \Big)  } = \intTO{ p \Grad \psi \cdot \Big(  \mathcal{A}^{\nabla}[\phi \vr]) \Big)  }, \\
& \lim_{h\to 0} \intTO{   (\vrh  \vuh )^{\vartriangle} \cdot  \mathcal{A}^{\nabla}[\phi \vr_h] \, D_t \psi  } =  \intTO{   \vr  \vu \cdot  \mathcal{A}^{\nabla}[\phi \vr] \, \pd_t \psi  },\\
& \lim_{h\to 0}\intTOB{ \vrh \vuh \mathcal{A}^{\rm div} [\vrh  \vuh \psi]  \cdot \Grad \phi  -  \Grad \psi \cdot \vrh\vuh \otimes \vuh \cdot  \mathcal{A}^{\nabla}[\phi \vr_h] } \br
& =  \intTOB{ \vr \vu \mathcal{A}^{\rm div} [\vr  \vu \psi]  \cdot \Grad \phi -  \Grad \psi \cdot \vr\vu \otimes \vu  \cdot \mathcal{A}^{\nabla}[\phi \vr] }, \\  \label{lim-flux-3}
&\lim_{h\to 0} \intTO{\vrh \vc{g} \cdot  \psi \mathcal{A}^{\nabla}[\phi \vr_h]} =  \intTO{\vr \vc{g} \cdot  \psi \mathcal{A}^{\nabla}[\phi \vr]}
\end{align}
and
\begin{align}\label{lim-flux-2}
&\lim_{h\to 0} \left(\intTOB{   \phi \vrh \vuh  \cdot \Grad \left(  \mathcal{A}^{\rm div} [\vrh  \vuh \psi]   \right) -  \psi \vrh\vuh \otimes \vuh   : \Grad ( \mathcal{A}^{\nabla}[\phi \vr_h])}  \right)\br
&= \intTOB{ \phi \vr \vu  \cdot \Grad \left(  \mathcal{A}^{\rm div} [\vr  \vu \psi]   \right) -  \psi \vr\vu \otimes \vu   : \Grad ( \mathcal{A}^{\nabla}[\phi \vr]) }.
\end{align}
\end{Lemma}

\begin{proof}
{Thanks to the convergence of $\widetilde{\vr}_h$, 
$\widetilde{\vrh\vuh}$ stated in \eqref{contr}, \eqref{contm},  and the
smoothing properties of the operators $\mathcal{A}^{\rm div}$, 
$\mathcal{A}^{\nabla}$, we have} 
\begin{align}
&\mathcal{A}^{\nabla}[\phi \widetilde{\vrh}]  \to  \mathcal{A}^{\nabla}[\phi \vr] \quad   \mbox{ in } \ C_{weak}([0,T]; L^q(\Omega;\R^d)) \ \mbox{for any}\ 1 \leq q  < \infty, \\
&\mathcal{A}^{\nabla}[\phi \widetilde{\vrh \vuh}]  \to  \mathcal{A}^{\nabla}[\phi \vr \vu]  \quad  \mbox{ in } \ C_{weak}([0,T]; L^q(\Omega;\R^{d\times d})) \ \mbox{for any}\ 1 \leq q  < \infty. 
\end{align} 
{This, together with \eqref{prox1}, \eqref{prox2}, 
yields \eqref{lim-flux-1} -- \eqref{lim-flux-3}. } 

It remains to show \eqref{lim-flux-2}.
First, we reformulate the left-hand-side term with
\begin{align*}
&\intTOB{   \phi \vrh \vuh  \cdot \Grad \left(  \mathcal{A}^{\rm div} [\vrh  \vuh \psi]   \right) -  \psi \vrh\vuh \otimes \vuh   : \Grad ( \mathcal{A}^{\nabla}[\phi \vr_h])} 
= \intTO{   \vuh  \cdot \mathcal{H}_h} 
\end{align*}
and
\[
\mathcal{H}_h := \phi \vrh \Grad \left(  \mathcal{A}^{\rm div} [\vrh  \vuh \psi]   \right) -     \Grad ( \mathcal{A}^{\nabla}[\phi \vr_h]) \cdot \psi \vrh\vuh \in L^{\infty}([0,T] \times \Omega; \R^d).
\]
Note that $\mathcal{H}_h$ has exactly the form required in  Lemma \ref{lem-divcurl-1} with $B_n = \phi \vrh$ and $w_n = \psi \vrh \vuh$.
The only problem with applying Lemma \ref{lem-divcurl-1} is due to their low regularity in time. To solve this problem, let us introduce a continuous-in-time version of $\mathcal{H}_h$, denoted by $\mathcal{H}_h^L $:
\begin{align*}
\mathcal{H}_h^L := \phi \widetilde{\vrh} \Grad \left(  \mathcal{A}^{\rm div} [\widetilde{\vrh \vuh} \psi]   \right) -     \Grad ( \mathcal{A}^{\nabla}[\phi \widetilde{\vrh}]) \cdot \psi \widetilde{\vrh\vuh}.
\end{align*}
The error between $\mathcal{H}_h^L $ and $\mathcal{H}_h$ can be controlled as follows
\begin{align*}
&\mathcal{H}_h^L - \mathcal{H}_h \\
&= \phi \widetilde{\vrh} \Grad \left(  \mathcal{A}^{\rm div} [\widetilde{\vrh \vuh} \psi]   \right) - \phi \widetilde{\vrh} \Grad \left(  \mathcal{A}^{\rm div} [\vrh \vuh \psi]   \right) + \phi \widetilde{\vrh} \Grad \left(  \mathcal{A}^{\rm div} [\vrh \vuh \psi]   \right) - \phi \vrh \Grad \left(  \mathcal{A}^{\rm div} [\vrh  \vuh \psi]   \right)   \\
& -  \left(   \Grad ( \mathcal{A}^{\nabla}[\phi \widetilde{\vrh}]) \cdot \psi \widetilde{\vrh\vuh}  - \Grad ( \mathcal{A}^{\nabla}[\phi \vrh]) \cdot \psi \widetilde{\vrh\vuh} + \Grad ( \mathcal{A}^{\nabla}[\phi \vrh]) \cdot \psi \widetilde{\vrh\vuh} -     \Grad ( \mathcal{A}^{\nabla}[\phi \vr_h]) \cdot \psi \vrh\vuh \right),
\end{align*}
which gives 
\begin{align*}
\norm{\mathcal{H}_h^L - \mathcal{H}_h}_{L^2((0,T)\times\Omega; \R^d)} \aleq \TS \norm{D_t (\vrh, \vrh\vuh)}_{L^2((0,T)\times\Omega;\R^{d+1})} \aleq \TS^{1/2}.
\end{align*}
%Consequently, it holds that (up to a subsequence)
%\begin{align*}
%\abs{\mathcal{H}_h^L - \mathcal{H}_h} \to 0 \ \mbox{a.e. in}\  (0,T)\times\Omega.
%\end{align*}
Consequently, we have
\begin{align*}
\lim_{h\to 0}\intTO{\vuh \mathcal{H}_h} = \lim_{h\to 0}\intTO{\vuh \mathcal{H}_h^L}.
\end{align*}
%On the other hand, 
%thanks to convergences result \eqref{con-m}, i.e.\  
%\begin{align*}
%\vrh \vuh  \to \vr \vu  \ \mbox{ weakly in } L^q([0,T]\times \Omega; \R^d) \ \mbox{for any}\ 1 \leq q  < \infty
%\end{align*}
%and 
%\begin{align*}
%\abs{\widetilde{\vrh \vuh} -  \vrh \vuh} \leq \TS D_t (\vrh \vuh), \quad  \norm{D_t (\vrh \vuh)}_{L^2((0,T)\times\Omega)} \aleq \TS^{-1/2},
%\end{align*}
%we also have 
%\begin{align*}
%\widetilde{\vrh \vuh}  \to \vr \vu  \ \mbox{ weakly in } C_{\rm weak, loc}([0,T]; L^q(\Omega; \R^d)) \ \mbox{for any}\ 1 \leq q  < \infty.
%\end{align*}
{With the convergence results \eqref{contr}, \eqref{contm} and \eqref{prox1}, \eqref{prox2} at hand,  
we shall apply Lemma \ref{lem-divcurl-1} with $B_n = \phi \widetilde{\vrh}$ and $w_n = \psi \widetilde{\vrh \vuh}$ 
\emph{pointwise} for any fixed $t \in (0,T)$} to conclude 
\begin{align}
\mathcal{H}_h^L \to  \mathcal{H} \ \mbox{ weakly in } L^{q}((0,T)\times\Omega; \R^d) \ \mbox{for any}\ 1 \leq q  < \infty
\end{align}
with
\[
 \mathcal{H} :=  \phi \vr \Grad \left(  \mathcal{A}^{\rm div} [\vr \vu \psi]   \right) -     \Grad ( \mathcal{A}^{\nabla}[\phi \vr]) \cdot \psi \vr\vu.
\]
Now, applying the strong convergence of $\vuh$, i.e.\ \eqref{con-u}, we obtain
\begin{align}
\lim_{h\to 0}\intTO{\vuh\mathcal{H}_h^L} =  \intTO{\vu\mathcal{H}},
\end{align}
and finish the proof of Lemma~\ref{lem-lim}.
%Further, let us introduce the mollifier
%\begin{align*}
%R^{\delta}(x) = \delta^{-d} R\left( \frac{|x|}{\delta}\right), \quad 
%R(x) = \begin{cases}
%\kappa e^{-1/(1-|x|^2)}, & |x| < 1,\\
%0, & |x|\geq 1,
%\end{cases}
%\end{align*}
%where $\kappa$ is such that $\int_{\R^d} R \dx = 1$ and $\delta > 0$ is small. 
%Let us define the subdomain $\Omega_{\delta} = \{x \in \Omega \,|\, \mbox{dist}(x,\pd \Omega) > \delta \}$.
%Then, we have
%\begin{align*}
%\intTOI{\vuh \mathcal{H}_h^L} = \intTOI{(\vuh - R_{\delta} * \vuh) \mathcal{H}_h^L}  + \intTOI{R_{\delta} * \vuh \, \mathcal{H}_h^L}. 
%\end{align*}
%Since
%\begin{align*}
%\abs{\intTOI{(\vuh - R_{\delta} * \vuh) \mathcal{H}_h^L}} \aleq 
%\end{align*}
\end{proof}

\subsection{Proof of Lions' identity (Lemma~\ref{lem})}
 
{With all preliminary results at hand, we are ready to show Lions’ identity claimed in \ Lemma~\ref{lem}.}
%{\cgrey We start by applying the integration-by-parts formula \eqref{InByPa} on the momentum equation \eqref{VFV_M1} obtaining
%\begin{align}\label{eq-lem}
%&\intTO{ D_t  (\vrh  \vuh ) \cdot \vvh } - \int_{0}^{T} \intfacesint{ {\bf F}_h^{\alpha}  (\vrh \vuh ,\vuh ) \cdot \jump{\vvh} }\dt -\intTO{\vrh \vc{g} \cdot  \vvh}  \br
%& \quad =  -\intTO{ (\bS_h -p_h \bI ) : \Gradh \vvh }.
%\end{align}
%
%Next, using} 
Using the error estimates obtained in Lemmas~\ref{lem-bS} -- \ref{lem-f} we rewrite \eqref{eq-lem} in the form
\begin{align}\label{eq-lem-1}
&  \intTOB{ \psi \phi \vr_h p_h }  - (2\mu+\lambda) \intTOB{ \psi \phi \Divh \vc{u}_h \vr_h } = \sum_{i=1}^7 J_i^h,
\end{align}
with
\begin{align*}
& J_1^h = - \intTOB{ p_h \Grad \psi \cdot \Big(  \mathcal{A}^{\nabla}[\phi \vr_h]) \Big)  }, \\
& J_2^h = \mu \intTOB{  \Curlh \vc{u}_h \cdot \left(\Grad \psi \times  \mathcal{A}^{\nabla}[\phi \vr_h]\right) }  +(2\mu+\lambda) \intTOB{  \Divh \vc{u}_h  \Grad \psi \cdot  \mathcal{A}^{\nabla}[\phi \vr_h] },\\
& J_3^h = \intTOB{    \vrh \vuh \phi \cdot \Grad \left(  \mathcal{A}^{\rm div} [\vrh  \vuh \psi]   \right) -  \vrh\vuh \otimes \vuh  \psi  : \Grad ( \mathcal{A}^{\nabla}[\phi \vr_h])}, \br
&J_4^h = \intTOB{ \vrh \vuh \mathcal{A}^{\rm div} [\vrh  \vuh \psi]  \cdot \Grad \phi -  \Grad \psi \cdot \vrh\vuh \otimes \vuh  \cdot\mathcal{A}^{\nabla}[\phi \vr_h]  }, \br
& J_5^h = -\intTO{   (\vrh  \vuh )^{\vartriangle} \cdot  \mathcal{A}^{\nabla}[\phi \vr_h] \, D_t \psi  }, \\
& J_6^h =   -\intTO{\vrh \vc{g} \cdot  \psi \mathcal{A}^{\nabla}[\phi \vr_h]}, \\
& J_7^h = \mathcal{F}_3(\vr_h) - \mathcal{F}_4(\vrh) + \mathcal{F}_5(\vrh,\vrh \vuh).
\end{align*}

\medskip
\noindent Using Lemma~\ref{lem-lim}, we may pass to the limit in \eqref{eq-lem} obtaining
\begin{align*}
&J_1^h \to J_1 := - \int_0^T \intO{ p \Grad \psi \cdot \Big(  \mathcal{A}^{\nabla}[\phi \vr]) \Big)  }\dt, \\
& J_2^h \to  J_2 = \mu \int_0^{T}\int_{\Omega}  \Curl \vc{u}  \Grad \psi \times  \mathcal{A}^{\nabla}[\phi \vr]  \dxdt  +(2\mu+\lambda) \int_0^{T}\int_{\Omega}  \Div \vc{u}  \Grad \psi \cdot  \mathcal{A}^{\nabla}[\phi \vr]  \dxdt,\\
& J_3^h \to J_3 = \intTOB{ \phi \vr \vu  \cdot \Grad \left(  \mathcal{A}^{\rm div} [\vr  \vu \psi]   \right) -  \psi \vr\vu \otimes \vu   : \Grad ( \mathcal{A}^{\nabla}[\phi \vr]) },  \\
& J_4^h \to J_4 =  \intTOB{ \vr \vu \mathcal{A}^{\rm div} [\vr  \vu \psi]  \cdot \Grad \phi -  \Grad \psi \cdot \vr\vu \otimes \vu  \cdot \mathcal{A}^{\nabla}[\phi \vr]  },\\
& J_5^h \to J_5 =  -\intTO{   \vr  \vu \cdot  \mathcal{A}^{\nabla}[\phi \vr] \, \pd_t \psi  },\\
& J_6^h \to J_6 =  -\intTO{\vr \vc{g} \cdot  \psi \mathcal{A}^{\nabla}[\phi \vr]} ,\\
& J_7^h \to 0.
\end{align*}

Finally, since we already know that the limit $(\vr, \vu, \vt)$ satisfy the equation of continuity 
\eqref{w3} as well as the momentum equation \eqref{w5}, we may repeat the above procedure by considering 
$\psi \mathcal{A}^\nabla [\phi \vr]$. After a straightforward manipulation, for which we refer to \cite[Chapter 6]{Fe:2004}, 
we deduce the desired conclusion
\begin{align}\label{eq-lem-2}
\sum_{i=1}^6 J_i =   \intTO{ \psi \phi \vr p }  - (2\mu+\lambda) \intTO{ \psi \phi \vr \Div \vc{u}}. 
\end{align}
Combining \eqref{eq-lem-1} and \eqref{eq-lem-2} finishes the proof of Lemma \ref{lem}.

\subsection{Proof of Theorem~\ref{THM:density}}\label{subsec:con-d}
To begin, let us recall the renormalized continuity equation in \cite[Lemma 8.3]{FeLMMiSh} or \cite[Lemma A.1]{LSY:2025NM}.

\begin{Lemma}[Renormalized continuity equation {\cite[Lemma 8.3]{FeLMMiSh} or \cite[Lemma A.1]{LSY:2025NM}}]\label{lem_r2}
Let $(\vrh ,\vuh)$ satisfy \eqref{VFV_D}. Then for any $\phi_h \in Q_h$ and any function $B\in C^1(\R)$
we have
\begin{align} 
& \intO{ {\rm D}_t B(\vrh) \phi_h } - \intfacesint{ \Up[B(\vrh), \vuh] \cdot \jump{\phi_h} } + \intO{ \phi_h \; (\vrh B'(\vrh) - B(\vrh) ) \; \Divh \vuh }
\nonumber \\
& = - \frac{1}{\Delta t}\intO{\phi_h \EB{\vrh^\triangleleft |\vrh} } - h^\alpha \intfacesint{ \jump{\vrh} \jump{B'(\vrh) \phi_h}} \br
& -\intfacesint{ |\avs{\vuh } \cdot \vn | \phi_h^{\rm down} \EB{\vrh^{\rm up}|\vrh^{\rm down}} }, \label{renormalized_density}
\end{align}
where $E_f(v_1|v_2) = f(v_1) - f'(v_2)(v_1-v_2) - f(v_2),  f \in C^1(\R)$.
\end{Lemma}

Finally, we are ready to show the strong convergence of density mimicking the arguments of the existence theory.
\begin{proof}[Proof of Theorem~\ref{THM:density}]
We start from the discrete renormalized continuity equation, Lemma~\ref{lem_r2}.
Choosing $B(z)= z \log z$ and $\phi_h \equiv 1$ we obtain for a.a.\ $t\in (0,T)$,
\begin{equation}\label{disc-renorm-rlogr}
\intO{ (\vrh \log(\vrh))(t) }  + \int_{0}^{t}\intO{ \vrh \Divh \vuh } \dt \leq \intO{ (\vrh^0 \log(\vrh^0)) }.
\end{equation}
Passing to the limit $h \to 0$ (up to a subsequence) yields
\begin{equation}\label{limit-renorm-rlogr-ineq}
\intO{ \Ov{\vr \log(\vr)}(t) }  + \int_{0}^{t}\intO{ \Ov{\vr \Div \vu} } \dt \leq \intO{ \vr_0 \log(\vr_0) }
\qquad \mbox{for a.a.\ } t\in(0,T),
\end{equation}
where the overline in the above integral terms denotes their weak limits. We have used \eqref{idata} and \eqref{init-pro} %\eqref{init-conv} together with the uniform bounds \eqref{ap} 
to identify $\intO{ \vrh^0\log \vrh^0} \to \intO{ \vr_0 \log \vr_0}$.

Next, seeing that $(\vr,\vu)$ is a renormalized solution of the continuity equation \eqref{w4} we may
choose $b(z)=z\log z$ to obtain 
\begin{equation}\label{limit-renorm-rlogr-eq}
\intO{ (\vr \log\vr)(t) }  + \int_{0}^{t}\intO{ \vr \Div \vu } \dt = \intO{ \vr_0 \log(\vr_0) } \ \mbox{for a.a.}\ t \in (0,T).
\end{equation}
Subtracting \eqref{limit-renorm-rlogr-eq} from \eqref{limit-renorm-rlogr-ineq} we infer
\begin{equation}\label{defect-ineq}
\intO{ (\Ov{\vr \log(\vr)} - \vr \log\vr)(t) }  + \int_{0}^{t}\intO{ (\Ov{\vr \Div \vu} - \vr \Div \vu) } \dt \leq  0
\qquad \mbox{for a.a.\ } t\in(0,T).
\end{equation}
Note that the first term is non--negative as a consequence of convexity of the function $z \mapsto \log(z)$,
\[
0 \leq \Ov{\vr \log(\vr)} - \vr \log\vr .
\]

Next, Lemma~\ref{lem} (Lions' identity) together with the strong convergence of $\vth$ implies
\begin{equation}\label{defect-flux}
(2\mu+\lambda)\left(\Ov{ \vr\Div \vu} - \vr \Div \vu\right)
= \vt\left(\Ov{\vr^2} - \vr^2\right)
\qquad \mbox{in } \ \mathcal{D}'((0,T)\times\Omega),
\end{equation}
hence
\[
\int_{0}^{t}\intO{ (\Ov{\vr \Div \vu} - \vr \Div \vu) } \dt
= \frac{1}{2\mu+\lambda}\int_{0}^{t}\intO{ \vt\left(\Ov{\vr^2} - \vr^2\right)} \dt \ge 0
\]
for a.a.\ $t\in(0,T)$.
Combining with \eqref{defect-ineq} we conclude that both terms vanish, in particular,
\[
\intO{ (\Ov{\vr \log(\vr)} - \vr \log\vr)(t) } =0 \qquad \mbox{for a.a.\ } t\in(0,T).
\]
By strict convexity of $z\mapsto z\log z$ and boundedness of $\vrh$ both from above and from below, we derive 
\[
\vrh \to \vr \ \mbox{in}\ L^2((0,T) \times \Omega.
\]
This immediately yields the conclusion of Theorem \ref{conv-d}.
\end{proof}

We have completed the proof of Theorem \ref{THM:density}, which also finishes the proof of Theorem \ref{thm_main}.

\section*{Acknowledgments}
This work was partially supported by the Mathematical Research Institute Oberwolfach via the Oberwolfach Research Fellows project in 2025. The authors gratefully acknowledge the hospitality of the institute and its stimulating working atmosphere.

%\bibliography{citace}

\begin{thebibliography}{100}


\bibitem{AbFeNo:2021}
A. Abbatiello, E. Feireisl and A. Novotn\'y. 
\newblock Generalized solution to models of compressible viscous fluids. 
\newblock {\em Discrete Contin. Dyn. Syst.}, {\bf 40}(1):1-28, 2021.

%\bibitem{BLMSY}
%D.~Basari\'c, M.~Luk{\'a}{\v c}ov{\'a}-Medvi{\softd}ov{\'a}, H.~Mizerov\'a, B.~She and Y.~Yuan.
%\newblock Error estimates of a finite volume method for the Navier--Stokes--Fourier system.
%\newblock{\em Math. Comp.}, {\bf 92}:2543-2574, 2023.

%\bibitem{BelFeiOsch}
%P. Bella, E. Feireisl and F. Oschmann. 
%\newblock Rigorous derivation of the Oberbeck-Boussinesq
%approximation revealing unexpected term. 
%\newblock {\em Comm. Math. Phys.}, {\bf 403}(3):1245-1273, 2023.

\bibitem{BUCLGS}
T.~Buckmaster, G.~Cao-Labora and J.~G\'omez-Serrano.
\newblock Smooth self-similar imploding profiles to 3{D} compressible {E}uler.
\newblock {\em Quart. Appl. Math.}, {\bf 81}(3):517--532, 2023.

%\bibitem{CaHoHo}
%W.M. Castillo, Wm. H. Hoover and C.G. Hoover. 
%\newblock Coexisting attractors in Rayleigh-B\'enard flow. 
%\newblock {\em Phys. Rev. E}, {\bf 55}:5546-5550, 1997.


%\bibitem{ChauFei}
%N. Chaudhuri and E. Feireisl. 
%\newblock Navier-Stokes-Fourier system with Dirichlet boundary conditions. 
%\newblock {\em Appl. Anal.}, {\bf 101}(12):4076-4094, 2022.

%\bibitem{CheGli}
%G.-Q. Chen and J. Glimm. 
%\newblock Kolmogorov-type theory of compressible turbulence and inviscid limit of the Navier-Stokes equations in $\R^3$.
%\newblock {\em Phys. D}, {\bf 400}:132138, 10, 2019.


%\bibitem{ChenGli12}
%G.-Q. Chen and J. Glimm. 
%\newblock Kolmogorov’s theory of turbulence and inviscid limit of the Navier-Stokes equations in $\R^3$.
%\newblock {\em Comm. Math. Phys.}, {\bf 310}(1):267-283, 2012.


\bibitem{davidson}
P.A.~Davidson.
\emph{Turbulence. 
An introduction for scientists and engineers.}  Oxford University Press, Oxford, 2015.

\bibitem{DiPena-Lions}
R. J. DiPerna and P. L. Lions. 
\newblock Ordinary differential equations, transport theory and Sobolev spaces. 
\newblock{\em Invent. Math.}, {\bf 98}:511-547, 1989.

\bibitem{GL}
R.~Eymard, T.~Gallou\"et, R.~Herbin and J.~Latch\'e. 
\newblock A convergent finite element-finite volume scheme for the
compressible {S}tokes problem. {II}. The isentropic case.
\newblock {\em Math. Comp.}, {\bf 79}: 649--675, 2010.




%\bibitem{FanFeiHof}
%F. Fanelli, E. Feireisl, and M. Hofmanov\'a. 
%\newblock Ergodic theory for energetically open compressible fluid flows. 
%\newblock {\em Phys. D}, {\bf 423}:Paper No. 132914, 25, 2021.


\bibitem{Fe:2004}
E. Feireisl. 
\newblock{\em Dynamics of viscous compressible fluids.} \newblock %Oxford Lecture Series in Mathematics and its Applications, vol. 26. 
Oxford University Press, Oxford, 2004.






\bibitem{FeKaNo}
E.~Feireisl, T.~Karper and A.~Novotn{\'y}.
\newblock A convergent numerical method for the {N}avier--{S}tokes--{F}ourier
system.
\newblock {\em IMA J. Numer. Anal.}, {\bf 36}(4):1477-1535, 2016.

%\bibitem{FeiLuSun}
%E. Feireisl, Y. Lu, and Y. Sun. 
%\newblock Unconditional stability of equilibria in thermally driven compressible fluids. 
%\newblock {\em Arch. Ration. Mech. Anal.}, {\bf 248}(6):Paper No. 98, 38, 2024.


\bibitem{FeLMMiSh}
E. Feireisl, M. Luk{\'a}{\v c}ov{\'a}-Medvi{\softd}ov{\'a}, {H}. Mizerov{\'a} and B. She. 
\newblock {\em Numerical analysis of compressible fluid flows}. 
\newblock Springer-Verlag, Cham, 2021.


\bibitem{FeLMShYu:2024}
E. Feireisl, M. Luk{\'a}{\v{c}}ov{\'a}-Medvi{\softd}ov{\'a}, B. She and Y. Yuan.
\newblock Convergence of numerical methods for the Navier-Stokes-Fourier system driven by uncertain initial/boundary data.
\newblock{\em Found. Comput. Math.}, {\bf 25}:1507-1559, 2025.



\bibitem{FeLMShYu:2025II}
E. Feireisl, M. Luk{\'a}{\v{c}}ov{\'a}-Medvi{\softd}ov{\'a}, B. She and Y. Yuan.
\newblock 
Temperature-driven turbulence in compressible fluid flows.
Preprint 2026.

\bibitem{Feireisl-Novotny:2017}
E.~Feireisl and A.~Novotn\'y.
\newblock {\em Singular limits in thermodynamics of viscous fluids}.
\newblock Advances in Mathematical Fluid Mechanics. Birkh\"auser/Springer, Cham, 2017.
\newblock Second edition.

\bibitem{FeiNovOpen}
E. Feireisl and A. Novotn\'y. 
\newblock {\em Mathematics of open fluid systems}. 
Birkh\"auser–Verlag, Basel, 2022.

%\bibitem{FeiSwGw}
%E. Feireisl and A. Swierczewska-Gwiazda. 
%\newblock The Rayleigh–B\'enard problem for compressible fluid flows. 
%\newblock {\em Arch. Ration. Mech. Anal.}, {\bf 247}(1):Paper No. 9, 2023.

%\bibitem{GlLaCh}
%J. Glimm, D. Lazarev, and G.-Q. Chen. 
%\newblock Maximum entropy production as a necessary admissibility condition for the fluid Navier-Stokes and Euler equations. 
%\newblock {\em SN Applied Science}, {\bf 2}:pp. 2160, 2020.




\bibitem{HOF1}
D.~Hoff.
\newblock Global solutions of the Navier-Stokes equations for multidimensional compressible flow with discontinuous initial data.
\newblock {\em J. Differential Equations}, {\bf 120}:215-254, 1995.

\bibitem{HOF7}
D.~Hoff.
\newblock Dynamics of singularity surfaces for compressible viscous flows in two space dimensions.
\newblock {\em Commun. Pure Appl. Math.}, {\bf 55}:1365-1407, 2002.

\bibitem{HofSan}
D.~Hoff and M.~M. Santos.
\newblock Lagrangean structure and propagation of singularities in multidimensional compressible flow.
\newblock {\em Arch. Ration. Mech. Anal.}, {\bf 188}(3):509-543, 2008.


%\bibitem{JohSch}
%J. P. John and J. Schumacher. 
%\newblock Compressible turbulent convection in highly stratified adiabatic background. 
%\newblock {\em J. Fluid Mech.}, {\bf 972}:Paper No. R4, 12, 2023.


\bibitem{Karper}
T.~Karper.
\newblock A convergent FEM-DG method for the compressible Navier--Stokes equations.
\newblock {\em Numer. Math.}, {\bf 125}(3):441--510, 2013.




\bibitem{Lions}
P. L. Lions.  
\newblock {\em Mathematical topics in fluid mechanics. In: Compressible Models, vol. 2.} 
\newblock  Oxford University Press, New York, 1998.



\bibitem{LSY:2025NM}
M. Luk\'{a}\v{c}ov\'{a}-Medvid'ov\'{a}, B. She and Y. Yuan. 
\newblock Penalty method for the Navier-Stokes-Fourier system with Dirichlet boundary conditions: Convergence and error estimates. 
\newblock {\em Numer. Math.}, {\bf 157}:1079-1132, 2025. 



\bibitem{MeRaRoSz}
F.~Merle, P.~Rapha\"{e}l, I.~Rodnianski and J.~Szeftel.
\newblock On the implosion of a compressible fluid {I}: smooth self-similar inviscid profiles.
\newblock {\em Ann. of Math. (2)}, {\bf 196}(2):567-778, 2022.

\bibitem{MeRaRoSzbis}
F.~Merle, P.~Rapha\"{e}l, I.~Rodnianski and J.~Szeftel.
\newblock On the implosion of a compressible fluid {II}: singularity formation.
\newblock {\em Ann. of Math. (2)}, {\bf 196}(2):779-889, 2022.

%\bibitem{TiShVe}
%H. Tiwari, L. Sharma and M. K. Verma. 
%\newblock Compressible convective turbulence at very high Rayleigh numbers. 
%\newblock {\em Int. J. Heat Mass Transfer}, {\bf 242}:126821, 2025.

%\bibitem{WangX}
%X. Wang. 
%\newblock Numerical algorithms for stationary statistical properties of dissipative dynamical systems. 
%\newblock {\em Discrete Contin. Dyn. Syst.}, {\bf 36}(8):4599–4618, 2016.




\end{thebibliography}

%\bibliographystyle{plain}

\appendix
\section{Useful lemmas}
In this section, we recall the Div-Curl lemma  \cite[Theorem 11.28]{Feireisl-Novotny:2017}.
\begin{Lemma}[{\bf Div-Curl lemma}]\label{div-curl}
Let \( Q \in \R^N \) be an open set.  
Assume 
\[
\vc{f}_n \to \vc{f} \text{ weakly in } L^p(Q;\R^N), \quad
\vc{g}_n \to \vc{g} \text{ weakly in } L^q(Q;\R^N), 
\]
where 
\[
\frac1p + \frac1q = \frac1r < 1.
\]
In addition, let
\begin{align}
& \mbox{\rm DIV } \vc{f}_n \ \mbox{ precompact in }  W^{-1,s}(Q),     \\
& \mbox{\rm CURL } \vc{g}_n \ \mbox{ precompact in }  W^{-1,s}(Q; \R^{N\times N})  
\end{align}
for a certain \( s> 1\).
Then
\[
\vc{f}_n \cdot \vc{g}_n \to \vc{f} \cdot \vc{g} \text{ weakly in } L^r(Q;\R^N). 
\]
\end{Lemma}

\begin{Proposition}\label{col-div-curl}
Let $\{r_n, v_n\}_{n=1}^{\infty}$ satisfy
\begin{align*}
r_n \to r \text{ weakly in } L^2((0,T)\times \Omega), \quad
v_n \to v \text{ weakly in } L^q((0,T)\times \Omega), 
\end{align*}
where $q>2$, 
%\[
%\frac12 + \frac1q = \frac1m < 1, \, q>2
%\]
and 
\begin{align} 
&\partial_t r_n =  h_n^1 + h_n^2,  \quad
 h_n^1\in_b L^1((0,T)\times \Omega), \quad h_n^2 \in L^2(0,T;W^{-1,2}(\Omega)), \\ 
&\Grad v_n = D_n^1 + D_n^2, \quad   
D_n^1 \in_b L^1((0,T)\times \Omega), \quad
{D_n^2 \to 0 \mbox{ in } W^{-1,2}((0,T) \times \Omega; \R^d)).  } 
\end{align}

Then it holds %for a  $m = \frac{2q}{2+q}$ that
\[
r_h v_h \to r v \text{ weakly in } L^{\frac{2q}{2+q}}((0,T)\times \Omega).
\]
\end{Proposition}

\begin{proof}
The proof is the same as \cite[Lemma 8.1]{AbFeNo:2021}. For completeness, we present the main idea.
  Let $Q = (0,T)\times \Omega$ with $N=d+1$, which implies $\mbox{\rm DIV} = \pd_t + \Div$. 
Applying the Sobolev embedding theorem we get $L^1(Q) \hookrightarrow\hookrightarrow W^{-1,s}(Q)$ {for some $s > 1$}.  
{Thus a direct application of the Div-Curl Lemma \ref{div-curl} with}
\[
\vc{f} = [r_n, g_n], \quad \vc{g} = [v_n,\vc{0}], \quad p = 2, \  q > 2 
\]
yields the desired conclusion
\[
r_h v_h \to r v \text{ weakly in } L^{\frac{2q}{2+q}}(Q).
\]
\end{proof}

\section{Properties of discrete spatial differential operators}\label{sec-diffop}

In this section, we establish the intrinsic properties of discrete spatial differential operators.  
To begin, let us introduce another projection operator $\PiWi: L^1(Q) \to \Whi$ defined by%{\cgrey and $\PiFi$, in addition to $\PiQ$ in \eqref{proj-avg},}
%\begin{subequations}\label{proj}
\begin{align}\label{projW}
%&\PiQ: L^1(Q) \to Q_h, \quad
%\PiQ  \phi (x) = \sumK  \frac{ \mathds{1}_{K}(x)}{|K|} \int_K \phi \dx,\\
%& \PiWi: %Q_h
% L^1(Q) \to \Whi, \quad
\PiWi \phi (x) = \sumSi \frac{\mathds{1}_{D_\sigma}(x)}{|D_\sigma|} \int_{D_\sigma} \phi \dx, \quad 
%\\& \PiFi:  %W^{1,1}
%W^{1,1}(Q) \to \Whi, \quad
%\PiFi \phi(x) = \sumSi \frac{\mathds{1}_{D_\sigma}(x)}{|\sigma|} \int_{\sigma} \phi \, \ds,\\
%&\PiF \vv = \left(\PiF^{(1)}v_1, \cdots, \PiF^{(d)}v_d \right),\quad 
\PiW \vv = \left(\PiW^{(1)}v_1, \cdots, \PiW^{(d)}v_d \right), \quad
\vv = (v_1,\dots, v_d).
\end{align}
%\end{subequations}
%where $\Whi$ represents the space of piecewise constant functions on the $\ith$ dual grid $\meshdi$, and $ \mathds{1}_{K}(x)$ means the characteristic function on $K$. 
Further, we define two differential operators
\begin{align}
%& \PiWi: %Q_h
% L^1(Q) \to \Whi, \quad
%\PiWi \phi (x) = \sumSi \frac{\mathds{1}_{D_\sigma}(x)}{|D_\sigma|} \int_{D_\sigma} \phi \dx,\\
&\pdmeshi: \Whi \to Q_h, \ \pdmeshi r_h(x) = \sumK \left(\pdmeshi r_h \right)_K \mathds{1}_K{(x)} , \
(\pdmeshi r_h)_K := \frac{r_h|_{\sigma_{K,i+}} - r_h|_{\sigma_{K,i-}}}{h}, \\
& \Curlh:  Q_h^d \to Q_h^d, \  \Curlh \vvh(x) = \sumK \left( \Curlh \vvh\right)_K \mathds{1}_K{(x)}, \ (\Curlh \vvh)_K =
\sum_{\sigma\in \facesK}   \vc{n} \times \frac{\avs{\vvh}}{h}. 
\end{align}

We are now ready to establish the relationships among differential operators: $\Gradh, \Divh, \Curlh$ and $\Gradd, \Laph$.

\begin{Lemma}[{\bf Identity}]\label{lem-op-comp}
For any $\vwh \in Q_h^d$ with any extension, there hold
\begin{align}\label{op-compat-h-1}
& \Gradh \Divh \vwh=  \Curlh \Curlh \vwh + \Divh \Gradh \vwh, \  
\Divh \Gradh^T \vwh = \Gradh \Divh \vwh,
\ K \in \mesh , \ K\cap \facesext = \emptyset. 
\end{align}
Additionally,  let us consider two different extensions 
\begin{itemize}
\item[i)] $\vwh$ is arbitrarily extended by two ghost cells outward from the boundary;

\item[ii)] $\vwh, \Gradh \vwh$ are extended with $\jump{\vwh}|_{\facesext} = 0, \jump{\Gradh \vwh}\cdot \vc{n}|_{\facesext} = 0$, respectively. 
\end{itemize} 
Then there hold
\begin{align}\label{op-compat-h-2}
& \Gradh \Divh \vwh=  \Curlh \Curlh \vwh + \Divh \Gradh \vwh, \
\Divh \Gradh^T \vwh = \Gradh \Divh \vwh,
\ K \in \mesh , \ K\cap \facesext \neq \emptyset. 
\end{align}
\end{Lemma}

\begin{proof}
%{\cgrey 
%Let us recall the definition of $\Whi$, the space of piecewise constants on the $\ith$ dual grid $\meshdi$. Then we shall introduce the following projection operator $\PiWi$ and differential operator $\pdmeshi$
%\begin{align*}
%& \PiWi: %Q_h
% L^1(Q) \to \Whi, \quad
%\PiWi \phi (x) = \sumSi \frac{\mathds{1}_{D_\sigma}(x)}{|D_\sigma|} \int_{D_\sigma} \phi \dx,\\
%&\pdmeshi: \Whi \to Q_h, \quad \pdmeshi r_h(x) = \sumK \left(\pdmeshi r_h \right)_K \mathds{1}_K{(x)} , \quad
%(\pdmeshi r_h)_K := \frac{r_h|_{\sigma_{K,i+}} - r_h|_{\sigma_{K,i-}}}{h}. 
%\end{align*}
%%with $\bWh =\Wh^{(1)} \times \cdots \times \Wh^{(d)}$,  $\PiF \vv = \left(\PiF^{(1)}v_1, \cdots, \PiF^{(d)}v_d \right),$ and $\PiW \vv = \left(\PiW^{(1)}v_1, \cdots, \PiW^{(d)}v_d \right)$, where $\vv = (v_1,\dots, v_d)$.
%Then we can rewrite $\Gradh, \Divh, \Curlh$ by means of $\PiWi, \pdmeshi$ as follows
%}

For the sake of simplicity, we prove \eqref{op-compat-h-1} and \eqref{op-compat-h-2} for the case $d=3$, and the case $d=2$ can be handled in the same manner. 

We start by expressing $\Gradh, \Divh, \Curlh$ in terms of $\PiWi$ and $\pdmeshi$:
 \begin{align} 
& \Gradh  = \left( \pdmesh^{(1)}\PiW^{(1)}, \dots,  \pdmesh^{(3)}\PiW^{(3)}\right),\quad
\Divh \vwh = \sum_{i=1}^3 \pdmeshi\PiWi  w^i_{h}, %\quad
%\Laph r_h = \Divmesh \Gradd r_h,
\\ 
& \Curlh \vwh =\Big(
\pdmesh^{(2)}\PiW^{(2)} w^3_{h} - \pdmesh^{(3)}\PiW^{(3)} w^3_{h},\
\pdmesh^{(3)}\PiW^{(3)} w^1_{h} - \pdmesh^{(1)}\PiW^{(1)} w^3_{h}, \
\pdmesh^{(1)}\PiW^{(1)} w^2_{h} - \pdmesh^{(2)}\PiW^{(2)} w^1_{h}
\Big),
%\Curlh \vwh =
%\begin{cases}
%\pdmesh^{(1)}\PiW^{(1)} w^2_{h} - \pdmesh^{(2)}\PiW^{(2)} w^1_{h}, & d=2, \\
%\Big(
%\pdmesh^{(2)}\PiW^{(2)} w^3_{h} - \pdmesh^{(3)}\PiW^{(3)} w^3_{h},\
%\pdmesh^{(3)}\PiW^{(3)} w^1_{h} - \pdmesh^{(1)}\PiW^{(1)} w^3_{h}, & \\
%\hspace{4.5cm }\pdmesh^{(1)}\PiW^{(1)} w^2_{h} - \pdmesh^{(2)}\PiW^{(2)} w^1_{h}
%\Big), & d = 3,
%\end{cases} 
\end{align}
where $\vwh = (w^1_{h}, w^2_{h},w^3_{h})$. 
Further,  it holds
\begin{align}\label{op-commut}
\pdmeshi\PiWi \pdmesh^{(j)}\PiW^{(j)} =  \pdmesh^{(j)}\PiW^{(j)} \pdmeshi\PiWi,
\end{align}
for any inner element $K\in \mesh, K\cap \facesext = \emptyset$ with any $i,j$; or for any boundary element $K\in \mesh, K\cap \facesext \neq\emptyset$ together with $i \neq j$ or $i = j =1,2$.

\medskip

Therefore, direct calculations yield \eqref{op-compat-h-1}.
It remains to prove \eqref{op-compat-h-2}  for the second extension, whereas the first extension can be handled in the same way as \eqref{op-compat-h-1}. 
Without loss of  generality, let us consider the upper boundary $x_3 = 1$ and denote the boundary cell by $K, K\cap \facesext \neq\emptyset$. Its external (resp.\ internal) neighbour in $x_3$-direction is denoted by $N$ (resp.\ $S$), 
see Figure \ref{figpib-1}.
\begin{figure}[hbt]
\centering
\begin{subfigure}{0.45\textwidth}
%\raggedright
\begin{tikzpicture}[scale=1.]
\draw[-,very thick](0,0)--(6,0)--(6,6)--(0,6)--(0,0);
\draw[-,very thick](0,2)--(6,2);
\draw[-,very thick,blue = 90!](-1,4)--(7,4);
\path node at (7,4.3) {$\cblue x_3 = 1$};
\draw[-,very thick](2,0)--(2,6);
\draw[-,very thick](4,0)--(4,6);

%\draw[-,fill=blue!20,very thick, red=90!, pattern=north east lines, pattern color=red!30](1,1)--(1,5)--(5,5)--(5,1)--(1,1);
%\fill (1,3) circle (2pt) node[below left]{$W$};
%\fill (5,3) circle (2pt) node[below right]{$E$};
\fill (3,1) circle (2pt) node[below right]{$S$};
\fill (3,5) circle (2pt) node[above right]{$N$};
\fill (3,3) circle (2pt) node[below right]{$K$};

%\fill (1,1) circle (2pt) node[below left]{$SW$};
%\fill (5,1) circle (2pt) node[below right]{$SE$};
%\fill (1,5) circle (2pt) node[above right]{$NW$};
%\fill (5,5) circle (2pt) node[below right]{$NE$};

\draw[->,very thick, red = 90!] (3,3)--(4.5,3);
\path node at (4.7,3.2) {$\cred \ve_i, i = 1,2$};
\draw[->,very thick, red = 90!] (3,3)--(3,4.5);
\path node at (3.2,4.7) {$\cred \ve_3$};
\end{tikzpicture}
\end{subfigure}
\caption{Boundary cell $K$.}\label{figpib-1}
\end{figure}
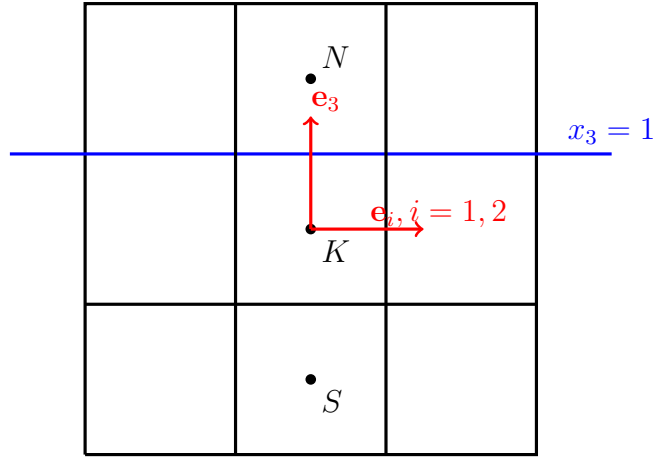

With the extended $\Gradh \vwh$, we have 
\begin{align*}
\left(\Gradh\Divh \vwh \right)^j &= \begin{cases}
\pdmesh^{(j)}\PiW^{(j)} \Divh \vwh & \mbox{ if } j = 1,2, \\
\frac{1}{2h}\left( \left( \Divh \vwh \right)_N - \left( \Divh \vwh \right)_S \right) & \mbox{ if } j = 3, 
\end{cases}\\
\left(\Divh \Gradh \vwh\right)^j &= \sum_{i=1}^2 \pdmesh^{(i)}\PiW^{(i)} \left( \Gradh \vwh \right)^{j,i}  +  \pdmesh^{(3)}\PiW^{(3)} \left( \Gradh \vwh \right)^{j,3} \\
&= \sum_{i=1}^2 \pdmesh^{(i)}\PiW^{(i)} \left( \Gradh \vwh \right)^{j,i}  +  \frac{1}{2h}\left( \left( \Gradh \vwh \right)_N^{j,3} - \left( \Gradh \vwh \right)_S^{j,3} \right), \\
\left(\Divh \Gradh^T \vwh\right)^j &%= \sum_{i=1}^2 \pdmesh^{(i)}\PiW^{(i)} \left( \Gradh \vwh \right)^{i,j}  +  \pdmesh^{(3)}\PiW^{(3)} \left( \Gradh \vwh \right)^{3,j} \\
= \sum_{i=1}^2 \pdmesh^{(i)}\PiW^{(i)} \left( \Gradh \vwh \right)^{i,j}  +  \frac{1}{2h}\left( \left( \Gradh \vwh \right)_N^{3,j} - \left( \Gradh \vwh \right)_S^{3,j} \right),
\end{align*}
and
\begin{align*}
\Curlh\Curlh\vwh = \begin{pmatrix}
\pdmesh^{(2)}\PiW^{(2)} \left( \Curlh\vwh \right)^3 - \frac{1}{2h}\left( \left( \Curlh \vwh \right)_N^{2} - \left( \Curlh \vwh \right)_S^{2} \right)\\
\frac{1}{2h}\left( \left( \Curlh \vwh \right)_N^{1} - \left( \Curlh \vwh \right)_S^{1} \right) - \pdmesh^{(1)}\PiW^{(1)} \left( \Curlh\vwh \right)^3  \\
\pdmesh^{(1)}\PiW^{(1)} \left( \Curlh\vwh \right)^2 - \pdmesh^{(2)}\PiW^{(2)} \left( \Curlh\vwh \right)^1
\end{pmatrix}.
\end{align*}
Hence,
\begin{align*}
&\Gradh \Divh \vwh -  \Curlh \Curlh \vwh - \Divh \Gradh \vwh  \\
=   &\begin{pmatrix}
\pdmesh^{(3)}\PiW^{(3)} \left( \pdmesh^{(1)}\PiW^{(1)} w_h^3 \right) - \frac1{2h}\left( \left( \pdmesh^{(1)}\PiW^{(1)} w_h^3 \right)_N -  \left( \pdmesh^{(1)}\PiW^{(1)} w_h^3 \right)_S \right) \\
\pdmesh^{(3)}\PiW^{(3)} \left( \pdmesh^{(2)}\PiW^{(2)} w_h^3 \right) - \frac1{2h}\left( \left( \pdmesh^{(2)}\PiW^{(2)} w_h^3 \right)_N -  \left( \pdmesh^{(2)}\PiW^{(2)} w_h^3 \right)_S \right) \\
- \sum_{i=1}^2 \left[ \pdmesh^{(3)}\PiW^{(3)} \left( \pdmesh^{(i)}\PiW^{(i)} w_h^i \right) - \frac1{2h}\left( \left( \pdmesh^{(i)}\PiW^{(i)} w_h^i \right)_N -  \left( \pdmesh^{(i)}\PiW^{(i)} w_h^i \right)_S \right) \right]
\end{pmatrix}, \\
&\Divh \Gradh^T \vwh - \Gradh \Divh \vwh  \\
=   &\begin{pmatrix}
\pdmesh^{(3)}\PiW^{(3)} \left( \pdmesh^{(1)}\PiW^{(1)} w_h^3 \right) - \frac1{2h}\left( \left( \pdmesh^{(1)}\PiW^{(1)} w_h^3 \right)_N -  \left( \pdmesh^{(1)}\PiW^{(1)} w_h^3 \right)_S \right) \\
\pdmesh^{(3)}\PiW^{(3)} \left( \pdmesh^{(2)}\PiW^{(2)} w_h^3 \right) - \frac1{2h}\left( \left( \pdmesh^{(2)}\PiW^{(2)} w_h^3 \right)_N -  \left( \pdmesh^{(2)}\PiW^{(2)} w_h^3 \right)_S \right) \\
 \sum_{i=1}^2 \left[ \pdmesh^{(3)}\PiW^{(3)} \left( \pdmesh^{(i)}\PiW^{(i)} w_h^i \right) - \frac1{2h}\left( \left( \pdmesh^{(i)}\PiW^{(i)} w_h^i \right)_N -  \left( \pdmesh^{(i)}\PiW^{(i)} w_h^i \right)_S \right) \right]
\end{pmatrix}.
\end{align*}
Note that the right-hand-side is the residual caused by the extension.
Together with the extension $\jump{\vwh}|_{\facesext} = 0, \jump{\Gradh \vwh}\cdot \vc{n}|_{\facesext} = 0$, we finish the proof.
\end{proof}

\begin{Lemma}[{\bf Integration-by-parts}]\label{lem-IBP}
For any $f_h, v_h \in Q_h, \ \vc{f}_h, \vwh \in Q_h^d$ with any extensions, the following integration-by-parts formulae hold
\begin{itemize}
\item
\begin{subequations}\label{InByPa}
\begin{align}\label{InByPa-1}
 \intO{ \left( \Gradh f_h \cdot \vwh +  f_h \cdot \Divh\vwh\right)}  & =  \intfacesext{\left(\avs{f_h}  \vwh^{in}  + \frac12 f_h^{in}  \jump{\vwh}   \right) \cdot \vn} \br
 & = \intfacesext{\left(  f_h^{in} \avs{\vwh}  + \frac12 \jump{f_h} \vwh^{in}   \right) \cdot \vn}.
 \end{align}
%Further, if $\jump{f_h}|_{\facesext} = \avs{\vwh}\cdot \vc{n}|_{\facesext} = 0$  or $\avs{f_h}|_{\facesext} = \jump{\vwh}\cdot \vc{n}|_{\facesext} = 0$, it holds $ \intO{ \Gradh f_h \cdot \vwh} + \intO{  f_h \cdot \Divh\vwh} = 0$.

\item
\begin{align}\label{InByPa-2}
&\intO{ \Laph f_h \cdot v_h} = \intfacesext{\frac{\jump{f_h}}{h} v_h^{\rm in} }-h  \intfacesint{\Gradd f_h \cdot \Gradd v_h}.  %\quad f_h \in Q_h, \ g \in Q_h,
\end{align}
%if $\avs{v_h}|_{\facesext} = 0$  or $\jump{f_h}|_{\facesext} = 0$.

\item
\begin{align}\label{InByPa-4}
\intO{ \left( \Curlh \vc{f}_h \cdot \vwh - \vc{f}_h \cdot \Curlh  \vwh \right)} 
&= \intfacesext{  \left( \avs{\vc{f}_h} \times \vwh^{in} + \frac12 \vc{f}_h^{in} \times \jump{\vwh} \right) \cdot \vn}   \br
&= \intfacesext{ \left( \vc{f}_h^{in} \times \avs{\vwh} + \frac12 \jump{\vc{f}_h} \times \vwh^{in} \right) \cdot \vn }  . 
\end{align}
\end{subequations}
\end{itemize}
\end{Lemma}

\begin{proof}
Here we only provide the proof of \eqref{InByPa-4}, as equations \eqref{InByPa-1} and \eqref{InByPa-2} can be analyzed similarly. 
With the definition of $\Curlh$, we can reformulate the left-hand-side of \eqref{InByPa-4} as 
\begin{align*}
& \intO{\left( \Curlh \vc{f}_h \cdot \vwh - \vc{f}_h \cdot \Curlh  \vwh \right)}\\
&= \sumK \frac{|K|}{h}  \sumSK \left( \vn \times \avs{\vc{f}_h} \cdot \vw_K -  \vn \times \avs{\vw_h} \cdot \vc{f}_K \right) \\
&= \intfacesext{  \left(\avs{\vc{f}_h} \times \vw_h^{in} -  \avs{\vw_h} \times \vc{f}_h^{in} \right) \cdot \vn} +  \intfacesint{  \left(\avs{\vc{f}_h} \times \jump{\vw_h} -\avs{\vc{f}_h} \times \jump{\vw_h} \right) \cdot \vn} \\
&= \intfacesext{  \left(\avs{\vc{f}_h} \times \vw_h^{in} -  \avs{\vw_h} \times \vc{f}_h^{in} \right) \cdot \vn}  +  \intfacesext{  \vw_h^{in} \times \vc{f}_h^{in} \cdot \vn} \\
&= \intfacesext{  \left( \avs{\vc{f}_h} \times \vwh^{in} + \frac12 \vc{f}_h^{in} \times \jump{\vwh} \right) \cdot \vn}   = \intfacesext{ \left( \vc{f}_h^{in} \times \avs{\vwh} + \frac12 \jump{\vc{f}_h} \times \vwh^{in} \right) \cdot \vn }  . 
\end{align*}
Note that we have used in the third equality that
\begin{align*}
& \left(\avs{\vc{f}_h} \times \jump{\vw_h} -\avs{\vc{f}_h} \times \jump{\vw_h} \right) \cdot \vn \\
=& -\avs{w_h^2} \jump{f_h^1} + \avs{w_h^1} \jump{f_h^2} -  \big( -\avs{f_h^2} \jump{w_h^1} + \avs{f_h^1} \jump{w_h^2}  \big) \\
=& \jump{w_h^1 f_h^2} - \jump{f_h^1 w_h^2} = \jump{w_h^1 f_h^2 - f_h^1 w_h^2} = \jump{\vw_h \times \vc{f}_h \cdot \vn}
\end{align*}
with $\vn = (0,0,1), \vc{f}_h = (f^1_h,f^2_h,f^3_h), \vw_h = (w^1_h,w^2_h,w^3_h)$. In all, we complete the proof.
\end{proof}

\begin{Lemma}\label{lem_norm}
Let $\vuh\in Q_h^3$ satisfy the boundary condition $\avs{\vuh}|_{\facesext} = 0$. 
Then there holds 
\begin{align}
\intO{ \mathbb{S}_h : \nabla_h \vuh } = \mu  \norm{\Gradh \vuh}_{L^2(\Omega)}^2 + \left( \eta + \frac{1}{3} \mu \right) \norm{\Divh \vuh}_{L^2(\Omega)}^2.
\label{BB10}
\end{align}

\end{Lemma}

\begin{proof}
Let us consider an extended function $\widehat{\vu}_h$, defined with
\begin{equation}\label{extend-uh}
\widehat{\vu}_h(t, x_h,x_3) = \begin{cases}
- \vuh(t, x_h, -2H - x_3) \ & \mbox{for} \ x_3 < -H,\\ 
\vuh(t, x_h,x_3) \ & \mbox{for}\ x_3 \in [-H,H], \\ 
- \vuh(t, x_h, 2H - x_3) \ & \mbox{for} \ x_3 > H.
\end{cases}
\end{equation}
The function $\widehat{\vu}_h$ is then defined on a (periodic) torus 
$\widehat{\Omega} = \mathbb{T}^{2} \times  [-2H, 2H] \Big|_{-2H, 2H}$, where the desired relation \eqref{BB10} follows easily from the integration-by-parts formula \eqref{coerc}. Seeing that  
\begin{align}
\intO{ \mathbb{S}_h : \nabla_h \vuh } &= \frac{1}{2} 
\int_{\widehat{\Omega}} \mathbb{S}\left(\nabla_h \widehat{\vu}_h\right) : \nabla_h \widehat{\vu}_h \ \dx, \br \norm{\Gradh \vuh}_{L^2(\Omega)}^2  &= 
\frac{1}{2} \norm{\Gradh \widehat{\vu}_h}_{L^2(\widehat{\Omega})}^2,\ 
\norm{\Divh \vuh}_{L^2(\Omega)}^2 = 
\frac{1}{2} \norm{\Divh \widehat{\vu}_h}_{L^2(\widehat{\Omega})}^2, 
\nonumber
\end{align}
we conclude the proof.
%Hence, we have
%\begin{align*}
%& 2 \norm{\Dhuh}_{L^2(\Omega)}^2  = 2 %\norm{\mathbb{D}(\widehat{\vuh})}_{L^2(\Omega)}^2 = 2 %\intO{\mathbb{D}(\widehat{\vuh}) : \mathbb{D}(\widehat{\vuh})} \\
%& =  \intO{\Gradh \widehat{\vuh} : \Gradh \widehat{\vuh}} +   %\intO{\Gradh \widehat{\vuh} : \Gradh^T \widehat{\vuh}}  \\
%\overset{\eqref{InByPa-1}}{=} \norm{\Gradh %\widehat{\vuh}}_{L^2(\Omega)}^2 -  \intO{ \widehat{\vuh} \cdot  \Divh %\Gradh^T \widehat{\vuh}} + \frac12 \intfacesext{  \widehat{\vuh}^{in} %\cdot  \jump{\Gradh^T \widehat{\vuh}}  \cdot \vn} \\
%&\overset{\eqref{op-compat-h-2}}{=} \norm{\Gradh %\widehat{\vuh}}_{L^2(\Omega)}^2 -  \intO{ \widehat{\vuh} \cdot  \Gradh %\Divh  \widehat{\vuh}} + \frac12 \intfacesext{  \widehat{\vuh}^{in} %\cdot  \jump{\Gradh^T \widehat{\vuh}}  \cdot \vn} \\
%&\overset{\eqref{InByPa-1}}{=} \norm{\Gradh \widehat{\vuh}}_{L^2(\Omega)}^2 +  \intO{ \Divh\widehat{\vuh} \cdot   \Divh  \widehat{\vuh}} + \frac12 \intfacesext{  \widehat{\vuh}^{in} \cdot  \jump{\Gradh^T \widehat{\vuh} - \Divh \widehat{\vuh} \bI}  \cdot \vn} \\
%&\overset{\eqref{extend-uh}}{=} \norm{\Gradh \widehat{\vuh}}_{L^2(\Omega)}^2 +  \norm{\Divh \widehat{\vuh}}_{L^2(\Omega)}^2 =  \norm{\Gradh \vuh}_{L^2(\Omega)}^2 + \norm{\Divh \vuh}_{L^2(\Omega)}^2,
%\end{align*}
%and complete the proof.
\end{proof}

\begin{Lemma}\label{lem-lap}
Let $\vuh\in Q_h^d$ and  $ \vvh\in Q_h^d$ satisfy the boundary conditions
\begin{align*}
\avs{\vuh}|_{\facesext} = 0, \; \jump{\vvh}|_{\facesext} = 0\; \jump{\Gradh \vvh}\cdot \vc{n}|_{\facesext}=0. 
\end{align*}
Denote
\begin{align}
E(\vuh, \vvh) = \intO{ \Gradh \vuh : \Gradh \vvh} - \intO{\Curlh \vuh \cdot \Curl \vv_L} - \intO{\Divh \vuh \cdot \Div \vv_L}
\end{align}
for any $\vv_L \in W^{1,2}(\Omega)$.
Then it holds 
\begin{align}
\abs{E(\vuh,\vvh)} \leq \norm{\Gradh \vuh}_{L^2(\Omega)} \norm{\Gradh \vvh - \Grad \vv_L}_{L^2(\Omega)}.
\end{align}
\end{Lemma}

\begin{proof}
Applying Lemmas \ref{lem-op-comp} and \ref{lem-IBP} we have
\begin{align*}
&\intO{ \Gradh \vuh : \Gradh \vvh} = - \intO{  \vuh \cdot \Divh \Gradh \vvh} 
= - \intO{ \vuh \cdot \left(  -\Curlh \Curlh \vvh +  \Gradh \Divh \vvh \right) }\\
&=  \intO{  \Curlh \vuh \cdot   \Curlh \vvh } + \intO{ \Divh \vuh \,   \Divh \vvh  }.
\end{align*}
%Note that we have used Lemma \ref{lem-op-comp} in the second equality.
Consequently, we can reformulate $E(\vuh,\vvh)$ as 
\begin{align*}
E(\vuh,\vvh) = \intO{\Curlh \vuh \cdot \left(\Curlh \vvh - \Curl \vv_L\right)} - \intO{\Divh \vuh  \, \left( \Divh \vvh - \Div \vv_L \right)}
\end{align*}
which leads to
\[
\abs{E(\vuh,\vvh)} \aleq \norm{\Gradh \vuh}_{L^2(\Omega)} \norm{\Gradh \vvh - \Grad \vv_L}_{L^2(\Omega)},
\]
and finishes the proof.
\end{proof}

\section{Consistency and compatibility}

\subsection{Proof of the consistency estimates stated in Lemmas \ref{Lcons1}--\ref{Lconsbe}}\label{ap_css}
\begin{proof}
%{\cgrey The consistency of the continuity equation, the momentum equation as well as the entropy inequality have been proven in \cite[Lemma A.7]{FeLMShYu:2024}. }
 The consistency results of the continuity equation, the momentum equation as well as the entropy inequality are similar as in \cite[Lemma A.7]{FeLMShYu:2024}. The sole difference lies in the temporal regularity of test functions. 
To see this, we revisit the consistency error of the time difference operator defined by
\begin{align*} 
E_t(r_h,\phi)= \intO{ r_h^0 \phi(0,\cdot) }
+ \intTO{ \big( r_h \pd_t \phi + D_t r_h \phi \big)}, \ \mbox{ with } \ r_h = \vrh, \vrh \vuh, \vrh s_h. 
\end{align*}
Thanks to the piecewise linear interpolation, we rewrite $E_t(r_h,\phi)$ as 
\begin{align*}
E_t(r_h,\phi)= \intO{ \widetilde{r_h}^0 \phi(0,\cdot) }
+ \intTO{ \big( r_h \pd_t \phi + \pd_t \widetilde{r_h} \phi \big)} = \intTO{ \big( r_h - \widetilde{r_h} \big) \pd_t \phi },
\end{align*}
%{\cred don't we want to control $E_t(\widetilde{r_h},\phi)$? please check the above two lines}
which can be controlled with $\abs{E_t(r_h,\phi)} \aleq \TS \norm{\pd_t \phi}_{L^2((0,T)\times \Omega)} \norm{\pd_t \widetilde{r_h}}_{L^2((0,T)\times \Omega)} \aleq \TS^{1/2}$ and then yields \eqref{con-e1}, \eqref{cone2} and \eqref{cons5}.

\medskip

%{\cgrey The proof of consistency of the ballistic energy inequality differs slightly due to the introduction of the test function the time dependent test function  $\psi$. }
As for consistency of the ballistic energy inequality, the difference comes from the introduction of 
the time dependent test function  $\psi$. Here we only outline the key steps. 

We first recall the ballistic energy balance in \cite[Lemma A.2]{FeLMShYu:2024}
\begin{align}\label{BE0}
&D_t \intO{ \left(\frac{1}{2} \vrh |\vuh |^2 + c_v \vrh \vth - \vrh s_h \hvth \right) }
+ \intO{ \frac{\hvth}{\vth }\difuh } - \intfacesint{ \frac{\kappa}{h} \avs{\hvth} \jump{\vth } \jump{\frac{1}{\vth } } }
\br
& \hspace{1cm}
+2 \frac{\kappa}{ h }  \intfacesext{ \frac{ (\vt_h^{\rm in} - \vthB)^2 }{\vt_h^{\rm in}}   }
- \intO{\vrh \vc{g} \cdot  \vuh}
+ D_s(\hvth) + D_{\rm E} - R_{B} (\hvth) + R_{s}(\hvth) 
\br
&\hspace{1cm} +\intO{ \vrh s_h ( D_t \hvth + \vuh \cdot \Gradh \hvth) }
- \intfacesint{ \frac{\kappa}{h} \jump{\vth } \jump{\hvth} \avs{ \frac{1}{\vth } } }= 0,
\end{align}
where $D_{\rm E} \geq 0, D_s(\hvth) \geq 0 \mbox{ for } \hvth \geq 0$.
%\begin{equation}\label{ddd}
%\begin{aligned}
%& D_{\rm E} \geq 0; \quad D_s(\hvth) \geq 0 \mbox{ for } \hvth \geq 0; \\ 
%& \intT{\abs{R_{B}(\hvth)}} \aleq h^{(1-\alpha)/2}, \quad 
%\intT{\abs{R_{s}(\hvth)}} \aleq h^{(1+\alpha)/2} +h^{1+\alpha}, 
%\end{aligned}
%\end{equation}
%see \cite[Lemma A.2]{FeLMShYu:2024}. and Steps 3 and 4 of the proof of \cite[Lemma A.7]{FeLMShYu:2024}. 

Adding zero-term $\int_0^T \psi \cdot \eqref{BE0} \dt$ upon consistency formulation \eqref{cons-4}, we obtain 
\begin{align*}
\myangle{{\it error}_{BE}, \psi} &= \myangle{e_{BE}^1, \psi}  + \myangle{e_{BE}^2, \psi}, \\ 
-e_{E,1}^{h}(\hvt)  = & -E_{B,t}(\psi) + E_{B, \vt} +E_{B, \Grad\vt} + E_{B, res} + \intT{\psi (R_B(\hvth) - R_s(\hvth))},
\br
 e_{E,2}^{h}(\hvt)   =&  \intT{ \psi \big(D_s(\hvth)+D_E\big)}  \geq 0,
\end{align*}
with $E_B = \intOB{ \frac{1}{2} \vrh |\vuh |^2 + c_v \vrh \vth - \vrh s_h \hvt }$ and
{\small
\begin{align*}
E_{B,t}(\psi)  =& 
\intT{\psi D_t  E_B(t) } +  \int_0^T \partial_t \psi  \cdot E_B(t)\dt  - \psi(0)E_B(0), 
\\
E_{B, \vt} =& -\psi(0)\intQ{ \left( \vrh s_h (\hvt - \hvth) \right) (0, \cdot)}  + \intT{\psi\intO{ \frac{(\hvt-\hvth)}{\vth }\difuh}} ,
\br
E_{B, \Grad\vt} 
=& -\frac{\kappa}{ h }  \int_0^{T} \psi \intfacesint{  \PiW (\hvth - \hvt ) \frac{\jump{\vth }^2}{\vth^{\rm out} \vth} } \dt
+2 \frac{\kappa}{ h }  \int_0^{T}\psi \intfacesext{ \frac{ (\vt_h^{\rm in} - \vthB)^2 }{\vt_h^{\rm in}}  \left( \frac{\PiW\hvt }{ \vthB} - 1\right) } \dt,
\br
E_{B, res} = & 
\intT{\psi \intO{ \bigg( \vrh s_h ( \partial_t \hvt - D_t \hvth) + \vrh s_h \vuh \cdot (\Grad \hvt- \Gradh \hvth) \bigg)}}
\\
& -\intT{\psi \intO{  \kappa \frac1{\vth} \; \Gradd \vth \cdot (\Grad \hvt- \Gradd \hvth)} }. 
\end{align*}}
Considering a regular $\psi$, say $\psi \in C^2_c[0,T)$, we can directly employ the estimates \cite[Lemma A.7]{FeLMShYu:2024} to conclude
\begin{align*}
& \abs{E_{B, \vt}} + \abs{E_{B, \Grad\vt}} + \abs{ E_{B, res}}  + \abs{\intT{\psi (R_B(\hvth) - R_s(\hvth))}}  \aleq \norm{\psi}_{L^{\infty}(0,T)} \Big( h + h^{(1-\alpha)/2} + h^{(1+\alpha)/2} \Big).
\end{align*}
In addition, 
\begin{align*}
E_{B,t}(\psi) & = \intT{\psi \pd_t  \widetilde{E_B}(t) } +  \int_0^T \partial_t \psi  \cdot E_B(t)\dt  - \psi(0)\widetilde{E_B}(0)  = - \intT{\pd_t  \psi \cdot \left( \widetilde{E_B}-E_B(t) \right)}
\end{align*}
yields $\abs{E_{B,t}(\psi)} \aleq \norm{\pd_t \psi}_{L^{2}(0,T)}\TS^{1/2}$. 
Altogether, we finish the proof of Lemma~\ref{Lconsbe}. 
\end{proof}

\subsection{Consistency of linear in time interpolants}\label{sec-rmk}
In this section, we show the consistency errors of the piecewise linear approximations $\widetilde{\vrh}, \widetilde{\vrh \vuh}, \widetilde{\vrh \vth} $.
We start by studying the properties of projection operators.
\begin{Lemma}\label{lem-proj}
Let the projection operators $\PiQ, \PiFi, \PiWi, i = 1,\dots, d,$ be defined  in \eqref{proj-avg}, \eqref{projE} and \eqref{projW}. 
Then it holds
\begin{align}\label{proj-es-1}
 \norm{\Pi_X \phi - \phi}_{L^2(\Omega)} \aleq h \norm{\Grad \phi} _{L^2(\Omega)}  \ \mbox{with } \  \Pi_X = \PiQ,  \PiWi, \PiWi \PiQ,  \PiFi, \quad \mbox{for } \phi \in W^{1,2}(\Omega).
\end{align}
%with $\Pi_X = \PiQ,  \PiWi, \PiWi \PiQ,  \PiFi$. 
%{\cgrey Furthermore, it holds
% \begin{align}\label{proj-es-2}
%\abs{\PiQ \phi(x)} + \abs{\PiWi \phi(x)}  \aleq h^{1-d/p} \norm{\phi} _{W^{1,p}(\Omega)}, \quad  x \in \{ y | \dist(y, \pd \Omega) \aleq h \}, \quad p > d,
%\end{align}
%for any $\phi \in W^{1,p}(\Omega), \phi|_{\pd \Omega} = 0$.}
\end{Lemma}

\begin{proof}[Proof of Lemma~\ref{lem-proj}] 
The estimates 
$
\norm{\PiQ \phi - \phi}_{L^2(\Omega)} + \norm{\PiWi \phi - \phi}_{L^2(\Omega)} \aleq h \norm{\Grad \phi} _{L^2(\Omega)} 
$
directly follows from a scaling argument and the Poincar\'e inequality. Applying triangle inequality we have 
\[
\norm{\PiWi\PiQ \phi - \phi}_{L^2(\Omega)} \leq 2 \norm{\PiQ \phi - \phi}_{L^2(\Omega)} \aleq h \norm{\Grad \phi} _{L^2(\Omega)}.
\]
and 
\begin{align*}
\norm{\PiFi \phi - \phi}_{L^2(\Omega)} & = \norm{\PiFi \phi -\PiWi \phi + \PiWi \phi- \phi}_{L^2(\Omega)} = \norm{\PiFi \Big(\phi -\PiWi \phi \Big) - \Big( \phi - \PiWi \Phi \Big)}_{L^2(\Omega)} \\
& \leq  \norm{\phi - \PiWi \Phi}_{L^2(\Omega)} \aleq h \norm{\Grad \phi} _{L^2(\Omega)}.
\end{align*}
Note that we have used the projection property of $\PiFi$, i.e.\ $(\PiFi )^2 = \PiFi$, for the first inequality. 
%{\cgrey It remains to show \eqref{proj-es-2}. Applying the embedding theorem $W^{1,p} \hookrightarrow C^{0,1-d/p}$ with $p>d$ and the boundary condition $\phi|_{\pd \Omega} = 0$ we obtain 
%\[
%\sup_{x \notin \pd \Omega} \frac{\phi(x) - 0}{\dist(x, \pd \Omega)^{1-d/p}} \aleq \norm{\phi}_{W^{1,p}}(\Omega),
%\]
%which gives \eqref{proj-es-2} and finishes the proof.}
\end{proof}

We are proceed to show the consistency estimates \eqref{con-e1a}, \eqref{cone2a}, and \eqref{cons4}.
\medskip

Noticing that $\pd_t \tilde{\phi}_h(t) = D_t \phi_h^k$ a.a.\ for $t\in[0,T]$, we have 
\begin{align*}
& \myangle{h_{\vr}; \phi} =  E_F(\vrh, \phi),\quad 
\myangle{h_{\vm}; \Phi} =  E_F(\vrh \vuh, \Phi) + E_{\vm}(\bfphi),\quad
 \myangle{h_{\vt}; \psi} = E_F(\vrh \vth, \psi) + E_{\vt}(\psi)
\end{align*}
with
\begin{align}
E_F(r_h, \phi)  & = \int_{0}^{T} \intfacesint{ {\bf F}_h^{\alpha}  (r_h ,\vuh ) \cdot \jump{\phi_h} } \dt - \intTO{r_h \vuh  \cdot \Grad \phi } ,\\ 
E_{\vm}(\bfphi) & = \intTO{ \big(\bS_h - p_h \I \big): \Big(\Grad \bfphi -  \Gradh \bfphi_h \Big) },
\\ 
E_{\vt}(\psi) &=   \kappa \intTO{  \Gradd \vth \cdot \Grad \psi } - \int_0^T \intfacesint{  \frac{\kappa}{ h } \jump{\vth^{k}}  \jump{ \psi_h}  }\dt 
\br
& - \int_0^T  \intfacesext{ 2 \frac{\kappa}{ h } \left( (\vt_h^{k})^{\rm in} - \vthB \right)  \psi_h^{\rm in}  }.
\end{align}
%{\cgrey The uniform bounds \eqref{HP} and \eqref{ap} directly gives
% \begin{equation}\label{me-bb}
%E_{\vm}(\bfphi),\ E_{\vt}(\psi) \  \mbox{ bounded in } L^{2}((0,T);W^{- 1,2}(\Omega;\R^d)) \mbox{ uniformly for } h \to 0.
%\end{equation}}
The uniform bounds \eqref{HP} directly gives
 \begin{equation}\label{me-bb}
E_{\vm}(\bfphi)\  \mbox{ bounded in } L^{2}(0,T;W^{- 1,2}(\Omega;\R^d)) \mbox{ uniformly for } h \to 0.
\end{equation}
As $\psi \in L^2(0,T; W^{1,2}_0(\Omega))$, we extend $\psi$ over $\widehat{\Omega} = \mathbb{T}^2\times[-2H,2H]$ with  
\begin{equation*}
\psi(t, x_h,x_3) = \begin{cases}
- \psi(t, x_h, -2H - x_3) \ & \mbox{for} \ x_3 < -H,\\ 
\psi (t, x_h, x_3) \ & \mbox{for}\ x_3 \in [-H,H], \\ 
- \psi(t, x_h, 2H - x_3) \ & \mbox{for} \ x_3 > H. 
\end{cases}
\end{equation*}
Hence, we have $\psi \in  L^2(0,T; W^{1,2}(\widehat{\Omega}))$ and  $\avs{\psi_h}|_{\facesext} = 0$, which gives
\begin{align*}
E_{\vt}(\psi) &=   \kappa \intTO{  \Gradd \vth \cdot \Big(\Grad \psi - \Gradd \psi_h \Big)} 
\end{align*}
and
 \begin{equation}\label{me-bb1}
E_{\vt}(\psi)\  \mbox{ bounded in } L^{2}(0,T;W^{- 1,2}(\Omega;\R^d)) \mbox{ uniformly for } h \to 0.
\end{equation}

The rest is to analyze the convective term $E_F(r_h, \phi)$ with $r_h = \vrh, \vrh \vuh, \vrh \vth$.

\medskip

{\bf Convective term.}
Let us reformulate the convection term into four terms $E_F(r_h, \phi) =\sum_{j=1}^4 E_j(r_h, \phi)$
with
\begin{equation*}
\begin{aligned}
& E_1(r_h, \phi) = \frac12 \int_{0}^{T}  \intfacesint{ |\avs{\vuh} \cdot \vc{n}| \jump{r_h } \jump{ \phi_h} } \dt, \quad
E_3(r_h, \phi)
= \muh \int_{0}^{T}  \intfacesint{ \jump{r_h } \jump{ \phi_h} } \dt ,
\\
& E_2(r_h, \phi) = \frac14 \int_{0}^{T}  \intfacesint{ \jump{\vuh} \cdot \vc{n} \jump{r_h } \jump{ \phi_h} } \dt ,
%\\&
\quad E_4(r_h, \phi) =\int_{0}^{T}  \intTd{ r_h \vuh \cdot \Big(\Grad \phi - \Gradh \phi_h \Big) } \dt.
\end{aligned}
\end{equation*}
Hence, applying H\"older inequality and the uniform bounds \eqref{ap}, we have 
\begin{align*}
\abs{E_1(r_h, \phi)+E_2(r_h, \phi)} & \aleq \left(\int_{0}^{T}  \intfacesint{ \abs{\jump{(\vrh, \vuh, \vth)}}^2 } \dt \right)^{1/2}  \left(\int_{0}^{T}  \intfacesint{ \jump{\phi_h }^2 } \dt \right)^{1/2} \\
& \aleq h^{-\alpha/2} h^{1/2} \norm{\Grad \phi}_{L^2((0,T)\times\Omega; \R^{d})} =  h^{(1-\alpha)/2} \norm{\Grad \phi}_{L^2((0,T)\times\Omega; \R^{d})},
\end{align*}
and
\begin{align*}
\abs{E_3(r_h, \phi)} & \aleq h^{\alpha}\left(\int_{0}^{T}  \intfacesint{ \abs{\jump{(\vrh, \vuh, \vth)}}^2 } \dt \right)^{1/2}  \left(\int_{0}^{T}  \intfacesint{ \jump{\phi_h }^2 } \dt \right)^{1/2} \\
& \aleq h^{(1+\alpha)/2} \norm{\Grad \phi}_{L^2((0,T)\times\Omega; \R^{d})}.
\end{align*}
As for $E_4(r_h, \phi)$ we have 
\begin{align*}
E_4(r_h, \phi) %& = \sum_{K\in \mesh} \sum_{\sigma \in \facesK} |\sigma| r_h \vuh \PiF \phi \cdot \vn - \sum_{K\in \mesh} \sum_{\sigma \in \facesK} |\sigma| r_h \vuh \avs{\phi_h} \cdot \vn\\
& = \int_{0}^{T}  \sum_{K\in \mesh} \sum_{\sigma \in \facesK} |\sigma| r_h \vuh \big( \PiF \phi  - \PiW \phi_h\big) \cdot \vn \dt \\
& = \int_{0}^{T} \intfacesint{ \jump{r_h \vuh} \big(   \phi  - \PiW \phi_h \big)  \cdot \vn} \dt + \int_{0}^{T}  \intfacesext{ r_h^{in}  \vuh^{in} \big(  \phi  - \PiW \phi_h \big) \cdot \vn } \dt.
\end{align*}
Then it holds
\begin{align*}
\abs{E_4(r_h, \phi)} & \aleq   \left(\int_{0}^{T}  \intfacesint{ \abs{\jump{(\vrh, \vuh, \vth)}}^2 } \dt \right)^{1/2}  \left(\int_{0}^{T}  \intfacesint{\big( \phi  - \PiW \phi_h\big)^2 } \dt \right)^{1/2} \\
&\quad  +  \left(\int_{0}^{T}  \intfacesext{ \abs{(\vrh, \vuh, \vth))^{in}}^2 } \dt \right)^{1/2}  \left(\int_{0}^{T}  \intfacesext{\big( \phi  - \PiW \phi_h\big)^2 } \dt \right)^{1/2} \\
& \aleq h^{-\alpha/2} h^{1/2} \norm{\Grad \phi}_{L^2((0,T)\times\Omega; \R^{d})} + h^{1/2} \norm{\Grad \phi}_{L^2((0,T)\times\Omega; \R^{d})} \\
&= \big( h^{(1-\alpha)/2} + h^{1/2} \big) \norm{\Grad \phi}_{L^2((0,T)\times\Omega; \R^{d})}
\end{align*}
for any test function $\phi$ together with extension $\jump{\phi_h}|_{\facesext} = 0$, or any test function $\phi, \phi|_{\pd Q} = 0$ together with extension $\avs{\phi_h}|_{\facesext} = 0$.

Altogether, we have
\begin{align}\label{es-conv}
\abs{E_F(r_h, \phi)} \aleq \left[ h^{(1-\alpha)/2}+h^{(1+\alpha)/2} \right] \norm{\Grad \phi}_{L^2((0,T)\times\Omega; \R^{d})}.
\end{align}
Combine the above with \eqref{me-bb} finishes the proofs.

\subsection{Proof of Lemma \ref{lem_CI}}\label{ap_comp}
\begin{proof}
%We start by showing \eqref{cf1}. %estimate $\left< e_{D\vu}; \bbT \right>$ given in \eqref{cf1}.  
Denoting $\bbT_h = \PiQ \bbT$ %for $\bbT(x) \in \R^{d\times d}_{\rm sym}$ 
and $ \bbT_K = \bbT_h|_{K}$, and using the boundary condition 
$\avs{\vuh}|_{\facesext} =0$, we have
{\small
\begin{equation}\label{edu}
\begin{aligned}
&  \intO{ \Big( \vuh \cdot \Div \bbT+ \Dhuh : \bbT \Big) } 
=   \sum_K \vu_K \cdot \int_{\partial K} \bbT \cdot \vn \ds+  \sum_K  \bbT_K : \sum_{\sigma \in \facesK } |\sigma|\avs{\vuh}  \otimes \vn
\\
& =  - \intfacesint{  \jump{\vuh} \cdot  \bbT \cdot \vn } + \intfacesext{ \vuh^{\rm in} \cdot  \bbT \cdot \vn }
+ \frac12 \sum_K  \bbT_K : \sum_{\sigma \in \facesK } |\sigma|\jump{\vuh}  \otimes \vn
\\
& =  - \intfacesint{  \jump{\vuh} \cdot  \bbT \cdot \vn } + \intfacesext{ \vuh^{\rm in} \cdot  \bbT \cdot \vn }
 +  \intfacesint{ \jump{\vuh}\cdot \avs{\bbT_h} \cdot \vn } + \frac12 \intfacesext{  \jump{\vuh} \cdot \bbT_h^{\rm in} \cdot \vn }
% \\& =  - \intfacesint{ \jump{\vuh} \cdot \big(  \bbT - \avs{\bbT_h} \big)\cdot \vn } - \frac12 \intfacesext{ \jump{\vuh} \cdot ( \bbT - \bbT_h^{\rm in})\cdot \vn }
\\& =  - \intfacesint{ \jump{\vuh} \cdot \big(  \bbT - \avs{\bbT_h} \big)\cdot \vn } + \intfacesext{ \vuh^{\rm in} \cdot ( \bbT - \bbT_h^{\rm in})\cdot \vn }.
\end{aligned}
\end{equation}
}
Applying the H\"older inequality, the uniform bounds \eqref{ap3}, \eqref{ap5} and the projection estimates \eqref{proj-es-1}, we obtain
\begin{align*}
&%\abs{ \left< e_{D\vu}; \bbT \right> }   = 
 \abs{\int_0^T \left(\intfacesint{ \jump{\vuh} \cdot \Big(  \bbT - \avs{\bbT_h} \Big)\cdot \vn } + \intfacesext{ \vuh^{\rm in} \cdot ( \bbT - \bbT_h^{\rm in})\cdot \vn }\right)\dt} 
\\&
\aleq \left(\left(\int_0^T \intfacesint{ \jump{\vuh}^2 }\dt \right)^{1/2} + \left(\int_0^T \intfacesext{ |\vuh^{\rm in}|^2} \dt \right)^{1/2} \right)
h^{1/2} \norm{\bbT}_{L^2(0,T;W^{1,2}(\Omega;\R^{d\times d}))}
\\&
 \aleq  h^{(1-\alpha)/2}, 
\end{align*}
which implies \eqref{cf1}. 
%where $C$ depends on  $ \norm{\bbT}_{W^{1,2}(\Omega;\R^{d\times d})}$ and the data stated in Lemma \ref{lm_ub}.

It remains to show \eqref{cf3}. 
%Secondly, we show \eqref{cf3}. %estimate $\left< e_{\Grad \vt}; \Psi \right>$ given in \eqref{cf3}. 
With the definitions of projection operators \eqref{projW} and the boundary condition $\avs{\vth}_\sigma= \vt_{B,h}|_\sigma $, $\sigma \in \mathcal{E}_{\rm ext}$, we have 
%\begin{align*}
%& \intOB{   \vth \Div \Psi +  \Gradd\vth  \cdot \Psi } 
%= 
%  \sum_K  \vth \int_{\partial K}  \Psi \cdot \vn \ds +  \intfacesint{ \PiW \Psi \cdot \vc{n} \jump{\vth}} 
%  + \frac12\intfacesext{ \PiW \Psi \cdot \vc{n} \jump{ \vth}}
%\\ 
%&=   - \intfacesint{   \Psi \cdot \vc{n} \jump{\vth}} 
%+ \intfacesint{ \PiW \Psi \cdot \vc{n} \jump{\vth}} + \intfacesext{ \PiW \Psi \cdot \vc{n} ( \vthB - \vth^{in})}
%\\
%& = \intfacesint{ (  \PiW \Psi - \Psi)\cdot \vc{n} \jump{\vth}} - \intfacesext{ (  \Psi - \PiW \Psi) \cdot \vc{n} ( \vthB - \vth^{in}) }.
%\end{align*}
\begin{align*}
& \intOB{   
(\vth- \vtB) \Div \Psi +  (\Gradd\vth- \Grad\vtB)  \cdot \Psi } 
\\&= 
  \sum_K  \vth \int_{\partial K}  \Psi \cdot \vn \ds +  \intfacesint{ \PiW \Psi \cdot \vc{n} \jump{\vth}} 
  + \frac12\intfacesext{ \PiW \Psi \cdot \vc{n} \jump{ \vth}}  -\intfacesext{\vtB \Psi \cdot \vn}
\\ 
&=   - \intfacesint{   \Psi \cdot \vc{n} \jump{\vth}}  + \intfacesext{ (\vth^{\rm in}-\vtB) \Psi \cdot \vc{n} } 
\\ & \quad + \intfacesint{ \PiW \Psi \cdot \vc{n} \jump{\vth}} + \intfacesext{ \PiW \Psi \cdot \vc{n} ( \vthB - \vth^{in})}
\\
& = \intfacesint{ (  \PiW \Psi - \Psi)\cdot \vc{n} \jump{\vth}} +\intfacesext{ (  \PiW \Psi - \Psi  ) \cdot \vc{n} ( \vtB - \vth^{in}) }.
\end{align*}

Consequently, owing to the estimate \eqref{ap1}, \eqref{ap5} and \eqref{proj-es-1}, we obtain 
\begin{align*}
&%\abs{\left< e_{\Grad \vt}; \Psi \right>} =
\abs{ \int_0^T \left(\intfacesint{ (  \PiW \Psi - \Psi)\cdot \vc{n} \jump{\vth}} - \intfacesext{ ( \PiW  \Psi - \Psi) \cdot \vc{n} ( \vtB - \vth^{in}) }  \right)\dt}
\\& \aleq
 h^{1/2}\norm{\Psi}_{L^2(0,T;W^{1,2}(\Omega;\R^d))} \left( h^{1/2}\norm{\Gradd \vth}_{L^2((\timezoneA) \times \Omega)} + \left(\int_0^T \intfacesext{( \vtB - \vth^{in})^2} \dt\right)^{1/2} \right)
\aleq h,
\end{align*}	
which implies \eqref{cf3}. 
\end{proof}

\end{document}